\documentclass[a4paper,11pt]{amsart}
\usepackage{amsthm}
\usepackage{tikz}
\usepackage{epsfig}
\usepackage{amsfonts}
\usepackage{amsmath}
\usepackage{amssymb}
\usepackage{hyperref,latexsym}
\PassOptionsToPackage{unicode}{hyperref}

\usepackage[english]{babel}
\usepackage{amsmath}
\usepackage{amsfonts}
\usepackage{enumerate}
\usepackage{amssymb,mathtools,caption}
\usepackage{color}
\usepackage{mathrsfs}
\usepackage{xkeyval,xifthen,bm,twoopt}

\newtheorem{theo}{Theorem}[section]
\newtheorem{lemm}[theo]{Lemma}
\newtheorem{coro}[theo]{Corollary}
\newtheorem{prop}[theo]{Proposition}
\newtheorem{rema}[theo]{Remark}

\theoremstyle{definition}
\newtheorem{hyp}{Hypothesis}
\newtheorem*{exampl*}{Examples}
\newtheorem*{conj*}{Conjecture}

\newtheorem*{hist*}{History for $p=2$}
\newtheorem{clm}{Claim}
\newtheorem{asm}[theo]{Assumption}

\DeclareMathOperator{\dv}{div}

\newcommand{\lbr}[1][{(}]{\left#1}
\newcommand{\rbr}[1][{)}]{\right#1}

\newcommand{\bo}{\mathcal{O}}
\newcommand{\bigo}[1]{\bo \lbr #1 \rbr}

 \everymath{\displaystyle}

\numberwithin{equation}{section}
\newcommand{\R}{\mathbb{R}}

\newcommand{\eqdef}{\stackrel{{\rm{def}}}{=}}

\newcommand{\ga} {\gamma}

\usepackage[a4paper]{geometry}
\newcommand{\de}{\delta}
\newcommand{\al}{\alpha}

\newcommand{\gpgt}[1][]{\ifthenelse{\isempty{#1}}{g \left( \frac{\ga}{2} \right)}{\frac{\ga}{#1}}}

\newcommand{\taylor}[4][]{
\ifthenelse{\equal{#1}{2}}{#2 \lbr #3 \rbr + #2'\lbr#3 \rbr \lbr #4\rbr + \frac{#2'' \lbr\zeta\rbr \lbr #4 \rbr^2}{2} }{}
\ifthenelse{\equal{#1}{3}}{#2 \lbr #3 \rbr + #2'\lbr#3 \rbr \lbr #4 \rbr + \frac{#2'' \lbr#3\rbr \lbr #4 \rbr^2}{2} + \frac{#2''' \lbr\zeta\rbr \lbr #4 \rbr^3}{3!} }{}
}

\newcommand{\taylorinv}[4][]{
\ifthenelse{\equal{#1}{2}}{#2\ \lbr #3 \rbr + \lbr #2\rbr '\lbr#3 \rbr\lbr #4\rbr + \frac{\lbr #2\rbr'' \lbr\zeta\rbr \lbr #4 \rbr^2}{2} }{}
\ifthenelse{\equal{#1}{3}}{#2 \lbr #3 \rbr + \lbr #2\rbr'\lbr#3 \rbr \lbr #4\rbr + \frac{\lbr #2\rbr'' \lbr#3\rbr \lbr #4 \rbr^2}{2} + \frac{\lbr #2\rbr''' \lbr\zeta\rbr \lbr #4 \rbr^3}{3!} }{}
}

\newcommand{\gpg}[2][]{
    \ifthenelse{\isempty{#1}}
    {\ifthenelse{\isempty{#2}}{g(\ga)}{g\lbr#2\rbr} }
    {\ifthenelse{\isempty{#2}}{g^{(#1)}(\ga)}{g^{(#1)}\lbr#2\rbr} }
    }

\newcommand{\gpu}[1][]{
		\ifthenelse{\isempty{#1}}%
		      {g(u)}
		      {g^{(#1)}(u)}
		      }

\newcommand{\gpgd}[2][]{
\ifthenelse{\isempty{#2}}{
		\ifthenelse{\isempty{#1}}{g(\ga)}{}
		\ifthenelse{\equal{#1}{1}}{g'(\ga)}{}
		\ifthenelse{\equal{#1}{2}}{g''(\ga)}{}
		\ifthenelse{\equal{#1}{3}}{g'''(\ga)}{}
		\ifthenelse{\equal{#1}{4}}{g''''(\ga)}{}}
		{
		\ifthenelse{\isempty{#1}}{g\lbr#2\rbr}{}
		\ifthenelse{\equal{#1}{1}}{g'\lbr#2\rbr}{}
		\ifthenelse{\equal{#1}{2}}{g''\lbr#2\rbr}{}
		\ifthenelse{\equal{#1}{3}}{g'''\lbr#2\rbr}{}
		\ifthenelse{\equal{#1}{4}}{g''''\lbr#2\rbr}{}}
		      }

\newcommand{\gpud}[1][]{
		\ifthenelse{\isempty{#1}}{g(u)}{}
		\ifthenelse{\equal{#1}{1}}{g'(u)}{}
		\ifthenelse{\equal{#1}{2}}{g''(u)}{}
		\ifthenelse{\equal{#1}{3}}{g'''(u)}{}
		}		      
		      
\newcommand{\vsig}{\varsigma}

\newcommand{\vu}[1][]{
		\ifthenelse{\isempty{#1}}%
		      {\vsig(u)}
		      {\vsig^{(#1)}(u)}
		      }

\newcommand{\vg}[1][]{
		\ifthenelse{\isempty{#1}}%
		      {\vsig(\ga)}
		      {\vsig^{(#1)}(\ga)}
		      }
		      
\newcommand{\gk}[1][]{
		\ifthenelse{\isempty{#1}}
		{\ga^q}
		{\ga^{q-#1}}
		}

\newcommand{\ginv}[2][]{
\ifthenelse{\isempty{#2}}{
		\ifthenelse{\isempty{#1}}{g^{-1} \lbr \ga \rbr}{}
		\ifthenelse{\equal{#1}{1}}{\lbr g^{-1}\rbr' \lbr \ga \rbr}{}
		\ifthenelse{\equal{#1}{2}}{\lbr g^{-1}\rbr'' \lbr \ga \rbr}{}
		\ifthenelse{\equal{#1}{3}}{\lbr g^{-1}\rbr''' \lbr \ga \rbr}{}
		}{
		\ifthenelse{\isempty{#1}}{g^{-1} \lbr #2 \rbr}{}
		\ifthenelse{\equal{#1}{1}}{\lbr g^{-1}\rbr' \lbr #2 \rbr}{}
		\ifthenelse{\equal{#1}{2}}{\lbr g^{-1}\rbr'' \lbr #2 \rbr}{}
		\ifthenelse{\equal{#1}{3}}{\lbr g^{-1}\rbr''' \lbr #2 \rbr}{}
		    }
		    }

\newcommand{\td}{T_{\de}}

\newcommand{\gtd}[1][]{
  \ifthenelse{\isempty{#1}}{y(\td)}{}
  \ifthenelse{\equal{#1}{1}}{g'(y(\td))}{}
  \ifthenelse{\equal{#1}{2}}{g''(y(\td))}{}
  \ifthenelse{\equal{#1}{3}}{g'''(y(\td))}{}
}

\newcommand{\ytd}[1][]{
  \ifthenelse{\isempty{#1}}{y(\td)}{}
  \ifthenelse{\equal{#1}{1}}{y'(\td)}{}
  \ifthenelse{\equal{#1}{2}}{y''(\td)}{}
  \ifthenelse{\equal{#1}{3}}{y'''(\td)}{}
}

\newcommand{\plap}[2][]{\ifthenelse{\isempty{#1}}{\lvert #2 \rvert^{n-2} #2}{\lbr \lvert #2 \rvert^{n-2} #2 \rbr}}

\newcommand{\gt}[2][]{\ifthenelse{\isempty{#2}}{
		\ifthenelse{\isempty{#1}}{g(y(t))}{}
		\ifthenelse{\equal{#1}{1}}{g'(y(t))}{}
		\ifthenelse{\equal{#1}{2}}{g''(y(t))}{}
		\ifthenelse{\equal{#1}{3}}{g'''(y(t))}{}
		}{
		\ifthenelse{\isempty{#1}}{g(y(#2))}{}
		\ifthenelse{\equal{#1}{1}}{g'(y(#2))}{}
		\ifthenelse{\equal{#1}{2}}{g''(y(#2))}{}
		\ifthenelse{\equal{#1}{3}}{g'''(y(#2))}{}
		    }
		    }

\newcommand{\psit}[2][]{\ifthenelse{\isempty{#2}}{
		\ifthenelse{\isempty{#1}}{\psi(y(t))}{}
		\ifthenelse{\equal{#1}{1}}{\psi'(y(t))}{}
		\ifthenelse{\equal{#1}{2}}{\psi''(y(t))}{}
		\ifthenelse{\equal{#1}{3}}{\psi'''(y(t))}{}
		}{
		\ifthenelse{\isempty{#1}}{\psi(y(#2))}{}
		\ifthenelse{\equal{#1}{1}}{\psi'(y(#2))}{}
		\ifthenelse{\equal{#1}{2}}{\psi''(y(#2))}{}
		\ifthenelse{\equal{#1}{3}}{\psi'''(y(#2))}{}
		    }
		    }

\newcommand{\nminn}[1][]{ \ifthenelse{\isempty{#1}}{ \lbr \frac{n-1}{n} \rbr}{\lbr \frac{n-1}{n} \frac{1}{#1}\rbr}} 

\newcommand{\funcpow}[4][y]{
\ifthenelse{\isempty{#2}}{\lbr #1 (#3) \rbr^{#4}}{}
\ifthenelse{\equal{#2}{1}}{\lbr #1' (#3) \rbr^{#4}}{}
\ifthenelse{\equal{#2}{2}}{\lbr #1'' (#3) \rbr^{#4}}{}
\ifthenelse{\equal{#2}{3}}{\lbr #1''' (#3) \rbr^{#4}}{}
\ifthenelse{\equal{#2}{4}}{\lbr #1'''' (#3) \rbr^{#4}}{}
\ifthenelse{\equal{#2}{5}}{\lbr #1''''' (#3) \rbr^{#4}}{}
\ifthenelse{\equal{#2}{6}}{\lbr #1'''''' (#3) \rbr^{#4}}{}
}

\usepackage{graphicx}
\usepackage{graphics}
\usepackage{epsfig}
\usepackage{amsmath}
\usepackage{amsfonts}
\usepackage{psfrag}

\begin{document}


\title[Uniqueness of positive radial solutions with exponential nonlinearity]{Uniqueness of positive solutions of a $n$-Laplace equation in a ball in $\R^n$ with exponential nonlinearity}
\author{Adimurthi}
\address{T.I.F.R. CAM, P.B. No. 6503,\\
   Sharadanagar, Chikkabommasandra \\
   Bangalore 560065, India}
\email{adiadimurthi@gmail.com \and aditi@math.tifrbng.res.in}

\author{Karthik A}
\address{Department of Mathematics,\\
Louisiana State University,\\
USA}
\email{kadimu1@lsu.edu \and karthikaditi@gmail.com}

\author{Jacques  Giacomoni}
\address{LMAP (UMR CNRS 5142) Bat. IPRA,
   Avenue de l'Universit\'e \\
   F-64013 Pau, France}
\email{jacques.giacomoni@univ-pau.fr}

%
%
%
%
%
%

%
%


\begin{abstract}
Let $n\geq 2$ and $\Omega\subset\R^n$ be a bounded domain. Then by Trudinger-Moser embedding, $W^{1,n}_0(\Omega)$ is embedded in an Orlicz space consisting of exponential functions. Consider the corresponding semi linear $n$-Laplace equation with critical or sub-critical exponential nonlinearity in a ball $B(R)$ with dirichlet boundary condition. In this paper, we prove that under suitable growth conditions on the nonlinearity, there exists an $\gamma_0>0$, and a corresponding $R_0(\gamma_0)>0$ such that for all $0<R<R_0$, the problem admits a unique non degenerate positive radial solution $u$ with $\Vert u\Vert_\infty\geq \gamma_0$.
\end{abstract}
\maketitle 
\section{Introduction}
Let $B(R)\subset\R^n$ be the ball of radius $R$ with center at the origin and $1<p\leq n$ with $f\in C^0[0,+\infty)$. Consider the following problem
\begin{eqnarray}\label{eq1.1}
\left\{\begin{array}{ll}
&-\dv (\vert\nabla u\vert^{p-2}\nabla u)=f(u)\quad\mbox{in }B(R)\\
&u=0\quad\mbox{on }\partial B(1), \; u>0\quad\mbox{in } B(R).
\end{array}\right.
\end{eqnarray}
Existence of a solution to Problem \eqref{eq1.1} has been studied extensively under the banner of Emden-Fowler type equations or Yamabe type equations. In this paper we make some progress regarding the question of uniqueness.

\begin{hist*} Let  $\lambda>0$ and define
\begin{eqnarray}\label{eq1.2}
\displaystyle f(u)=\left\{\begin{array}{ll}
u^r+\lambda u\quad&\mbox{if } n\geq 2 \text{ and } 1< r < \infty \\
 \lambda \ h(u) e^{u^q}\quad&\mbox{if } n=2 \text{ and } 1\leq q\leq 2,
\end{array}\right.
\end{eqnarray}
where $h\in C^1(0,\infty)$ with $h(0) = 0$ and is of lower order growth than $e^{u^q}$. In this context, H. Brezis raised the following question:

\begin{center} \emph{``Does \eqref{eq1.1} admit utmost one solution?''}\end{center}

From {\sc Gidas-Ni-Nirenberg} \cite{GiNiNi}, we know that any solution of \eqref{eq1.1} with $p=2$  is radial. Using this and by scaling argument when $\lambda=0$, in \cite{GiNiNi}, it was proven that \eqref{eq1.1} with $p=2$ admits utmost one solution. When $\lambda\neq 0$, their approach fails to yield uniqueness. Using the Pohozaev's identity for $n \geq 2$ and $r<\infty$, the uniqueness question was answered in  \cite{AdYa,ErTa,Li,Sr,NiNu}. In  \cite{ErTa,Ta}, the authors used the Pohozaev's identity in an ingenious way to obtain uniqueness for a wide class of nonlinearities $f(u)$. 

Next, consider the case when $n=2$ and the nonlinearity is of exponential type. If $f(u)=u e^u$, then uniqueness was proved in  \cite{Ad1,Ta}. In \cite{Ta}, they again made use of Pohozaev's identity, though unfortunately, their technique  cannot be extended if $q>1$. Hence the question remained:

\begin{center}\emph{``What happens if $n=2$ and $1<q\leq 2$ ?''} \end{center}
\end{hist*} 
In this paper we try to answer the uniqueness by adopting the linearization technique from \cite{Ad1} and the asymptotic analysis of {\sc Atkinson-Peletier} \cite{AtPe,LePe} (see also {\sc Volkmer} \cite{Vo}). Under suitable conditions on $f$ we show that there exists a $\lambda_0>0$ such that  for $p=n$, $0<\lambda<\lambda_0$, $q\in (1,\frac{n}{n-1}]$ and $\|u\|_{\infty} > \text{'some large quantity'}$, the radial solutions of  \eqref{eq1.1} are unique and nondegenerate.
For the existence and nonexistence of positive solutions  see \cite{Ad,AdPr2}.

\begin{hyp}
\label{hypothesis-f}
Let $f\in C [0,\infty)\cap C^3(0,\infty) $ with $f(u)>0$ for $u>0$. Let $a>0$, $\lambda>0$, $\rho\in C^3(0,\infty)$ and $1<q\leq \frac{n}{n-1}$ be  such that 
\begin{eqnarray}\label{eq1.3}
f(u)=\lambda e^{g(u)}, \quad g(u)=a u^q +\rho(u).
\end{eqnarray}
Let $\rho^{(k)}$ denote the $k$-th derivative of $\rho$. 
Assume that there exist $0<\alpha\leq 1$, $\beta>0$ such that
\begin{eqnarray}\label{eq1.4}
\begin{array}{ll}
f(s)-f(0)=\bigo{s^\beta}\text{  as  } s \rightarrow 0^+\\
f'(s)=\bigo{s^{-1+\alpha}}\text{  as  } s \rightarrow 0^+.
\end{array}
\end{eqnarray}
Now the following assumptions will be stated as and when they are used in the theorems: 
\begin{itemize}
\item[(H1)] $\displaystyle\lim_{\gamma\to\infty}\frac{\rho^{(k)}(\gamma)}{\gamma^{q-k}}=0$ for $k\in \{0,1,2,3\}$.
\item[(H2)] There exists a $b \in \R$ such that 

$$
\displaystyle\varliminf_{\gamma\to\infty}g(\gamma)-\left(\frac{n-1}{n}\right) \gamma g'(\gamma) \geq b
$$
\item[(H3)] 
$
\displaystyle\lim_{\gamma\to\infty}\frac{g'(\gamma)}{g''(\gamma) \lbr[[] \log g'(\gamma)\rbr[]]^4} \lbr g'(\gamma)-(n-1)\gamma g''(\gamma)\rbr=\infty.
$
\end{itemize}
Here $g'$, $g''$ denote the first and second  derivatives of $g$. 
\end{hyp}


We now have  the following main result:
\begin{theo}\label{theo1.5}
Assume $f$ satisfies (H1) and (H3) from Hypothesis \ref{hypothesis-f} and that there exists a  $p \geq 0$ and a constant  $C_p\neq 0$ such that
\begin{eqnarray}\label{eq1.6}
\displaystyle\lim_{u\to 0^+}\frac{f(u)}{u^p}=C_p.
\end{eqnarray}
Then there exists a  $\lambda_0>0$, $\gamma_0>0$ such that for all $\lambda\in (0,\lambda_0)$
\begin{eqnarray}\label{eq1.7}
\left\{\begin{array}{ll}
&-\dv(\vert\nabla u\vert^{n-2}\nabla u)=\lambda f(u)\quad\mbox{in }B(1),\\
&u>0,\quad u\mbox{ radial in } B(1),\\
& u=0\quad\mbox{on }\partial B(1),
\end{array}\right.
\end{eqnarray}
admits a unique solution satisfying $\Vert u\Vert_\infty>\gamma_0$. Furthermore, the solution is non degenerate.
\end{theo}
\begin{rema}
 By integration, (H1) and (H3) of Hypothesis \ref{hypothesis-f} implies (H2). In fact, one can find $\ga_0 > 0 $, $M>0$ and a constant $c \in \R$ such that for all $\ga \geq \ga_0$, we have the estimate
 \begin{equation*}
  g(\ga) - \lbr \frac{n-1}{n} \rbr \ga g'(\ga) \geq M (\log \ga)^5 + c.
 \end{equation*}
It is easy to see in the subcritical case, that is when $1<q<\frac{n}{n-1}$, (H3) of Hypothesis \ref{hypothesis-f} always holds. The question now remains if one can remove the hypothesis (H3)!
\end{rema} 
First consequence of Theorem \ref{theo1.5} is that large solution coincides with the solution obtained by the moutainpass lemma and we have the following Corollary:
\begin{coro}
Assume that the hypothesis in Theorem \ref{theo1.5} are satisfied. Then there exists $\lambda_0>0$ and $\gamma_0>0$ such that for all $\lambda\in (0,\lambda_0)$,  solutions to \eqref{eq1.7}  satisfying $\Vert u\Vert_\infty>\gamma_0$ has a Morse index equal to one.
\end{coro}

\begin{exampl*} We give some examples of $f$ that satisfy conditions in Hypothesis \ref{hypothesis-f}. 
\begin{itemize}
\item[(i)] Let $f(u)= u^p e^{u^q}$, $0\leq p<\infty$, then we have
\begin{gather*}
g(\gamma)=\gamma^q+p\log(\gamma), \quad \rho(\gamma)=p\log(\gamma),\\
g(\gamma)-\left(\frac{n-1}{n}\right)\gamma g'(\gamma)=\left(\frac{n-1}{n}\right)\gamma^q\left(\frac{n}{n-1}-q\right)+p\log(\gamma)-\frac{p(n-1)}{n}\\
g'(\gamma)-(n-1)\gamma g''(\gamma)=(n-1)q\gamma^{q-1}\left(\frac{n}{n-1}-q\right)+\frac{pn}{\gamma}.
\end{gather*}
 Therefore, if $1<q\leq\frac{n}{n-1}$, we see that  (H1) and (H2) of Hypothesis \ref{hypothesis-f} is satisfied. If $1<q<\frac{n}{n-1}$, then (H1), (H2) and  (H3) are satisfied, but when $q=\frac{n}{n-1}$, (H3) is not satisfied.
\item[(ii)] Let $f(u)=\lambda u^p e^{u^q+\beta u}$ with $\beta>0$. In this case, (H1), (H2) and (H3) are satisfied for $1\leq q\leq \frac{n}{n-1}$.
\end{itemize}
\end{exampl*}

From the Picone's Identity \cite{AlHu}, assuming \eqref{eq1.6}, there exist a $\lambda_1>0$, $\gamma_1>0$ such that (see Lemma \ref{lemma2.1}):
\begin{itemize}
\item If $0\leq p< n-1$, then for $\lambda\in (0,\lambda_1)$,  there exists a solution $u$ of \eqref{eq1.7} with $\Vert u\Vert_\infty<\gamma_1$.
\item If $p\geq n-1$, then for $\lambda\in (0,\lambda_1)$, there does not exist a solution with $\Vert u\Vert_\infty<\gamma_1$.
\end{itemize}
Hence by plotting the solution curve $\|u_\lambda\|_{\infty}$ with respect to $\lambda$, we have the following diagrams.
\begin{figure*}[ht!]
\flushleft
\begin{minipage}{0.40\textwidth}
\begin{tikzpicture}[scale=0.7]
\draw[->] [thick] (0,0) node[below]{$0$} -- (9.5,0) node[below]{$\|u_\lambda\|_{\infty}$};
\draw[->] [thick] (0,0)--(0,5) node[left]{$\lambda$};
\draw [thick, color=black] plot [smooth] coordinates {(0,0)  (0.7,0.6) (1.3,1.3) (2.3,3) (2.9,3) (3.5,1.5) (4.3,1.5) (5.5,4) (6.3,4) (7.2,2) (8,1.1)  (9.6, 0.7)    };
\draw[dashed][thick]  (2.6,0)node[below]{$\gamma_1$}--(2.6,3.1);
\draw[dashed][thick]  (7,0)node[below]{$\gamma_0$}--(7,2.3);
\end{tikzpicture}
\caption*{Case $0\leq p<n-1$.}
\end{minipage}
\hspace*{1.6cm}
\begin{minipage}{0.40\textwidth}
 \begin{tikzpicture}[scale=0.7]
\draw[->] [thick] (0,0) node[below]{$0$} -- (9.5,0) node[below]{$\|u_\lambda\|_{\infty}$};
\draw[->] [thick] (0,0)--(0,5) node[left]{$\lambda$};
\draw [thick, color=black] plot [smooth] coordinates {(0,4.5) (1,3.9)(3,1.5) (4.3,1.5) (5.5,4) (6.3,4) (7.2,2) (8,1.1)  (9.6, 0.7)    };
\draw[dashed][thick]  (7,0)node[below]{$\gamma_0$}--(7,2.3);
\end{tikzpicture}
{\caption*{Case $p=n-1$.}}
\end{minipage}
\end{figure*}

\begin{figure}[ht!]
 \centering
\begin{tikzpicture}[scale=0.7]
\draw[->] [thick] (0,0) node[below left]{$0$} -- (10,0) node[below]{$\|u_\lambda\|_\infty$};
\draw[->] [thick] (0,0)--(0,5) node[left]{$\lambda$};
\draw [thick, color=black] plot [smooth] coordinates {(0.3,4.5)  (0.8,1.5) (1.5, 1.5) (2.3,3) (2.9,3) (3.5,1.5) (4.3,1.5) (5.5,4) (6.3,4) (7.2,2) (8,1.1)  (9.6, 0.7)    };
\draw[dashed][thick]  (7,0)node[below]{$\gamma_0$}--(7,2.3);
\end{tikzpicture}
\caption*{Case $p>n-1$.}
 \end{figure}

\begin{rema}
We now recall the Trudinger-Moser imbedding \cite{Mo}: let $\Omega\in \R^n$ be a bounded  open set, then  the Trudinger-Moser imbedding says
\begin{eqnarray*}
A=\displaystyle\sup_{\substack{u\in W^{1,n}_0(\Omega)\\ \| u\|\leq 1}}\int_\Omega e^{(n \omega_n^{\frac{1}{n-1}})\vert u\vert^{\frac{n}{n-1}}}<\infty, 
\end{eqnarray*}
where $\omega_n=$ volume of $S^{n-1}$ and $n \omega_n^{\frac{1}{n-1}}$ is the best constant. It was shown by {\sc Carleson-Chang} \cite{CaCh} that if $\Omega$ is a ball, then $A$ is achieved by a radial function $u_0\in W^{1,n}_0(\Omega)$ with $u_0>0$ satisfying the  equation
\begin{eqnarray}\label{eq1.8}
\left\{\begin{array}{ll}
&-\dv(\vert\nabla u_0\vert^{n-2}\nabla u_0)=\lambda_0u_0^{\frac{1}{n-1}} e^{(n\omega^{\frac{1}{n-1}})u_0^{\frac{n}{n-1}}}\quad\mbox{in }\Omega,\\
& u_0>0, \quad u_0 \textrm{ radial  in }\Omega,\\
& u_0=0\quad\mbox{on }\partial\Omega,
\end{array}\right.
\end{eqnarray}
for some $\lambda_0>0$. 
\begin{conj*} Maximizers  are unique!\end{conj*}

In case of $n=2$, if one can show that \eqref{eq1.7} admits utmost one solution for all $\lambda>0$, then the above conjecture is positively answered. 
\end{rema}
\begin{rema}
In \cite{AdSa}, the following singular Trudinger Moser Imbedding has been shown: 
\begin{eqnarray*}
\displaystyle\sup_{\Vert u\Vert_{W^{1,p}_0(\Omega)}\leq 1}\int_{\Omega}\frac{e^{\alpha\vert u\vert^{\frac{n}{n-1}}}}{\vert x\vert^\beta}<\infty, 
\end{eqnarray*}
holds if and only if $\alpha$ and $\beta $ satisfy $\frac{\alpha}{\delta_n}+\frac{\beta}{n}\leq 1$  where $\delta_n:=n \omega_n^{\frac{1}{n-1}}$.
\end{rema}
 In view of this, consider the following singular semilinear equation
\begin{eqnarray}\label{eq1.8singular}
\left\{\begin{array}{ll}
&-\dv(\vert\nabla u\vert^{n-2}\nabla u)=\lambda \frac{f(u)}{\vert x\vert^\beta}\quad\mbox{in }B(1),\\
& u>0, \quad u\mbox{ radial  in }B(1),\\
& u=0\quad\mbox{on }\partial B(1).
\end{array}\right.
\end{eqnarray}
We now have the following theorem concerning the singular Trudinger-Moser embedding:
\begin{theo}\label{theosingular}
Assume that $\beta<n$ and $f$ satisfies all the conditions in Hypothesis \ref{hypothesis-f}, then there exists a $\lambda_0>0$ and  $\gamma_0>0$ such that for all $\lambda\in (0,\lambda_0)$,  \eqref{eq1.8singular} admits a unique solution $u$ provided $\Vert u\Vert_\infty\geq \gamma_0$.
\end{theo}
We now give the plan of the paper: The paper is divided into four parts; In Section \ref{section2}, using the Sturm's transformation, we reduce the problem to an initial value problem starting at $\infty$. By shooting argument, we rephrase Theorem \ref{theo1.5} into the behaviour of first zero of the solution with respect to the initial condition as in \cite{AtPe}. We state two theorems without proof which deals with the asymptotic behaviour of the first zero and its derivative with respect to the initial data. Then we deduce the proof of Theorems \ref{theo1.5} and \ref{theosingular}.

In Section \ref{section3}, we prove the asymptotic behaviour  of first zero of the solution. Basically, this is in the {\sc Atkinson-Peletier} \cite{AtPe} analysis for general $\gamma$. Here we make use of the ideas from \cite{Vo}  to obtain finer estimates. Since some of the estimates are used in the proof of the theorem and we were unable to find exact references, we will give the complete proof.

In Section \ref{section4}, we prove the asymptotic behaviour of the derivative of the first zero with respect to the initial data using a new identity. This is much more delicate.
The proof of the Theorem \ref{theo1.5} follows from this finer asymptotic behaviour which will be given at the end. 
\section{Proof of Theorem \ref{theo1.5}.}\label{section2}
First we  reduce the problem into an initial value problem as follows: Let $R>0$ and define $v(x)=u\lbr \frac{x}{R}\rbr$, then, for $\vert x\vert<R$, $v$ satisfies:
\begin{eqnarray*}
\displaystyle -\dv(\vert\nabla v\vert^{n-2}\nabla v) =-\frac{1}{R^n}\dv(\vert\nabla v\vert^{n-2}\nabla v)(\frac{x}{R})=\frac{\lambda}{R^n} f(v).
\end{eqnarray*}
Choosing $R=\lambda^{\frac{1}{n}}$, we see that  $v$ satisfies 
\begin{eqnarray}\label{eq2.1}
\left\{\begin{array}{ll}
& -(r^{n-1}\vert v'\vert^{n-2}v')'=f(v)r^{n-1}\quad\mbox{in }(0,R),\\
& v>0\quad \mbox{in }(0,R),\\
& v'(0)=v(R)=0,
\end{array}\right.
\end{eqnarray}
where $v'$ denotes the derivatives of $v$ with respect to $r:=\vert x\vert$. Hence, for $r>0$, consider the following initial value problem:
\begin{eqnarray}\label{eq2.2}
\left\{\begin{array}{ll}
&-(r^{n-1}\vert w'\vert^{n-2}w')'=f(w)r^{n-1},\\
& w(0)=\gamma, w'(0)=0.
\end{array}\right.
\end{eqnarray}
We denote by $R(\gamma)$ to be the first zero of $w$ defined by
\begin{eqnarray*}
\displaystyle R(\gamma)=\sup\{r \ : \, w(s)>0 \quad\mbox{for all } s\in [0,r]\}.
\end{eqnarray*}
Let $v(0)=\gamma$, then \eqref{eq2.1} and \eqref{eq2.2} are related by $w(0)=\gamma$, $R(\gamma)=R$, $v(r)=w(r)$ and 
\begin{equation*}\label{eq2.3}
R(\gamma)=\lambda^{\frac{1}{n}}.
\end{equation*}
Hence the existence and uniqueness of solutions \eqref{eq1.8} are related to studying the behaviour of $\gamma\mapsto R(\gamma)$. Denote the solution of \eqref{eq2.2} by $w=w(r,\gamma)$ and now consider the Sturm's change of variables:
\begin{eqnarray*}
r=ne^{-\frac{t}{n}},\quad y(t,\gamma)=w(r,\gamma),\quad R(\gamma)=ne^{-\frac{T(\gamma)}{n}},
\end{eqnarray*}
then $y$ satisfies
\begin{eqnarray}\label{eq2.4}
\left\{\begin{array}{ll}
&-(\vert y'\vert^{n-2}y')'=f(y) e^{-t}\quad\mbox{in } (T(\gamma),\infty),\\
& y>0\quad\mbox{in } (T(\gamma),\infty),\\
& y(\infty,\gamma)=\gamma,\quad y'(\infty,\gamma)=0, \quad y(T(\gamma),\gamma)=0.
\end{array}\right.
\end{eqnarray}
Before we proceed further, we prove the following important Lemma:
\begin{lemm}\label{lemma2.1}
Let $f\in C^1(0,\infty)\cap C^0[0,\infty)$ such that $f\geq 0$ and $\gamma>0$. Then, there exists a unique solution to \eqref{eq2.4} such that $\gamma\mapsto T(\gamma)$ and  $\gamma\mapsto y(t,\gamma)$ are $C^1$ mappings. Furthermore if there exists a  $p \geq 0$ and a corresponding constant $C_p>0$ with  $f$ satisfying
\begin{eqnarray}\label{eq2.5}
\displaystyle\varlimsup_{\theta\to 0}\frac{f(\theta)}{\theta^p}=C_p,
\end{eqnarray}
then there exists $\theta_0\in \R$ such that
\begin{eqnarray}\label{eq2.6}
\left\{\begin{array}{lll}
&\displaystyle\varlimsup_{\gamma\to 0} T(\gamma)=-\infty\quad \mbox{if } p>n-1,\\
\\
&\displaystyle\varlimsup_{\gamma\to 0} T(\gamma)<\theta_0\quad \mbox{if } p=n-1,\\
\\
&\displaystyle\varlimsup_{\gamma\to 0} T(\gamma)=\infty\quad \mbox{if } 0\leq p <n-1.
\end{array}\right.
\end{eqnarray}
\end{lemm}
\proof
Consider the following initial value problem
\begin{eqnarray}\label{eq2.7}
\left\{\begin{array}{ll}
&-(\vert y'\vert^{n-2}y')'=f(y)e^{-t}=e^{g(y)-t}, \\
&y(\infty)=\gamma,\ \  y'(\infty)=0.
\end{array}\right.
\end{eqnarray}
Local existence and uniqueness of solution to \eqref{eq2.7} were proved in \cite{NaTh}. Let $T(\gamma)$ denote the first zero of $y$ as defined by
\begin{eqnarray*}
T(\gamma):=\inf\{t\ : \  y(s)>0 \mbox{ for all  } s\in (t,\infty)\}.
\end{eqnarray*}
Integrating \eqref{eq2.7}, we see that 
\begin{eqnarray}\label{eq2.8}
\vert y'(t)\vert^{n-2}y'(t)=\int_t^\infty f(y(s))e^{-s}{\rm d}s, 
\end{eqnarray}
which gives $y'>0$ and hence we see that  $y$ must be an increasing function. We must either have  $T(\gamma)>-\infty$ or $T(\gamma)=-\infty$.

Suppose $T(\gamma)=-\infty$, then $y(t)>0$ for all $t\in\R$. Fixing any  $t\leq 0$, we get 
\begin{eqnarray*}
y'(t)^{n-1}=\int_t^\infty f(y(s))e^{-s}{\rm d}s\geq \int_0^\infty f(y(s))e^{-s}{\rm d}s\eqdef\beta^{n-1}>0.
\end{eqnarray*}
Integrating the above expression, we obtain $0\leq y(t)\leq y(0)+\beta t\to-\infty$ as $t\to -\infty$ which is a contradiction. Hence $T(\gamma)\in\R$ and $y$ is an increasing function in $(T(\gamma),\infty)$.

Denote the solution $y(\cdot)\eqdef y(\cdot,\gamma)$, then  \cite{NaTh} gives  $\gamma\mapsto y(t,\gamma)$ and  $\gamma\mapsto T(\gamma)$ are  $C^1$ maps.

Let $f$ satisfy \eqref{eq2.5}, then integrating \eqref{eq2.8}, we obtain
\begin{eqnarray}\label{eq2.9}
\gamma=y(T(\gamma))+\int_{T(\gamma)}^\infty\left(\int_\theta^\infty f(y)e^{-s} {\rm d}s\right)^{\frac{1}{n-1}}{\rm d}\theta.
\end{eqnarray}
Let $p\geq n-1$, then from \eqref{eq2.5}, it is easy to see that there exists constants  $c>0$ and  $\varepsilon<1$ such that for all $\gamma\in (0,\varepsilon)$, the following holds:
\begin{eqnarray}\label{eq2.10}
f(\gamma)\leq c\gamma^{p}.
\end{eqnarray}
From \eqref{eq2.9} and \eqref{eq2.10}, we now have
\begin{eqnarray*}
\gamma\leq (c\gamma^p)^{\frac{1}{n-1}}\int_{T(\gamma)}^\infty\left(\int_\theta^\infty e^{-s}{\rm d}s\right)^{\frac{1}{n-1}}{\rm d}\theta=c^{\frac{1}{n-1}}(n-1)\gamma^{\frac{p}{n-1}} e^{-\frac{T(\gamma)}{n-1}},
\end{eqnarray*}
that is 
\begin{equation*}
 \begin{array}{cl}
T(\gamma)\leq \log(c (n-1)^{n-1}) +(p-(n-1))\log(\gamma)\leq \log(c (n-1)^{n-1}) &\text{ if } p=n-1 \\
\\
\lim_{\gamma \rightarrow \infty}T(\gamma)\to -\infty &\textrm{ if } p>n-1 .
 \end{array}
\end{equation*}

Let $0\leq p<n-1$ and suppose for some sequence $\gamma\to 0$, $T(\gamma)$ satisfies
\begin{eqnarray*}
\alpha:=\displaystyle\limsup_{\gamma\to 0^+}T(\gamma)<\infty.
\end{eqnarray*}
Let $m\in \R$ and $\beta=e^{\frac{T(\gamma)-m}{(n-1)-p}}$,  then from the above assumption, $\beta$ is bounded as $\gamma\to 0$. Now define $w$ by
\begin{eqnarray*}
w(t)\eqdef \beta y(t-m+T(\gamma)),
\end{eqnarray*} which satisfies
\begin{eqnarray*}
\left\{\begin{array}{lll}
&-((w')^{n-1})'(t)=\left(\frac{f(y(t-m+T(\gamma)))}{y^p(t-m+T(\gamma))}\right)w^{p-(n-1)}(t)w^{n-1}(t)e^{-t}\quad\mbox{in }(m,\infty)\\
& w>0 \quad \mbox{in }(m,\infty),\\
& w(\infty)=\beta \gamma,\quad w'(\infty)=w(m)=0.
\end{array}\right.
\end{eqnarray*}
Let $\varphi$ and $m_0$ satisfy
\begin{eqnarray*}
\left\{\begin{array}{lll}
&-((\varphi')^{n-1})'(t)=\varphi^{n-1}e^{-t}\quad\mbox{in }(m_0,\infty)\\
& \varphi>0 \quad \mbox{in }(m_0,\infty),\\
& \varphi(\infty)=1,\quad \varphi'(\infty)=\varphi(m_0)=0.
\end{array}\right.
\end{eqnarray*}
Let $m=m_0+1$ and define  $a(t):=\frac{f(y(t-m+T(\gamma)))}{y^p(t-m+T(\gamma))}w(t)^{p-(n-1)}$. Then from \eqref{eq2.5} for $\gamma$ sufficiently small, we have that $a(t)\geq 1$ for $t\in (m_0,\infty)$. Since $m>m_0$ and $w>0$ in $(m,\infty)$, we have by the Picone's identity
\begin{eqnarray*}
0\leq \int_{m_0}^\infty\left((\varphi')^n-\left(\frac{\varphi^n}{w^{n-1}}\right)'(w')^{n-1}\right){\rm d}t=\int_{m_0}^\infty(1-a(t))\varphi^n e^{-t}{\rm d}t<0
\end{eqnarray*}
which is a contradiction and this proves the lemma. \hfill \qed

Next we consider the behaviour of $\gamma\to T(\gamma)$ as $\gamma\to\infty$ and we have the following Theorem whose proof will be given in section \ref{section3}.
\begin{theo}\label{theo2.5}
Assume that $f$ satisfies (H1), (H2) of Hypothesis \ref{hypothesis-f} and let $\beta>0$ as in \eqref{eq1.4}. Then there exists a $\gamma_0>0$ such that for all $\gamma>\gamma_0$,we have
\begin{eqnarray}\label{2.11}
\displaystyle y'(T(\gamma),\gamma)=\frac{n}{n-1}\frac{1}{g'}+\frac{n^2\alpha_n g''}{(n-1)(g')^3}+\bo\left(\frac{\delta^2g''}{(g')^4}+\frac{e^{-(g-(\frac{n-1}{n})\gamma g')}}{g'}\right),
\end{eqnarray}
\begin{eqnarray}\label{eq2.12}
\displaystyle T(\gamma)&=&(g-\left(\frac{n-1}{n}\right)\gamma g')+(n-1)\log(\left(\frac{n-1}{n}\right)g')+\frac{\alpha_n(n-1)\gamma g''}{g'}\nonumber\\
&+&\bo\left(\frac{(\log g')^2}{g'}
+\frac{(\log g')^{\beta +1}}{(g')^\beta}+ e^{-(g-\left(\frac{n-1}{n}\right)\gamma g')}\right)
\end{eqnarray}
where $g=g(\gamma)$, $g'=g'(\gamma)$, $g''=g''(\gamma)$, $g'''=g'''(\gamma)$, $\delta=\log(g')$ and $\alpha_n=1+\frac{1}{2}+\cdot\cdot\cdot+\frac{1}{n}$. 
\end{theo}
\noindent {\it Linearization}: Let $V_1(t)= V_1(t,\gamma)=\frac{\partial y(t,\gamma)}{\partial\gamma}$. Then differentiating \eqref{eq2.4} with respect to $\gamma$, we obtain:
\begin{eqnarray}\label{eq2.13}
\left\{\begin{array}{ll}
&-((y')^{n-2}V_1')'(t)=\frac{f'(y(t))V_1e^{-t}}{n-1}\quad\mbox{in }(T(\gamma),\infty)\\
& V_1(\infty)=1, V_1'(\infty)=0.
\end{array}\right.
\end{eqnarray}
Differentiating $y(T(\gamma),\gamma)=0$, we obtain
\begin{eqnarray}\label{eq2.17}
T'(\gamma)=-\frac{\frac{\partial y(T(\gamma),\gamma)}{\partial\gamma}}{y'(T(\gamma),\gamma)}=-\frac{V_1(T(\gamma),\gamma)}{y'(T(\gamma),\gamma)}
\end{eqnarray}
where $T'(\gamma)$ denotes the derivative of $T(\gamma)$ with respect to $\gamma$.
Then using Theorem \ref{theo2.5} together with the asymptotics of $V_1(T(\gamma),\gamma)$ proved in Section \ref{section4}, we are able to prove the following asymptotic behaviour of $T'(\gamma)$ as $\gamma\to\infty$.
\begin{theo}\label{theo2.6}
There exists a $\gamma_0>0$ such that for all $\gamma>\gamma_0$
\begin{eqnarray}\label{eq2.14}
 T'(\gamma) = \frac{1}{n}(g'-(n-1)\gamma g'')+\bo\left(g''\frac{(\log(g'))^4 }{g'}\right).
\end{eqnarray}
\end{theo}

\begin{proof}[Proof of Theorem \ref{theo1.5}] 
From \eqref{eq2.2}, $\frac{\gamma g''}{g'}=\bo(1)$ and hence from (H2) of Hypothesis \ref{hypothesis-f} and \eqref{eq2.12}, we see that $T(\gamma)\to\infty$ as $\gamma\to\infty$. That is $\lambda^{\frac{1}{n}}=R(\gamma)\to 0$ as $\gamma\to\infty$. Hence in order to prove the theorem, it is enough to show that there exists a $\gamma_0$ large such that for all $\gamma>\gamma_0$, $T(\gamma)$ is strictly increasing function.  From (H3), we see that $g'-(n-1)\gamma g''>0$ and $\frac{g''}{g'}(\log(g'))^4 =o(1)(g'-(n-1)\gamma g'')$ .
Hence from \eqref{eq2.14},  we have for $\gamma_0$ large and $\gamma>\gamma_0$,
\begin{eqnarray*}
T'(\gamma)>0.
\end{eqnarray*}
 This gives that $\gamma\mapsto T(\gamma)$ is a strictly increasing  function and $y$ is nondegenerate. This proves Theorem \ref{theo1.5}.
 \end{proof} 
 
\begin{proof}[Proof of Theorem \ref{theosingular}] If $u$ is a solution of \eqref{eq1.8singular}, then, as earlier, converting the equation \eqref{eq1.8singular} to an initial value problem like in \eqref{eq2.2} we have for $R(\gamma)=\lambda^{\frac{1}{n}}$, $y(t,\gamma)=u(r,\gamma)$, $u(0)=\gamma$ and $r=n e^{-\frac{t}{n}}$;  $y$ satisfies:
\begin{eqnarray*}
-((y')^{n-1})'=\frac{f(y)e^{-t}}{n^\beta e^{-\frac{\beta t}{n}}}=\frac{1}{n^\beta}f(y) e^{-(1-\frac{\beta}{n})t}.
\end{eqnarray*}
Let $a=(1-\frac{\beta}{n})$ and $y(t)=y(\frac{t}{a})$, then $y$ satisfies
\begin{eqnarray*}
\left\{\begin{array}{ll}
&-((y')^{n-1})'=\frac{f(y)}{n^\beta a^n}e^{-t}\quad \mbox{in } (T(\gamma),\infty)\\
&y(\infty)=\gamma,\; y'(\infty)=0,\; y(T(\gamma)=0.
\end{array}\right.
\end{eqnarray*}
Let $\tilde{f}(y)\eqdef\frac{f(y)}{n^\beta a^n}$, then $\tilde{g}(s)\eqdef \log(\tilde{f}(s))=\log(f(s))+\log(\frac{1}{n^\beta a^n})=g(s)+\log(\frac{1}{n^\beta a^n})$. Therefore the theorem now follows from Theorem \ref{theo1.5}.
\end{proof}
\section{Proof of Theorem \ref{theo2.5}.}\label{section3}
We shall use the following notation throughout the rest of the paper.

\noindent {\bf Notation:} Let $A(\gamma)$, $B(\gamma)$ be two functions on any interval $J\subset \R$. We then say $A(\gamma)\sim B(\gamma)$ if there exists a constant $C>0$ such that 
\begin{eqnarray}\label{eq2.15}
\frac{1}{C}A(\gamma)\leq B(\gamma)\leq C A(\gamma) \quad \text{  holds for all  } \quad \gamma\in J.
\end{eqnarray}
\begin{prop}\label{proposition3.1}
Let $a>0$, $q>0$, $A$ and $B$ are functions on $[s_0,\infty)$ such that 
\begin{gather*}
\label{eq3.2} A(s)=as^q+B(s), \\
\label{eq3.3} \displaystyle\lim_{s\to\infty}\frac{B(s)}{s^q}=0.
\end{gather*}
Let $0<\delta\leq C\log(s)$ for some $C>0$. Then there exists an $s_1=s_1(\delta)>0$ and a $C_1=C_1(\delta)>0$ such that for all $s\geq s_1$, $\eta\in [s-\delta,s]$:
\begin{gather}
\label{eq3.4} A(s)\sim s^q \\
\label{eq3.5} \frac{1}{C_1} A(s)\leq A(\eta)\leq C_1 A(s) \nonumber.
\end{gather}
\end{prop}
\proof \eqref{eq3.4} follows from \eqref{eq3.3}. Let $\eta=s-\alpha$ with $0\leq\alpha\leq \delta\leq C\log(s)$, and 
\begin{eqnarray*}
\displaystyle A(\eta)=a(s-\alpha)^q+B(s-\alpha)= \frac{A(s)}{(1+\frac{B(s)}{a s^q})}\left(1-\frac{\alpha}{s}\right)^q\left(1+\frac{B(s-\alpha)}{a(s-\alpha)^q}\right).
\end{eqnarray*}
Hence,
$\displaystyle\lim_{s\to\infty}\frac{A(\eta)}{A(s)}=1$
and this proves the proposition.\qed
\begin{prop}\label{proposition3.2}
Let $g$ satisfy (H1) and $g^{(k)}$ denote the $k^{th}$ derivative of $g$. Then there exists a $s_0$ such that for all $\gamma>s_0$:
\begin{itemize}
\item[(i)] $k\in \{0,1,2\}$, $g^{(k)}(\gamma)\sim \gamma^{q-k}$ and $g^{(3)}(\gamma)=\bo(\gamma^{q-3})$,
\item[(ii)] Let $0<\delta\leq C_1\log(g'(\gamma))$, $\eta\in [\gamma-\delta,\gamma]$, $0\leq \theta_1$,$\theta_2\leq \delta$ and $k\in\{0, 1,2\}$. Then there exist $C_2=C_2(\delta)$ and $\gamma_0=\gamma_0(\delta)$ such that for all $\gamma\geq \gamma_0$,
\begin{eqnarray*}
\frac{1}{C_2}g^{(k)}(\gamma)\leq g^{(k)}(\eta)\leq C_2 g^{(k)}(\gamma),
\end{eqnarray*}
and
\begin{eqnarray*}
\frac{\gamma g^{(k+1)}(\gamma-\theta_1)}{g^{(k)}(\gamma-\theta_2)}=\bo(1).
\end{eqnarray*}
\item[(iii)] $g^{-1}(g(\gamma)-\eta)=\gamma-\frac{\eta}{g'(\gamma)}+\bo\left(\frac{\eta^2 g^{(2)}(\gamma)}{(g'(\gamma))^3}\right)$.
\end{itemize}
\end{prop}
\proof
(i) follows from \eqref{eq3.4}. (ii) follows from (H1) of Hypothesis \ref{hypothesis-f}, (i) of Proposition \ref{proposition3.2} and Proposition \ref{proposition3.1}.
Differentiating $\gamma=g^{-1}(g(\gamma))$, we obtain
\begin{eqnarray*}
\displaystyle (g^{-1})'(g(\gamma))=\frac{1}{g'(\gamma)}, \quad (g^{-1})^{(2)}(g(\gamma))=-\frac{g^{(2)}(\gamma)}{(g'(\gamma))^3},\quad \displaystyle (g^{-1})^{(2)}(\xi)=-\frac{g^{(2)}(g^{-1}(\xi))}{(g'(g^{-1}(\xi)))^3}.
\end{eqnarray*}
For $\gamma$ large, $g(\frac{\gamma}{2})\leq g(\gamma)-\eta$ and hence for $\xi\in [g(\gamma)-\eta, g(\gamma)]$, $\frac{\gamma}{2}\leq g^{-1}(\xi)\leq \gamma$. Hence, from (i) and (ii), we see that 
\begin{eqnarray*}
(g'(g^{-1}(\xi)))\sim g'(\gamma) \quad\mbox{and } g^{(2)}(g^{-1}(\gamma))\sim g^{(2)}(\gamma).
\end{eqnarray*}
Therefore, by Taylor's theorem, there exists an $\xi\in [g(\gamma)-\eta, g(\gamma)]$ such that 
\begin{eqnarray*}
&\displaystyle g^{-1}(g(\gamma)-\eta)=\gamma-\frac{\eta}{g'(\gamma)}+\frac{1}{2}(g^{-1})^{(2)}(\xi)\eta^2\\
&\displaystyle=\gamma-\frac{\eta}{g'(\gamma)}-\frac{1}{2}\frac{g^{(2)}(g^{-1}(\xi))\eta^2}{(g'(g^{-1}(\xi)))^3}\\
&\displaystyle=\gamma-\frac{\eta}{g'(\gamma)}+\bo(\frac{\eta^2 g^{(2)}(\gamma)}{(g'(\gamma))^3}).
\end{eqnarray*}
This proves (iii) and hence the proposition. \qed

\begin{prop}\label{proposition3.3}
Let $a>0$ and $F(x,b)=x^n-a x^{n-1} -b$. Then there exists an $\varepsilon_0>0$ and a smooth function $X\, :\, (-\varepsilon_0,\varepsilon_0)\to\R$ such that for $\vert b\vert <\varepsilon_0$:
\begin{eqnarray*}
X(0)=a,\, F(X(b),b)=0,\, X(b) =a+\frac{b}{a^{n-1}}+ \bigo{\frac{b^2}{a^{2n-1}}}.
\end{eqnarray*}
\end{prop}
\proof Since $F(a,0)=0$ and $\frac{\partial F}{\partial x}(a,0)= a^{n-1} \neq 0$, hence by implicit function theorem, there exists $\varepsilon_0>0$ and a unique smooth function $X\,:\, (-\varepsilon_0,\varepsilon_0)\to\R$ such that $X(0)=a$ and $F(X(b),b)=0$ for $\vert b\vert<\varepsilon_0$. Furthermore,
\begin{equation*}
\begin{array}{rl}
0& =\frac{{\rm d}}{{\rm d}b}F(X(b),b)=\frac{\partial F}{\partial x}(X(b), b)X'(b)+\frac{\partial F}{\partial b}(X(b),b),\\
\\
X'(b)& =\frac{1}{n X(b)^{n-1}-(n-1)a X(b)^{n-2}} \\ \\
X''(b) & =-\left(\frac{n(n-1)X(b)^{n-2}-(n-1)(n-2)a X(b)^{n-3}}{(n X(b)^{n-1}-(n-1)a X(b)^{n-2})^2}\right)X'(b).
\end{array}
\end{equation*}
%
Therefore, $X'(0)=\frac{1}{a^{n-1}}$, $X''(0)=-\frac{2(n-1)}{a^{2n-1}}$ and
\begin{eqnarray*}
X(b)=X(0)+X'(0)b+\bo(X''(0)b^2)=a+\frac{b}{a^{n-1}}+\bo\left(\frac{b^2}{a^{2n-1}}\right).
\end{eqnarray*}
This completes the proof of the proposition. \qed

With these preliminary propositions, we will move on to the proof of the theorem.
Let $s_0>0$ such that $g$ is strictly increasing convex function in $[s_0,\infty)$ and for  $t\in\R$, $\theta\geq s_0$, $\gamma\geq s_0$  with $g^{(k)}=g^{(k)}(\gamma)$, define the following:
\begin{gather}
 \label{eq3.7} T_1 : =g+(n-1)\log\left(\frac{(n-1)g'}{n}\right),\nonumber  \\
 \label{eq3.8} T_0 :=g-\left(\frac{n-1}{n}\right)\gamma g'+(n-1)\log\left(\frac{(n-1)g'}{n}\right)-(n-1)\log\left(1-e^{\frac{-\gamma g'}{n}}\right),\\
 \label{eq3.9} z(t) :=\gamma-\frac{n}{g'}\log\left(1+e^{\frac{T_1-t}{n-1}}\right), \nonumber \\
 \label{eq3.10} \psi(\theta) :=g(\theta)-g+\left(\frac{n-1}{n}\right)(\gamma-\theta)g'-(n-1)\log\left(\frac{(n-1)g'}{n}\right).
\end{gather}
Then from direct computations, $z$ satisfies the following
\begin{eqnarray}\label{eq3.11}
z(T_0)=0, \quad z(t)>0\,\mbox{ for }t>T_0,
\end{eqnarray}
\begin{eqnarray}\label{eq3.12}
\displaystyle z'(t)=\left(\frac{n}{n-1}\right)\frac{1}{g'}\frac{e^{\frac{T_1-t}{n-1}}}{1+e^{\frac{T_1-t}{n-1}}}= \left(\frac{n}{n-1}\right)\frac{1}{g'}\left(1-\frac{1}{1+e^{\frac{T_1-t}{n-1}}}\right),
\end{eqnarray}
\begin{eqnarray}\label{eq3.13}
-((z'(t))^{n-1})'=\left(\frac{n}{(n-1)g'}\right)^{n-1}\frac{e^{T_1-t}}{(1+e^{\frac{T_1-t}{n-1}})^n}=e^{g-t+g'(z-\gamma)}
\end{eqnarray}
\begin{eqnarray}\label{eq3.14}
z(\infty)=\gamma,\quad z'(\infty)=0.
\end{eqnarray}
Let $\gamma>s_0$ and $y(t)=y(t,\gamma)$ be the corresponding solution of \eqref{eq2.7}.
Let $\tilde{T}>T(\gamma)$ be defined by 
\begin{eqnarray}\label{eq3.16}
y(\tilde{T})=s_0.
\end{eqnarray}
Then we have the following Lemma: 
\begin{lemm}\label{lemma3.3}
Let $\delta= \bo(\log{g'})$ and
\begin{eqnarray}\label{eq3.17}
T_\delta=T_1-(n-1)\delta.
\end{eqnarray}
Then, for all $\tilde{T}\leq t$,
\begin{eqnarray}\label{eq3.18}
y(t)\leq z(t),
\end{eqnarray}
\begin{eqnarray}\label{eq3.19}
g(y(t))\geq g-n\log\left(1+e^{\frac{T_1-t}{n-1}}\right)=g+g'(z-\gamma)
\end{eqnarray}
\begin{eqnarray}\label{eq3.20}
g(y(t))-t\leq \psi(y(t)).
\end{eqnarray}
Furthermore, there exists a $\gamma_0=\gamma_0(\delta)>s_0$ such that  for all $\gamma\geq \gamma_0$, $t\geq T_\delta$,
\begin{eqnarray}\label{eq3.21}
\tilde{T}< T_\delta,
\end{eqnarray}
\begin{eqnarray}\label{eq3.22}
\displaystyle y(t)=z(t)+\bo\left(\frac{g''\log(1+e^\delta)^2}{(g')^3}\right),
\end{eqnarray}
\begin{eqnarray}\label{eq3.23}
\displaystyle 0\leq g(y(t))-g+n\log\left(1+e^{\frac{T_1-t}{n-1}}\right)\leq \frac{n^2 g''}{2(g')^2}\log(1+e^{\frac{T_1-t}{n-1}})^2+\bo\left(\frac{g'''\log(1+e^\delta)^3}{(g')^3}\right)
\end{eqnarray}
\begin{eqnarray}\label{eq3.24}
\displaystyle y'(t)=z'(t)e^{\bo\left(\frac{g''\log(1+e^\delta)^2}{(g')^2}\right)}
\end{eqnarray}
and
\begin{eqnarray}\label{eq3.25}
\displaystyle g'(y(t))=g'-\frac{ng''}{g'}\left(\frac{T_1-t}{n-1}\right)-\frac{ng''}{g'}\log\left(1+e^{\frac{t-T_1}{n-1}}\right)+\bo\left(\frac{g''}{(g')^2}\log(1+e^\delta)^2\right).
\end{eqnarray}
\end{lemm}
\proof
Integrating \eqref{eq2.7} from $t$ to $\infty$ to obtain
\begin{eqnarray}\label{eq3.25b}
y'(t)^{n-1}=\int_t^\infty f(y(s))e^{-s}{\rm d}s.
\end{eqnarray}
Integrating \eqref{eq3.25b} along with Fubini, we get
\begin{eqnarray}\label{eq3.26}
\displaystyle y(t)=\gamma-\int_t^\infty f(y(s))e^{-s}\left(\int_t^s\frac{{\rm d}\theta}{y'(\theta)^{n-2}}\right){\rm d}s.
\end{eqnarray}
Since $\theta\mapsto g(\theta)$ and hence $\theta\mapsto f(\theta)$ are increasing functions for $\theta\geq s_0$ and hence for $t\geq \tilde{T}$, we see that 
\begin{eqnarray*}
f(y(t))e^{-t}\leq \int_t^\infty f(y(s))e^{-s}{\rm d}s\leq f(\gamma) e^{-t}.
\end{eqnarray*}
This implies 
\begin{eqnarray}\label{eq3.27}
\displaystyle\lim_{t\to\infty}(t+(n-1)\log(y'(t)))=g.
\end{eqnarray}
Define 
\begin{eqnarray}\label{eq3.28}
\displaystyle E(t): =y'(t)^{n-1}-\left(\frac{n-1}{n}\right)y'(t)^ng'(y(t))-e^{g(y(t))-t}.
\end{eqnarray}
It is easy to that  $E(\infty)=0$, and now  differentiating \eqref{eq3.28},  we obtain
\begin{eqnarray}\label{eq3.29}
\displaystyle E'(t)=-\left(\frac{n-1}{n}\right)(y'(t))^{n+1}g''(y(t))<0 \qquad \text{for  } t > \tilde{T},
\end{eqnarray}
since $g''(\theta)>0$ for $\theta> s_0$ by assumption. Hence, $E(t)$ is a decreasing function on $[\tilde{T}, \infty)$ and $E(t)>E(\infty)=0$. Let $t>\tilde{T}$, then integrating \eqref{eq3.29}, we obtain
\begin{eqnarray}\label{eq3.30}
\displaystyle 0\leq y'(t)^{n-1}-\left(\frac{n-1}{n}\right)y'(t)^n g'(y(t))+\left(y'(t)^{n-1}\right)'
=\left(\frac{n-1}{n}\right)\int_t^\infty y'(s)^{n+1}g''(y(s)){\rm d}s
\end{eqnarray}
which implies
\begin{eqnarray*}
0 & \leq & 1-\left(\frac{n-1}{n}\right)y'(t)g'(y(t))+\frac{(y'(t)^{n-1})'}{y'(t)^{n-1}} \\
 & = &\left(\frac{n-1}{n}\right)\int_t^\infty\left(\frac{y'(s)}{y'(t)}\right)^{n-1}y'(s)^2g''(y(s)){\rm d}s\\
 &\leq & 	\left(\frac{n-1}{n}\right) y'(t)\int_t^\infty y'(s)g''(y(s)){\rm d}s=	\left(\frac{n-1}{n}\right) y'(t)(g'-g'(y(t)).
\end{eqnarray*}
We rewrite the above inequality as 
\begin{eqnarray*}
\displaystyle 0\leq \left(t-\left(\frac{n-1}{n}\right)g(y(t))+(n-1)\log y'(t)\right)'\leq\left(\frac{n-1}{n}\right)( y(t) g'-g(y(t)))',
\end{eqnarray*}
and now integrating the above expression and by making use of \eqref{eq3.27}, we obtain
\begin{eqnarray}
\label{ineq3.1}& \displaystyle 0\leq -\left(t-\left(\frac{n-1}{n}\right)g(y(t))+(n-1)\log y'(t)\right)+ g-\left(\frac{n-1}{n}\right)g\\
\label{ineq3.2}&\leq \left(\frac{n-1}{n}\right)[(\gamma-y(t))g'-g+g(y(t)].
\end{eqnarray}
Inequality \eqref{ineq3.1} gives 
\begin{eqnarray*}
\displaystyle y'(t) e^{-\frac{g(y(t))}{n}}\leq e^{\frac{g}{n(n-1)}-\frac{t}{n-1}}.
\end{eqnarray*}
Since $g'(y(t)) \geq 0$ and increasing on $[\tilde{T},\infty)$, we see that 
\begin{eqnarray*}
-\left(e^{-\frac{g(y(t))}{n}}\right)'\leq \frac{g'}{n}e^{\frac{g}{n(n-1)}-\frac{t}{n-1}}.
\end{eqnarray*}
Integrating the above expression from $t$ to $\infty$ and simplifying, we obtain 
\begin{eqnarray*}
g(y(t))\geq g-n\log\left(1+\left(\frac{n-1}{n}\right)g'e^{\frac{g-t}{n-1}}\right)=g-n\log\left(1+e^{\frac{T_1-t}{n-1}}\right),
\end{eqnarray*}
which proves inequality \eqref{eq3.19}. 

Now  inequality \eqref{ineq3.2} which is given by 
\begin{eqnarray*}
y'(t) e^{\frac{g'}{n}(\gamma-y(t))}\geq e^{\frac{g-t}{n-1}}
\end{eqnarray*}
after integrating from $t$ to $\infty$ and simplifying, we obtain 
\begin{eqnarray*}
y(t)\leq \gamma-\frac{n}{g'}\log\left(1+\left(\frac{n-1}{n}\right)g'e^{\frac{g-t}{n-1}}\right)=\gamma-\frac{n}{g'}\log\left(1+e^{\frac{T_1-t}{n-1}}\right)
\end{eqnarray*}
which proves \eqref{eq3.18}. 

Using \eqref{eq3.18}, it is easy to see that
\begin{eqnarray*}
e^{\frac{T_1-t}{n-1}}\leq e^{\frac{g'}{n}(\gamma-y(t))}, 
\end{eqnarray*}
which after taking logarithms and simplifying, implies 
\begin{eqnarray*}
g(y(t))-t 
\leq \psi(y(t))
\end{eqnarray*}
and this completes the proof of  \eqref{eq3.20}.

Proof of inequality \eqref{eq3.21} will be by contradiction, suppose $T_\delta\leq \tilde{T}$, then $\frac{T_1-\tilde{T}}{n-1}\leq \delta$, then \eqref{eq3.19} gives
\begin{eqnarray*}
g(s_0)=g(y(\tilde{T}))\geq g-n\log\left(1+e^{\frac{T_1-\tilde{T}}{n-1}}\right)\geq g+\bo(\log(g')) \rightarrow \infty \quad \text{as  } \gamma \rightarrow \infty,
\end{eqnarray*}
which is a contradiction and hence there exists a $\ga_0 = \ga_0(\de) > 0$ such that inequality \eqref{eq3.21} holds for all $\ga > \ga_0$.

Let $t>T_\delta$, then $\frac{T_1-t}{n-1}\leq\delta$ which also implies that  $\log\left(1+e^{\frac{T_1-t}{n-1}}\right)\leq \log\left(1+e^\delta\right)$. Let $\gamma\geq \gamma_0(\delta)$, then from \eqref{eq3.19}, \eqref{eq3.20}, \eqref{eq3.21} and  setting $\eta=n\log\left(1+e^{\frac{T_1-t}{n-1}}\right)$ in assertion (iii) of Proposition \ref{proposition3.2}, we have
\begin{eqnarray*}
y(t) &\geq & g^{-1}\left(g-n\log(\left(1+e^{\frac{T_1-t}{n-1}}\right))\right)\\
&=& \gamma-\frac{n}{g'}\log\left(1+e^{\frac{T_1-t}{n-1}}\right)+\bo\left(\frac{g''}{(g')^3}\log(1+e^\delta)^2\right)\\
&=& z(t)+\bo\left(\frac{g''}{(g')^3}\log(1+e^\delta)^2\right).
\end{eqnarray*}
This inequality along with \eqref{eq3.18} proves \eqref{eq3.22}. 

From Taylor's series, \eqref{eq3.18} and from Proposition \ref{proposition3.2}, there exists  $\xi\in \left [\gamma-\frac{n}{g'}\log\left(1+e^{\frac{T_1-t}{n-1}}\right),\gamma\right ]$ such that 
\begin{eqnarray*}
g(y(t))&\leq & g(z(t))\\
&=& g-n\log\left (1+e^{\frac{T_1-t}{n-1}}\right)+\frac{g''n^2\log\left(1+e^{\frac{T_1-t}{n-1}}\right)^2}{2(g')^2} +\frac{g'''(\xi)n^3\log\left(1+e^{\frac{T_1-t}{n-1}}\right)^3}{3!(g')^3}\\
&=& g-n\log\left(1+e^{\frac{T_1-t}{n-1}}\right)+\frac{g''n^2\log\left(1+e^{\frac{T_1-t}{n-1}}\right)^2}{2(g')^2} + \bo\left(\frac{g'''}{(g')^3}\log\left(1+e^{\delta}\right)^3\right).
\end{eqnarray*}
The above inequality together with \eqref{eq3.19} proves \eqref{eq3.23}.

From \eqref{eq3.19}, it is easy to see that 
\begin{eqnarray}\label{compar-y-z}
y'(t)^{n-1}=\int_t^\infty e^{g(y(s))-s}{\rm d}s\geq\int_t^\infty e^{g+g'(z(s)-\gamma)-s}{\rm d}s=z'(t)^{n-1}.
\end{eqnarray}
Inequality  \eqref{eq3.23} implies
\begin{eqnarray*}
y'(t)^{n-1}&=&\int_t^{\infty}e^{g(y(s))-s}{\rm d}s\\
&\leq &e^{\bo(\frac{g''}{(g')^2}\log\left(1+e^\delta\right)^2)}\int_t^\infty e^{g+g'(z-\gamma)-s}{\rm d}s\\
&=& e^{\bo(\frac{g''}{(g')^2}\log\left(1+e^\delta\right)^2)}z'(t)^{n-1}
\end{eqnarray*}
and this proves \eqref{eq3.24}.

From \eqref{eq3.22}, we have
\begin{eqnarray*}
g'(y(t))&=&g'\left(\gamma-\frac{n}{g'}\log\left(1+e^{\frac{T_1-t}{n-1}}\right)+\bo\left(\frac{g''}{(g')^3}\log\left(1+e^\delta\right)^2\right)\right)\\
&=& g'-\frac{ng''}{g'}\log\left(1+e^{\frac{T_1-t}{n-1}}\right)+\bo\left(\frac{g''}{(g')^2}\log\left(1+e^\delta\right)^2\right)\\
&=&g'-\frac{ng''}{g'}\left(\frac{T_1-t}{n-1}\right)-\frac{ng''}{g'}\log\left(1+e^{\frac{t-T_1}{n-1}}\right)+\bo\left(\frac{g''}{(g')^2}\log\left(1+e^\delta\right)^2\right)
\end{eqnarray*}
which proves \eqref{eq3.25} and this completes the proof of the lemma.\qed
\begin{lemm}\label{lemma3.4}
Let $k\geq 1$ be an integer and $X(t)\eqdef e^{\frac{T_1-t}{n-1}}$. Then the following statements hold:
\begin{gather}
 \label{eq3.6} \al_k : = 1 + \frac12 + \frac13 + \ldots + \frac1k = - \sum_{r=1}^k \frac{(-1)^r C^k_r }{r} \\
%
\label{eq3.32}\displaystyle\sum_{r=0}^k\frac{(-1)^rC_r^k}{(r+2)}=\frac{1}{(k+1)(k+2)}, \\
\label{eq3.33}\int_t^\infty(z'(s))^{k+1}{\rm d}s=\left(\frac{n}{(n-1)g'}\right)^{k+1}\left [-\alpha_k+\log(1+X(t))-\displaystyle\sum_{r=1}^k\frac{(-1)^r C_r^k}{r(1+X(t))^r}\right ].
\end{gather}
\end{lemm}
\proof
By explicit integration, we get 
$$\alpha_k=\displaystyle\sum_{r=0}^{k-1}\int_0^1 x^r{\rm d}x=\int_0^1\frac{1-x^k}{1-x}{\rm d}x
=\int_0^1\frac{1-(1-x)^k}{x}{\rm d}x=-\displaystyle\sum_{r=1}^k\frac{(-1)^r C^k_r}{r}.$$
Similar calculation gives
\begin{eqnarray*}
\displaystyle\sum_{r=0}^k\frac{(-1)^rC^k_r}{(r+2)}=\int_0^1(1-x)^kx{\rm d}x=\int_0^1 x^k(1-x){\rm d}x=\frac{1}{(k+1)(k+2)}.
\end{eqnarray*}
Since ${\rm d}X(s)=-\frac{X(s){\rm d}s}{(n-1)}$, making use of \eqref{eq3.6}, we get
\begin{eqnarray*}
\int_t^\infty z'(s)^{k+1}{\rm d}s&=&\left(\frac{n}{(n-1)g'}\right)^{k+1}\int_t^\infty\frac{X(s)^{k+1}}{(1+X(s))^{k+1}}{\rm d}s\\
&=&\left(\frac{n}{(n-1)g'}\right)^{k+1}(n-1)\int_0^{X(t)}\frac{x^k}{(1+x)^{k+1}}{\rm d}x\\
&=&\left(\frac{n}{(n-1)g'}\right)^{k+1}(n-1)\displaystyle\sum_{r=0}^k (-1)^r C^k_r\int_0^{X(t)}\frac{{\rm d}x}{(1+x)^{r+1}}\\
&=&\left(\frac{n}{(n-1)g'}\right)^{k+1}(n-1)\left [\log(1+x)-\displaystyle\sum_{r=1}^k\frac{(-1)^rC_r^k}{r(1+x)^r}\right]_0^{X(t)}\\
&=&\left(\frac{n}{(n-1)g'}\right)^{k+1}(n-1)\left [-\alpha_k+\log(1+X(t))-\displaystyle\sum_{r=1}^k\frac{(-1)^rC_r^k}{r(1+X(t))^r}\right ].
\end{eqnarray*}
This completes the proof of the  lemma.\qed

\begin{lemm}\label{lemma3.5}
Let $k\geq3$, $\delta=k\log(g')$ and $T_\delta=T_1-(n-1)\delta$. Then there exists a $\gamma_0=\gamma_0(k)$ such that for all $\gamma\geq \gamma_0$,
\begin{eqnarray}\label{eq3.34}
y'(T_\delta)=\frac{n}{(n-1)g'}+\frac{n^2\alpha_n g''}{(n-1)(g')^3}+\bo\left(\frac{g''\delta^2}{(g')^4}+\frac{1}{(g')^{k+1}}\right)
\end{eqnarray}
with $\alpha_n$ defined as in \eqref{eq3.6}.
\end{lemm}
\proof
To simplify the notation for the proof of this lemma, denote $y=y(T_\delta)$, $y'=y'(T_\delta)$, $z = z(T_{\delta})$ and $z' = z'(T_{\delta})$, then from \eqref{eq3.30} we have
\begin{eqnarray*}
(y')^{n-1}-\left(\frac{n-1}{n}\right)(y')^n g'&=&e^{g(y)-T_\delta}-\left(\frac{n-1}{n}\right)(g'-g'(y))(y')^n\\
&&+\left(\frac{n-1}{n}\right)\int_{T_\delta}^\infty g''(y(s))(y'(s)^{n+1}-(z'(s))^{n+1}){\rm d}s\\
&&+\left(\frac{n-1}{n}\right)\int_{T_\delta}^\infty z'(s)^{n+1}(g''(y(s))-g''){\rm d}s\\
&&+\left(\frac{n-1}{n}\right)g''\int_{T_\delta}^\infty z'(s)^{n+1}{\rm d}s
\eqdef J_1+J_2+J_3+J_4+J_5.
\end{eqnarray*}
Since $\log\left(1+e^{\frac{T_1-T_\delta}{n-1}}\right)=\log(1+e^\delta)=\delta+\log(1+e^{-\delta})=\delta+\bo\left(\frac{1}{(g')^k}\right)$ and $X(T_\delta)=e^\delta=(g')^k$, we have from \eqref{eq3.13} and \eqref{eq3.23} 
\begin{eqnarray*}
J_1= e^{g(y)-T_\delta}&=&e^{\bo\left(\frac{g''\delta^2}{(g')^2}\right)} e^{g+g'(z-\gamma)-T_\delta}\\
&=&\bo\left(((z')^{n-1})'\right)=\bo\left(\frac{1}{(g')^{n-1}}\frac{X(T_\delta)^{n-1}}{(1+X(T_\delta))^n}\right)=\bo\left(\frac{1}{(g')^{k+n-1}}\right).
\end{eqnarray*}
From \eqref{eq3.24} and \eqref{eq3.25} we have
\begin{eqnarray*}
\displaystyle y'(T_\delta)^n&=&z'(T_\delta)^n e^{\bo\left(\frac{g''}{(g')^2}\delta^2\right)}\\
&=&\left(\frac{n}{(n-1)g'}\right)^n\left(1-\frac{1}{1+X(T_\delta)}\right)^n e^{\bo\left(\frac{g''}{(g')^2}\delta^2\right)}\\
&=&\left(\frac{n}{(n-1)g'}\right)^n\left(1+\bo(\frac{1}{(g')^k})\right)e^{\bo\left(\frac{g''}{(g')^2}\delta^2\right)}\\
&=&\left(\frac{n}{(n-1)g'}\right)^n\left(1+\bo\left(\frac{g''\delta^2}{(g')^2}+\frac{1}{(g')^k}\right)\right)
\end{eqnarray*}
and 
\begin{eqnarray*}
J_2 &=&\left(\frac{n-1}{n}\right)(y')^n(g'(y)-g')\\
&=&\left(\frac{n-1}{n}\right)\left(\frac{n}{(n-1)g'}\right)^n(1+\bo\left(\frac{g''\delta^2}{(g')^2}\right))\left(-\frac{ng''}{g'}\delta+\bo\left(\frac{g''\delta^2}{(g')^2}\right)\right)\\
&=&-\left(\frac{n}{n-1}\right)^{n-1}\frac{n g''\delta}{(g')^{n+1}}+\bo\left(\frac{g''\delta^2}{(g')^{n+2}}\right).
\end{eqnarray*}
From \eqref{eq3.24} and \eqref{eq3.33},
\begin{eqnarray*}
J_3 &=&\left(\frac{n-1}{n}\right)\int_{T_\delta}^\infty (y'(s)^{n+1}-z'(s)^{n+1})g''(y(s)){\rm d}s\\
&=&\bo\left(\frac{(g'')^2\delta^2}{(g')^2}\int_{T_\delta}^\infty z'(s)^{n+1}{\rm d}s\right)=\bo\left(\frac{(g'')^2\delta^3}{(g')^{n+3}}\right).
\end{eqnarray*}
From Proposition \ref{proposition3.2} and \eqref{eq3.22}, for $s\geq T_\delta$, there exist a $\xi\in [y(s),\gamma]$ such that
\begin{eqnarray*}
g''(y(s))&=& g''+g'''(\xi)(y(s)-\gamma)=g''+\bo(g'''(z(s)-\gamma))+ \bo\left(\frac{g''' g''\delta^2}{(g')^3}\right)\\
&=&g''+\bo\left(\frac{g'''\delta}{g'}\right)=g''+\bo\left(\frac{\gamma g'''}{g''}\frac{g''\delta}{\gamma g'}\right)=g''+\bo\left(\frac{g''\delta}{(g')^{1+\frac{1}{q-1}}}\right).
\end{eqnarray*}
Hence from \eqref{eq3.33}
\begin{eqnarray*}
J_4&=&\left(\frac{n-1}{n}\right)\int_{T_\delta}^\infty z'(s)^{n+1}(g''(y(s))-g''){\rm d}s=\bo\left(\frac{g''\delta^2}{(g')^{n+2}}\right).
\end{eqnarray*}
From \eqref{eq3.33}
\begin{eqnarray*}
J_5&=&\left(\frac{n-1}{n}\right)g''\int_{T_\delta}^\infty z'(s)^{n+1}{\rm d}s\\
&=&(n-1)\left(\frac{n}{n-1}\right)^n\frac{g''}{(g')^{n+1}}\left(-\alpha_n+\delta+\bo\left(\frac{1}{(g')^k}\right)\right)\\
&=&n\left(\frac{n}{n-1}\right)^{n-1}\frac{g''}{(g')^{n+1}}(-\alpha_n+\delta+\bo\left(\frac{1}{(g')^k}\right)).
\end{eqnarray*}
Combining all the estimates to obtain
\begin{eqnarray*}
(y')^{n-1}-\left(\frac{n-1}{n}\right)g'(y')^n&=& J_1+J_2+J_3+J_4+J_5\\
&=& -\left(\frac{n}{n-1}\right)^{n-1}\frac{ng''\delta}{(g')^{n+1}}\\
&&+\left(\frac{n}{n-1}\right)^{n-1}\frac{ng''}{(g')^{n+1}}\left(-\alpha_n+\delta+\bo\left(\frac{1}{(g')^k}\right)\right)\\
&&+\bo\left(\frac{g''\delta^2}{(g')^{n+2}}+\frac{1}{(g')^{k+n-1}}\right)
\end{eqnarray*}
Therefore,
\begin{eqnarray*}
(y')^n-\frac{n}{(n-1)g'}(y')^{n-1}-n\left(\frac{n}{n-1}\right)^n\frac{\alpha_n g''}{(g')^{n+2}}=\bo\left(\frac{g''\delta^2}{(g')^{n+3}}+\frac{1}{(g')^{k+n}}\right)
\end{eqnarray*}
and Proposition \ref{proposition3.3} now gives
\begin{eqnarray*}
y'=\frac{n}{(n-1)g'}+\frac{n^2\alpha_n g''}{(n-1)(g')^3}+\bo\left(\frac{g''\delta^2}{(g')^4}+\frac{1}{(g')^{k+1}}\right).
\end{eqnarray*}
This completes the proof of the lemma. \qed
\begin{lemm}\label{lemma3.6}
Let $k\geq 3$, $\delta=k\log(g')$, $T_\delta=T_1-(n-1)\delta$ and $b\in\R$. Assume that $g$ satisfies
\begin{eqnarray}\label{eq3.35}
\displaystyle\limsup_{s\to\infty}(g(s)-\left(\frac{n-1}{n}\right)sg'(s))\geq b+1.
\end{eqnarray}
Then, there exists a $\gamma_0(k)>0$ such that  for all $\gamma\geq \gamma_0(k)$, we have
\begin{gather}
\label{eq3.36}\int_{\tilde{T}}^{T_\delta}e^{g(y(s))-s}{\rm d}s=\bo\left(\frac{(g')^{\frac{q}{q-1}}}{(g')^{k+n-1}}\right), \\
\label{eq3.37} \int_{\tilde{T}}^{T_\delta}g'(y(s))e^{g(y(s))-s}{\rm d}s=\bo\left(\frac{(g')^{\frac{q}{q-1}}}{(g')^{k+n-2}}\right)\nonumber
\end{gather}
where $q$ is given as in \eqref{eq1.3}.
\end{lemm}
\proof
Choose $\gamma_1$ such that Lemma \ref{lemma3.3} holds for $\gamma\geq \gamma_1$, $\delta=k\log(g')$ and 
\begin{eqnarray}\label{eq3.38}
g-\left(\frac{n-1}{n}\right)\gamma g'\geq b.
\end{eqnarray}
Let $\psi$ be as in \eqref{eq3.10}, then, for $\theta\geq s_0$, we have $\theta\to\psi(\theta)$ is  a convex  function. Hence from \eqref{eq3.20}, we have
\begin{eqnarray}\label{eq3.39}
g(y(t))-t\leq \psi(y(t))\leq\max\left\{\psi(y(\tilde{T})),\psi(y(T_\delta))\right\} \qquad \text{ for   } t \in [\tilde{T}, T_\delta].
\end{eqnarray}
\noindent \emph{Estimate for $\psi(y(\tilde{T}))$: }From \eqref{eq3.38},
\begin{eqnarray*}
\psi(y(\tilde{T}))&=& g(s_0)-\left(\frac{n-1}{n}\right)s_0g'-(g-\left(\frac{n-1}{n}\right)\gamma g')-(n-1)\log(\left(\frac{n-1}{n}\right)g')\\
&\leq& g(s_0)-\left(\frac{n-1}{n}\right)s_0g'-b-(n-1)\log(\left(\frac{n-1}{n}\right)g'),
\end{eqnarray*}
which implies that
\begin{eqnarray}\label{eq3.40}
e^{\psi(y(\tilde{T}))}=\bo\left(\frac{e^{-\left(\frac{n-1}{n}\right)s_0g'}}{(g')^{n-1}}\right).
\end{eqnarray}

\noindent \emph{Estimate for $\psi(y(T_{\de}))$:} From \eqref{eq3.22} and \eqref{eq3.23}
\begin{eqnarray*}
\psi(y(T_\delta))&=& g(y(T_\delta))-g+\left(\frac{n-1}{n}\right)(\gamma-y(T_\delta))g'-(n-1)\log(\left(\frac{n-1}{n}\right)g')\\
&=& -n\log(1+e^\delta)+\left(\frac{n-1}{n}\right)(\gamma-z(T_\delta))g'-(n-1)\log(\left(\frac{n-1}{n}\right)g')+\bo\left(\frac{\delta^2}{g'}\right)\\
&=&-n\log(1+e^\delta)+(n-1)\log(1+e^\delta)-(n-1)\log(\left(\frac{n-1}{n}\right)g')+ \bo\left(\frac{\delta^2}{g'}\right)\\
&=&-\delta-(n-1)\log(g')+\bo(1).
\end{eqnarray*}
Thus,
\begin{eqnarray}\label{eq3.41}
e^{\psi(y(T_\delta))}=\bo\left(e^{-\delta-(n-1)\log(g')}\right)=\bo\left(\frac{1}{(g')^{k+n-1}}\right).
\end{eqnarray}

\noindent \emph{Estimate for $\tilde{T}-T_{\de}$:} From \eqref{eq3.11} and \eqref{eq3.18}, it is easy to see that $T_0<\tilde{T}$. Now from \eqref{eq3.8} and \eqref{eq3.21}, we have for  for $\gamma\geq \gamma_0$,
\begin{eqnarray}\label{eq3.42}
0<T_\delta-\tilde{T}\leq T_\delta-T_0&=&g+(n-1)\log(\left(\frac{n-1}{n}\right)g')-(n-1)\delta\nonumber\\
&&-g-(n-1)\log(\left(\frac{n-1}{n}\right)g')+\left(\frac{n-1}{n}\right)\gamma g'+\bo\left(e^{-\frac{\gamma g'}{n}}\right)\nonumber\\
&\leq& \bo\left((g')^{\frac{q}{q-1}}\right),
\end{eqnarray}
where we have used $\gamma=\bo\left((g')^{\frac{1}{q-1}}\right)$.

Hence from \eqref{eq3.40}, \eqref{eq3.41} and \eqref{eq3.42}, we have for $\gamma$ large
\begin{eqnarray*}
\int_{\tilde{T}}^{T_\delta}e^{g(y(s))-s}{\rm d}s &\leq & (T_\delta-\tilde{T})\max\left\{e^{\psi(y(\tilde{T}))}, e^{\psi(y(T_\delta))}\right\}\\
&=&\bo\left(\frac{(g')^{\frac{q}{q-1}}}{(g')^{k+n-1}}\right)
\end{eqnarray*}
and
\begin{eqnarray*}
\int_{\tilde{T}}^{T_\delta} g'(y(s))e^{g(y(s))-s}{\rm d}s=\bo\left(g'\int_{\tilde{T}}^{T_\delta} e^{g(y(s))-s}{\rm d}s\right)=\bo\left(\frac{(g')^{\frac{q}{q-1}}}{(g')^{k+n-2}}\right).
\end{eqnarray*} 
This completes the proof of the lemma.\qed

Immediate consequence of the above lemma gives the following asymptotics: 
\begin{lemm}\label{lemma3.7}
Let $g$ satisfies \eqref{eq3.35} with $k\geq	5+\frac{q}{q-1}$, $\delta=k\log(g')$ and $T_\delta=T_1-(n-1)\delta$. Then there exists a $\gamma_0>0$ such that for $\gamma>\gamma_0$ and for all $t\in [\tilde{T},T_\delta]$, we have
\begin{eqnarray}\label{eq3.43}
y'(t)=\frac{n}{(n-1)g'}+\frac{n^2\alpha_n g''}{(n-1)(g')^3}+\bo\left(\frac{\delta^2 g''}{(g')^4}\right),
\end{eqnarray}
\begin{eqnarray}\label{eq3.44}
\tilde{T}&=&\left(\frac{n-1}{n}\right)s_0g'+(g-\left(\frac{n-1}{n}\right)\gamma g')+(n-1)\log(\left(\frac{n-1}{n}\right)g') +\nonumber\\
&&+\frac{(n-1)g''\alpha_n(\gamma-s_0)}{g'}+\bo\left(\frac{(\log(g'))^2}{g'}\right),
\end{eqnarray}
\begin{eqnarray}\label{eq3.45}
T(\gamma)&\geq& \lbr g-\left(\frac{n-1}{n}\right)\gamma g'\rbr + (n-1)\log\lbr\left(\frac{n-1}{n}\right)g'\rbr+\frac{(n-1)\alpha_n\gamma g''}{g'} + \nonumber \\
  && + \bo\left(\frac{(\log(g'))^2}{g'}\right).
\end{eqnarray}
Furthermore, setting $t_0\eqdef(n+3)\log(g')$, if $T(\gamma)\geq t_0+\bo(1)$ then
\begin{eqnarray}\label{eq3.46}
T(\gamma)=g-\left(\frac{n-1}{n}\right)\gamma g'+(n-1)\log(\left(\frac{n-1}{n}\right) g')&+&\frac{(n-1)\alpha_n\gamma g''}{g'}+\nonumber\\
&+& \bo\left(\frac{(\log(g'))^2}{g'}\right).
\end{eqnarray}
Also for all $t\in [\max\{t_0,T(\gamma)\}, \tilde{T}]$, we have
\begin{eqnarray}\label{eq3.47}
y'(t)=y'(\tilde{T})+\bo\left(\frac{g''\delta^2}{(g')^4}\right).
\end{eqnarray}
If $T(\gamma)< t_0+\bo(1)$, then
\begin{eqnarray}\label{eq3.48}
y(t_0)=\bo\left(\frac{\log(g')}{g'}\right).
\end{eqnarray}
\end{lemm}
\proof Since $k\geq5+\frac{q}{q-1}$, we have for $t\in [\tilde{T},T_\delta]$ and from using \eqref{eq3.36}, \eqref{eq3.34}
\begin{eqnarray*}
y'(t)^{n-1}&=& y'(T_\delta)^{n-1}+\int_t^{T_\delta}e^{g(y(s))-s}{\rm d}s\\
&=&y'(T_\delta)^{n-1}+\bo\left(\frac{1}{(g')^{4+n}}\right)\\
&=& y'(T_\delta)^{n-1}\left(1+\bo\left(\frac{1}{(g')^{5}}\right)\right).
\end{eqnarray*}
Now by 	making use of  estimate \eqref{eq3.34}, we get 
\begin{eqnarray*}
y'(t)&=& y'(T_\delta)\left(1+\bo\left(\frac{1}{(g')^5}\right)\right)\\
&=&\frac{n}{(n-1)g'}+\frac{n^2\alpha_n g''}{(n-1)(g')^3}+\bo\left(\frac{\delta^2 g''}{(g')^4}\right).
\end{eqnarray*}
This proves \eqref{eq3.43}. 

From the mean value theorem and \eqref{eq3.43}, there exists a $\xi\in [\tilde{T},T_\delta]$ such that
\begin{eqnarray*}
\tilde{T}&=& T_\delta+\frac{y(\tilde{T})-y(T_\delta)}{y'(\xi)}=T_\delta+\frac{s_0-z(T_\delta)+\bo\left(\frac{g''\delta^2}{(g')^3}\right)}{y'(\xi)}\\
&=& T_{\delta} + \left(\frac{n-1}{n}\right)g'\left[1-\frac{n\alpha_n g''}{(g')^2}+\bo\left(\frac{\delta^2 g''}{(g')^3}\right)\right]\times\left[s_0-\gamma+\frac{n}{g'}\log(1+e^\delta)+\bo\left(\frac{g''\delta^2}{(g')^3}\right)\right]\\
&=& T_1-(n-1)\delta+\left(\frac{n-1}{n}\right)g'(s_0-\gamma)+(n-1)\delta+(n-1)\log(1+e^{-\delta})\\
&&-\frac{\alpha_n(n-1)g''(s_0-\gamma)}{g'}
-\frac{n(n-1)\alpha_n g''}{(g')^2}\log(1+e^\delta)+ \bo\left(\frac{\delta^2\gamma g''}{(g')^2}\right)\\
&=&\left(\frac{n-1}{n}\right)s_0g'+(g-\left(\frac{n-1}{n}\right)\gamma g')
+(n-1)\log(\left(\frac{n-1}{n}\right)g') +\frac{\alpha_n(n-1)g''(\gamma-s_0)}{g'}\\
&&+\bo\left(\frac{(\log(g'))^2}{g'}\right). 
\end{eqnarray*}
This proves \eqref{eq3.44}.

For $t\in [T(\gamma),\tilde{T}]$, we have $y'(t)\geq y'(\tilde{T})$ and hence from the mean value theorem, there exist a $\xi\in [T(\gamma),\tilde{T}]$ such that 
\begin{eqnarray*}
\tilde{T}-T(\gamma)=\frac{y(\tilde{T})-y(T(\gamma))}{y'(\xi)}\leq \frac{s_0}{y'(\tilde{T})}.
\end{eqnarray*}
Hence from \eqref{eq3.43} and \eqref{eq3.44}
\begin{eqnarray}\label{eq3.49}
T(\gamma)&\geq &\tilde{T}-\frac{s_0}{y'(\tilde{T})}\nonumber\\
&=&\left(\frac{n-1}{n}\right)s_0g'+(g-\left(\frac{n-1}{n}\right)\gamma g')+(n-1)\log(\left(\frac{n-1}{n}\right) g')+\bo\left(\frac{(\log(g'))^2}{g'}\right)\nonumber\\
&&+\frac{\alpha_n(n-1)g''(\gamma-s_0)}{g'}-s_0\left(\frac{n-1}{n}\right)g'\left[1-\frac{n\alpha_n g''}{(g')^2}+\bo\left(\frac{\delta^2 g''}{(g')^3}\right)\right]\nonumber\\
&=&(g-\left(\frac{n-1}{n}\right)\gamma g')+(n-1)\log(\left(\frac{n-1}{n}\right)g')+\frac{\alpha_n(n-1)\gamma g''}{g'}\nonumber\\
&&+\bo\left(\frac{(\log(g'))^2}{g'}\right).
\end{eqnarray}
This proves \eqref{eq3.45}. 

Let $t\in [\max\{t_0, T(\gamma)\},\tilde{T}]$, then
\begin{eqnarray*}
y'(t)^{n-1}&=&y'(\tilde{T})^{n-1}+\int_t^{\tilde{T}}f(y(s))e^{-s}{\rm d}s\\
&=&y'(\tilde{T})^{n-1}+\bo(e^{-T(\gamma)})= y'(\tilde{T})^{n-1}+\bigo{\frac{1}{(g')^{n+3}}}.
\end{eqnarray*}
Hence
\begin{eqnarray}\label{eq3.50}
y'(t)=y'(\tilde{T})+\bo\left(\frac{g'' \de^2}{(g')^4}\right).
\end{eqnarray}
Suppose that $T(\gamma)\geq t_0+\bo(1)$, then there exists $\xi\in [T(\gamma), \tilde{T}]$ and from \eqref{eq3.49} and \eqref{eq3.50} we have
\begin{eqnarray*}
T(\gamma)&=&\tilde{T}+\frac{y(T(\gamma))-y(\tilde{T})}{y'(\xi)}\\
&=&\tilde{T}-\frac{s_0}{y'(\tilde{T})}(1+\bo\left(\frac{g'' \delta^2}{(g')^3}\right))\\
&=&(g-\left(\frac{n-1}{n}\right)\gamma g')+(n-1)\log(\left(\frac{n-1}{n}\right)g')+\frac{\alpha_n(n-1)\gamma g''}{g'}+ \bo\left(\frac{(\log(g'))^2}{g'}\right).
\end{eqnarray*}
This proves \eqref{eq3.46}.

Let $T(\gamma)<t_0=(n+3)\log(g')+ \bo(1)$. Then \eqref{eq3.50} holds for $t\in [t_0,\tilde{T}]$ and this proves \eqref{eq3.47}. 
From \eqref{eq3.47} and mean value theorem, there exists a $\xi\in [t_0,\tilde{T}] \subset [T(\gamma),\tilde{T}]$ such that
\begin{eqnarray*}
y(t_0)&=&y(\tilde{T})+y'(\xi)(t_0-\tilde{T})\\
&=& s_0+y'(\tilde{T})(t_0-\tilde{T})+O\left(\frac{(t_0-\tilde{T})g'' \de^2}{(g')^4}\right)\\
&=& s_0 + (t_0 - \tilde{T}) \lbr \frac{n}{(n-1)g'} + \frac{\al_n n^2 g''}{(n-1)(g')^3} + \bigo{\frac{g''\de^2}{(g')^4}} \rbr \\
&=& s_0+\frac{n}{(n-1)g'} \lbr 1+\bo\left(\frac{g''}{(g')^2}\right) \rbr \lbr t_0-\left(\frac{n-1}{n}\right)s_0 g'  \rbr[.]\\
&&\lbr[.] -(g-\left(\frac{n-1}{n}\right)\gamma g')-(n-1)\log(\left(\frac{n-1}{n}\right)g')+\bo(1)\rbr\\
&=&\frac{n}{(n-1)g'}\left[t_0-(g-\left(\frac{n-1}{n}\right)\gamma g')-(n-1)\log(\left(\frac{n-1}{n}\right)g')\right]+\bo\left(\frac{1}{g'}\right)\\
&\leq &\frac{n}{(n-1)g'}\left[ t_0-b-(n-1)\log(\left(\frac{n-1}{n}\right)g')\right] +\bo\left(\frac{1}{g'}\right)=\bo\left(\frac{\log(g')}{g'}\right).
\end{eqnarray*}
This proves \eqref{eq3.48} and hence the lemma.\qed
\begin{lemm}\label{lemma3.8}
Assume that there exists a $\beta>0$ such that  as $s\to 0$,
\begin{eqnarray*}
f(s)-f(0)=\bo(s^\beta).
\end{eqnarray*}
Then there exists a bounded function $A\,:\,\gamma\to A(\gamma)$ such that $A(\gamma)=0$ if $f(0)=0$ and for large $\gamma$
\begin{eqnarray}\label{eq3.51}
T(\gamma)&=&(g-\left(\frac{n-1}{n}\right)\gamma g')+(n-1)\log(\left(\frac{n-1}{n}\right)g')+\frac{(n-1)\alpha_n\gamma g''}{g'}+A(\gamma)\nonumber\\
&&+ \bo\left(\frac{(\log(g'))^2}{g'} +\frac{(\log(g'))^{\beta+1}}{(g')^\beta}e^{-(g-\left(\frac{n-1}{n}\right)\gamma g')}\right).
\end{eqnarray}
Furthermore,
\begin{itemize}
\item[(i)] Assume that $\displaystyle\lim_{\gamma\to\infty}(g-\left(\frac{n-1}{n}\right)\gamma g')=\infty$, then
\begin{eqnarray}\label{eq3.52}
y'(T(\gamma))=\frac{n}{(n-1)g'}+\frac{n^2\alpha_n g''}{(n-1)(g')^3}+ \bo\left(\frac{\delta^2 g''}{(g')^4}+\frac{e^{-(g-\left(\frac{n-1}{n}\right)\gamma g')}}{(g')^{n-1}}\right).
\end{eqnarray}
\item[(ii)] Assume that $\displaystyle\limsup_{\gamma\to\infty}\left(g-\left(\frac{n-1}{n}\right)\gamma g'\right)<\infty$.
Then,
\begin{eqnarray}\label{eq3.53}
y'(T(\gamma))&=&\frac{n}{(n-1)g'}+\frac{n^2\alpha_n g''}{(n-1)(g')^3}+\left[\left(1+\frac{f(0)e^{-T(\gamma)}}{y'(t_0)^{n-1}}\right)^{\frac{1}{n-1}}-1\right]y'(\tilde{T})\nonumber\\
&&+ \bo\left(\frac{(\log(g'))^\beta}{(g')^\beta}+\frac{g''(\log(g'))^2}{(g')^4}\right).
\end{eqnarray}
\end{itemize}
\end{lemm}
\proof Let $t_0=(n+3)\log(g')$. If $T(\gamma)\geq t_0$, then \eqref{eq3.51} and \eqref{eq3.52} follows from Lemma \ref{lemma3.7}. 

Hence assume that $T(\gamma)<t_0$. From \eqref{eq3.45} it is easy to see that $(g-\frac{(n-1)}{n}\gamma g')=\bo(\log(g'))$ which gives $\tilde{T}=\bo(g')$ and from \eqref{eq3.50}, we have
\begin{equation}\label{eq3.55}\begin{array}{lll}
t_0-\frac{y(t_0)}{y'(t_0)}&=& t_0-\frac{y(\tilde{T})+y'(\xi)(t_0-\tilde{T})}{y'(\tilde{T})+\bo\left(\frac{g'' \de^2}{(g')^4}\right)}\\
&=& t_0-\frac{s_0+y'(\tilde{T})(t_0-\tilde{T})+\bo\left(\frac{(t_0-\tilde{T})g''\de^2}{(g')^4}\right)}{y'(\tilde{T})(1+\bo(\frac{g'' \de^2}{(g')^3}))}\\
&=& t_0-\frac{s_0}{y'(\tilde{T})}-(t_0-\tilde{T})+\bo\left(\frac{g'' \de^2}{(g')^3}\right)\\
&=&\tilde{T}-\frac{s_0}{y'(\tilde{T})}+\bo\left(\frac{g''\de^2}{(g')^3}\right)\\
&=& (g-\left(\frac{n-1}{n}\right)\gamma g')+(n-1)\log(\left(\frac{n-1}{n}\right)g')+\frac{(n-1)\alpha\gamma g''}{g'}+\bo\left(\frac{(\log(g'))^2}{g'}\right).
\end{array}\end{equation}
From \eqref{eq3.45}, \eqref{eq3.43} and \eqref{eq3.50} we have for $t\in [T(\gamma),t_0]$, 
\begin{eqnarray}\label{eq3.56}
\frac{e^{-t}}{y'(t_0)^{n-1}}=\bo\left(e^{-(g-\frac{n-1}{n}\gamma g')}\right).
\end{eqnarray}
Hence from \eqref{eq3.48} and \eqref{eq3.56}, we have
\begin{eqnarray*}
y'(t)^{n-1}&=&y'(t_0)^{n-1}+\int_t^{t_0}f(y(s))e^{-s}{\rm d}s\\
&=& y'(t_0)^{n-1}+ f(0)e^{-t}+\bo\left(\int_t^{t_0}y(s)^\beta e^{-s}{\rm d}s+\frac{1}{(g')^{n+3}}\right)\\
&=& y'(t_0)^{n-1}+f(0) e^{-t}+\bo\left(\frac{(\log(g'))^\beta}{(g')^\beta}e^{-t}+\frac{1}{(g')^{n+3}}\right).
\end{eqnarray*}
Therefore, from \eqref{eq3.56}
\begin{eqnarray}\label{eq3.58}
\frac{y'(t)}{y'(t_0)}&=&\left(1+\frac{f(0)e^{-t}}{y'(t_0)^{n-1}}\right)^{\frac{1}{n-1}}+\bo\left(\frac{(\log(g'))^\beta}{(g')^\beta}\frac{e^{-t}}{y'(t_0)^{n-1}}+\frac{1}{(g')^3}\right)\nonumber\\
&=&\left(1+\frac{f(0)e^{-t}}{y'(t_0)^{n-1}}\right)^{\frac{1}{n-1}}+\bo\left(\frac{(\log(g'))^\beta}{(g')^\beta}e^{-(g-\frac{n-1}{n}\gamma g')}+\frac{1}{(g')^3}\right).
\end{eqnarray}
Hence from \eqref{eq3.55}
\begin{eqnarray}\label{eq3.57}
\frac{y(t_0)}{y'(t_0)}&=&\int_{T(\gamma)}^{t_0}\frac{y'(t)}{y'(t_0)}{\rm d}t\nonumber\\
&=& t_0-T(\gamma)+\int_{T(\gamma)}^{t_0}\left[(1+\frac{f(0)e^{-t}}{y'(t_0)})^{\frac{1}{n-1}}-1\right]{\rm d}t\nonumber\\
&&+\bo\left(\frac{(\log(g'))^{\beta+1}}{(g')^\beta}e^{-(g-\frac{n-1}{n}\gamma g')}+\frac{(\log(g'))}{(g')^3}\right).
\end{eqnarray}
Let 
\begin{eqnarray*}
A(\gamma)\eqdef\int_{T(\gamma)}^{t_0}\left[(1+\frac{f(0)e^{-t}}{y'(t_0)^{n-1}})^{\frac{1}{n-1}}-1\right]{\rm d}t.
\end{eqnarray*}
If $f(0)=0$, then $A(\gamma)=0$. Let $f(0)\neq 0$ and $\theta_0\eqdef -\log(f(0))+(n-1)\log(y'(t_0))$. Then,
\begin{eqnarray*}
A(\gamma)&=&\int_{T(\gamma)+\theta_0}^{t_0+\theta_0}\left[(1+e^{-t})^{\frac{1}{n-1}}-1\right]{\rm d}t\quad\mbox{and }\\
 t_0+\theta_0&=&(n+3)\log(g')-(n-1)\log(g')+\bo(1)\to\infty\mbox{ as }\gamma\to\infty.
\end{eqnarray*}
Also from \eqref{eq3.45}
\begin{eqnarray*}
T(\gamma)+\theta_0\geq (n-1)\log(g')-(n-1)\log(g')+\bo(1)=\bo(1).
\end{eqnarray*}
Hence $A(\gamma)$ is bounded as $\gamma\to\infty$. Therefore we have from \eqref{eq3.57} and \eqref{eq3.55}
\begin{eqnarray*}
T(\gamma)&=& t_0-\frac{y(t_0)}{y'(t_0)}+A(\gamma)+\bo\left(\frac{(\log(g'))^{\beta+1}}{(g')^\beta}e^{-(g-\left(\frac{n-1}{n}\right)\gamma g')}+\frac{(\log(g'))}{(g')^3}\right)\\
&=&(g-\left(\frac{n-1}{n}\right)\gamma g') +(n-1)\log(\left(\frac{n-1}{n}\right)g')+\frac{(n-1)\alpha\gamma g''}{g'}\\
&&+ A(\gamma)+ \bo\left(\frac{(\log(g'))^2}{g'}+\frac{(\log(g'))^{\beta+1}}{(g')^\beta} e^{-(g-\left(\frac{n-1}{n}\right)\gamma g')}\right).
\end{eqnarray*}
This proves \eqref{eq3.51}.  Let $g-\left(\frac{n-1}{n}\right)\gamma g')\to\infty$ as $\gamma\to\infty$. Then for $t\in [T(\gamma), t_0]$ and from \eqref{eq3.50}
\begin{eqnarray*}
y'(t)^{n-1}=y'(t_0)^{n-1}+\int_t^{t_0}f(y(s))e^{-s}{\rm d}s=y'(t_0)^{n-1}+\bo\left(e^{-T(\gamma)}\right).
\end{eqnarray*}
Thus,
\begin{eqnarray*}
y'(t)&=&y'(t_0)+\bo\left(\frac{e^{-(g-\left(\frac{n-1}{n}\right)\gamma g')}}{(g')^{n-1}}\right)\\
&=& y'(\tilde{T})+\bo\left(\frac{1}{(g')^3}+\frac{e^{-(g-\left(\frac{n-1}{n}\right)\gamma g')}}{(g')^{n-1}}\right).
\end{eqnarray*}
This proves \eqref{eq3.52}. From \eqref{eq3.58} we have 
\begin{eqnarray*}
y'(T(\gamma))&=&y'(t_0)+y'(t_0)\left[(1+\frac{f(0)e^{-T(\gamma)}}{y'(t_0)^{n-1}})^{\frac{1}{n-1}}-1\right]\\
&&+\bo\left(\frac{(\log(g'))^\beta}{(g')^\beta}e^{-(g-\left(\frac{n-1}{n}\right)\gamma g')}+\frac{1}{(g')^3}\right)
\end{eqnarray*}
and \eqref{eq3.53} follows from \eqref{eq3.50}. This proves the lemma.\qed\\
\begin{proof}[Proof of Theorem \ref{theo2.5}]: This follows from Lemma \ref{lemma3.8}.\end{proof}

\section{Proof of Theorem \ref{theo2.6}.}\label{section4}
First we look at the linearization of $z$ which plays an important role in proving the theorem. Let $\theta\in\R$ and define $z_\theta(t)\eqdef z(t-\theta g')+\theta$, then $z_\theta$ satisfies
\begin{eqnarray}\label{eq4.1}
\left\{\begin{array}{ll}
&-((z'_\theta)^{n-1})'=e^{g+g'(z_\theta-\gamma)-t}\\
&z_\theta(\infty)=\gamma+\theta,\quad z'_\theta(\infty)=0.
\end{array}\right.
\end{eqnarray}
Let $V_2(t)=\frac{\partial z_\theta}{\partial \theta}\vert_{\theta=0}$. Then $V_2$ satisfies
\begin{eqnarray}\label{eq4.2}
V_2(t)=1-g'z'(t).
\end{eqnarray}
Differentiating \eqref{eq4.1} with respect to $\theta$ and evaluate at $\theta=0$  to obtain
\begin{eqnarray}\label{eq4.3}
\left\{\begin{array}{ll}
&-((z')^{n-2}V_2')'=\frac{g'}{n-1}V_2e^{g+g'(z-\gamma)-t}\\
&V_2(\infty)=1,\quad V_2'(\infty)=0.
\end{array}\right.
\end{eqnarray}
Moreover from \eqref{eq4.2}, $V_2$ satisfies
\begin{equation}\label{eq4.4}\begin{array}{lll}
\displaystyle V_2(t)&=&\frac{1-\frac{g'}{n}e^{\frac{g-t}{n-1}}}{1+\left(\frac{n-1}{n}\right)g'e^{\frac{g-t}{n-1}}}
=\frac{1-\frac{1}{n-1}e^{\frac{T_1-t}{n-1}}}{1+e^{\frac{T_1-t}{n-1}}}\\
&=&-\frac{1}{n-1}+\left(\frac{n}{n-1}\right)\frac{1}{1+e^{\frac{T_1-t}{n-1}}},
\end{array}\end{equation}
\begin{eqnarray}\label{eq4.5}
V_2'(t)=\frac{n}{(n-1)^2}\frac{e^{\frac{T_1-t}{n-1}}}{(1+e^{\frac{T_1-t}{n-1}})^2},
\end{eqnarray}
\begin{eqnarray}\label{eq4.6}
V''_2(t)=\frac{n}{(n-1)^3}\frac{e^{\frac{T_1-t}{n-1}}}{(1+e^{\frac{T_1-t}{n-1}})^2}\left[1-\frac{2}{1+e^{\frac{T_1-t}{n-1}}}\right].
\end{eqnarray}
Clearly $V'_2>0$ and $V_2$ has exactly one zero at $t=S_0$ given by 
\begin{eqnarray}\label{eq4.7}
S_0=g+(n-1)\log\left(\frac{g'}{n}\right)=T_1-(n-1)\log(n-1).
\end{eqnarray}
It is easy to see that 
\begin{eqnarray}\label{eq4.8}
V_2(-\infty)=-\frac{1}{n-1}.
\end{eqnarray}
Let $V_1$ be the linearization of $y(t)=y(t,\gamma)$ as in section \ref{section2} (see \eqref{eq2.13}) and set 
\begin{gather}
\label{eq4.9} \delta=\bo(\log(g')), \quad T_\delta=T_1-(n-1)\delta,\nonumber\\
\label{eq4.10} \rho_1(t)\eqdef\frac{1}{(n-1)} f'(y(t)) = \frac{1}{n-1} e^{g(y(t))} g'(y(t)),\\
\label{eq4.11}\rho_2(t)\eqdef \frac{g'}{(n-1)}e^{g+g'(z(t)-\gamma)}.\nonumber
\end{gather}
Then, $V_1$ and $V_2$ satisfy respectively:
\begin{eqnarray}\label{eq4.12}
-((y')^{n-2}V_1')'=\rho_1V_1 e^{-t},
\end{eqnarray}
\begin{eqnarray}\label{eq4.13}
-((z')^{n-2}V_2')'=\rho_2 V_2e^{-t}.
\end{eqnarray}
For $T(\gamma)\leq t<\eta$, integrating \eqref{eq4.12} and \eqref{eq4.13} gives
\begin{eqnarray}\label{eq4.14}
V_1'(t)=\left(\frac{y'(\eta)}{y'(t)}\right)^{n-2}V'_1(\eta)+\frac{1}{y'(t)^{n-2}}\int_t^\eta \rho_1 V_1e^{-s}{\rm d}s.
\end{eqnarray}
Thus,
\begin{eqnarray}\label{eq4.15}
V_1(t)=V_1(\eta)-V_1'(\eta)\int_t^\eta \left(\frac{y'(\eta)}{y'(s)}\right)^{n-2}{\rm d}s-\int_t^\eta\rho_1 V_1 e^{-s}\lbr \int_t^s\frac{{\rm d}\theta}{y'(\theta)^{n-2}}\rbr{\rm d}s.
\end{eqnarray}
Similarly,
\begin{eqnarray}\label{eq4.16}
V_2'(t)=\left(\frac{z'(\eta)}{z'(t)}\right)^{n-2}V'_2(\eta)+\frac{1}{z'(t)^{n-2}}\int_t^\eta \rho_2 V_2e^{-s}{\rm d}s
\end{eqnarray}
and
\begin{eqnarray}\label{eq4.17}
V_2(t)=V_2(\eta)-V_2'(\eta)\int_t^\eta \left(\frac{z'(\eta)}{z'(s)}\right)^{n-2}{\rm d}s-\int_t^\eta\rho_2 V_2 e^{-s}\lbr\int_t^s\frac{{\rm d}\theta}{z'(\theta)^{n-2}}\rbr{\rm d}s.
\end{eqnarray}
Let 
\begin{eqnarray}\label{eq4.18}
\rho_3(t)&\eqdef &\left((y'(t)^{n-2}-z'(t)^{n-2})V'_2(t)\right)'\nonumber\\
&=&(n-2)(y'(t)^{n-3}y''(t)-z'(t)^{n-3}z''(t))V_2'+(y'(t)^{n-2}-z'(t)^{n-2})V''_2.
\end{eqnarray}
Then we have
\begin{eqnarray*}
\left((y')^{n-2}(V_1' V_2-V_1V_2')\right)'&=& ((y')^{n-2}V_1')'V_2-((y')^{n-2}V_2')'V_1\\
&=&-\rho_1 V_1V_2 e^{-t}+\rho_2 V_1V_2 e^{-t}-\rho_3(t)V_1.
\end{eqnarray*}
Therefore,
\begin{eqnarray}\label{eq4.19}
V_1'(t)V_2(t)-V_1(t)V'_2(t)&=& \left(\frac{y'(\eta)}{y'(t)}\right)^{n-2}(V_1'(\eta)V_2(\eta)-V_1(\eta)V'_2(\eta)) +\nonumber\\
&&+\frac{1}{y'(t)^{n-2}}\int_t^\eta(\rho_1-\rho_2)V_1V_2 e^{-s}{\rm d}s+\frac{1}{y'(t)^{n-2}}\int_t^\eta \rho_3V_1(s){\rm d}s.
\end{eqnarray}
In all the future lemmas, the statement deals for  large $\gamma$; that is there exists $\gamma_0>0$ such that for $\gamma>\gamma_0$, the lemma is true. Note that $\ga_0$ may be different for different lemmas.  With abuse of notation, we delete the statement $\gamma>\gamma_0$ or for large $\gamma$.

We need the following estimates proved in Section \ref{section3} (see \eqref{eq3.23} and \eqref{eq3.25}) and some more estimates. Let $t\geq T_\delta$, then 
\begin{eqnarray}\label{eq4.20}
g(y(t))=g+g'(z(t)-\gamma)+\bo\left(\frac{g''}{(g')^2}(\log(1+e^\delta))^2\right),
\end{eqnarray}
\begin{eqnarray}\label{eq4.21}
g'(y(t))=g'-\frac{ng''}{g'}\log(1+e^\delta)+\bo\left(\frac{g''}{(g')^2}(\log(1+e^\delta))^2\right)
\end{eqnarray}
\begin{eqnarray}\label{eq4.22}
g''(y(t))&=&g''+\bo\left(\frac{g'''}{g'}\log(1+e^\delta)\right)
=g''\left[1+\bo\left(\frac{1}{\gamma g'}\log(1+ e^\delta)\right)\right]\nonumber\\
&=& g''\left[1+\bo\left(\frac{\log(1+e^\delta)}{(g')^{1+\frac{1}{q-1}}}\right)\right]
=g''\left[1+\bo\left(\frac{\log(1+e^\delta)}{(g')^n}\right)\right].
\end{eqnarray}
Since $g'\sim \gamma^{q-1}$, $1<q\leq \frac{n}{n-1}$, $\frac{g''}{(g')^2}\sim\left(\frac{\gamma g''}{g'}\right)\frac{1}{\gamma g'}\sim \frac{1}{(g')^{\frac{q}{q-1}}}$ and $n\leq\frac{q}{q-1}$. Hence,
\begin{eqnarray}\label{eq4.23}
\frac{g''(y(t))}{g'(y(t))}&=&\frac{g''}{g'}\left[1+\bo\left(\frac{\log(1+e^\delta)}{(g')^n}\right)\right]\left[1+\frac{n g''}{(g')^2}\log(1+e^\delta)+\bo\left(\frac{g''}{(g')^3}(\log(1+e^\delta))^2\right)\right]\nonumber\\
&=&\frac{g''}{g'}\left[1+\bo\left(\frac{g''\log(1+e^\delta)}{(g')^2}\right)\right]
\end{eqnarray}
and
\begin{eqnarray}\label{eq4.24}
\frac{g'''(y(t))}{g'(y(t))}=\bo\left(\frac{g'''}{g'}\right)=O\left(\frac{\gamma g'''}{g''}\frac{g''}{\gamma g'}\right)=\bo\left(\frac{g''}{(g')^n}\right).
\end{eqnarray}
Therefore,
\begin{eqnarray*}
\rho_1(t)&=&\frac{1}{n-1}g'(y(t))e^{g(y(t))}\\
&=&\displaystyle\frac{g'}{(n-1)}\left(1+\bo\left(\frac{g''}{(g')^2}\log(1+e^\delta)\right)\right)e^{g+g'(z(t)-\gamma)+\bo\left(\frac{g''}{(g')^2}(\log(1+e^\delta))^2\right)},
\end{eqnarray*}
and
\begin{eqnarray}\label{eq4.25}
\displaystyle\rho_1(t)=\rho_2(t)e^{\bo\left(\frac{g''}{(g')^2}\log(1+e^\delta)^2\right)}.
\end{eqnarray}
Let $T_\delta\leq t_1<t_2<\infty$, then from \eqref{eq3.13} and by integration by parts
\begin{eqnarray}\label{eq4.26}
0<\int_{t_1}^{t_2}\rho_2 e^{-s}\left(\int_{t_1}^s\frac{{\rm d}\theta}{(z'(\theta))^{n-2}}\right){\rm d}s&=&\frac{g'}{(n-1)}\int_{t_1}^{t_2}e^{g+g'(z-\gamma)-s}\left(\int_{t_1}^s\frac{{\rm d}\theta}{(z'(\theta))^{n-2}}\right){\rm d}s\nonumber \\
&=&\frac{g'}{n-1}\left[z(t_2)-z(t_1)-z'(t_2)\int_{t_1}^{t_2}\left(\frac{z'(t_2)}{z'(t)}\right)^{n-2}{\rm d}t\right]\nonumber\\
&\leq &\frac{g'}{n-1}\vert z(t_2)-z(t_1)\vert\leq \frac{n}{(n-1)^2}\vert t_2-t_1\vert.
\end{eqnarray}
Setting $t_2=\infty$, we get
\begin{eqnarray}\label{eq4.27}
0<\int_t^{\infty}(\rho_2 e^{-s})\left(\int_t^s\frac{{\rm d}\theta}{(z'(\theta))^{n-2}}\right){\rm d}s&\leq&\frac{g'}{n-1}(\gamma-z(t))\nonumber\\
&\leq &\frac{n}{n-1}\log(1+e^{\frac{T_1-t}{n-1}})
\end{eqnarray}
and
\begin{eqnarray}\label{eq4.28}
\frac{1}{z'(t_1)^{n-2}}\int_{t_1}^{t_2}\rho_2 e^{-s}{\rm d}s=\frac{g'}{(n-1)}\left(z'(t_1)-\left(\frac{z'(t_2)}{z'(t_1)}\right)^{n-2}z'(t_2)\right)\leq\frac{n}{(n-1)^2}.
\end{eqnarray}
\underline{{\it Idea of the proof:}}
Basic idea is
\begin{itemize}
\item[(1)] Let $S<S_1<\infty$ be the first turning point  and first zero of $V_1$ before infinity. Then in the first few lemmas, we show that $S_1$ exists. Then estimates about $\vert V_1(t)-V_2(t)\vert+\vert V_1'(t)-V'_2(t)\vert$ for $t\in [S,\infty)$ are established.
\item[(2)] Since $V_2(t)\to -\frac{1}{n-1}$ as $t\to -\infty$, hence for $T(\gamma)<t<S$, $V_1$ and $V_2$ depart. Therefore we need a better estimate of $S$ and this is achieved by using the identity \eqref{eq4.58} in Lemma \ref{lemma4.6}. This identity plays an important role in proving the crucial estimate \eqref{eq4.63} in Lemma \ref{lemma4.8} which yields the necessary result.
\end{itemize}
$$
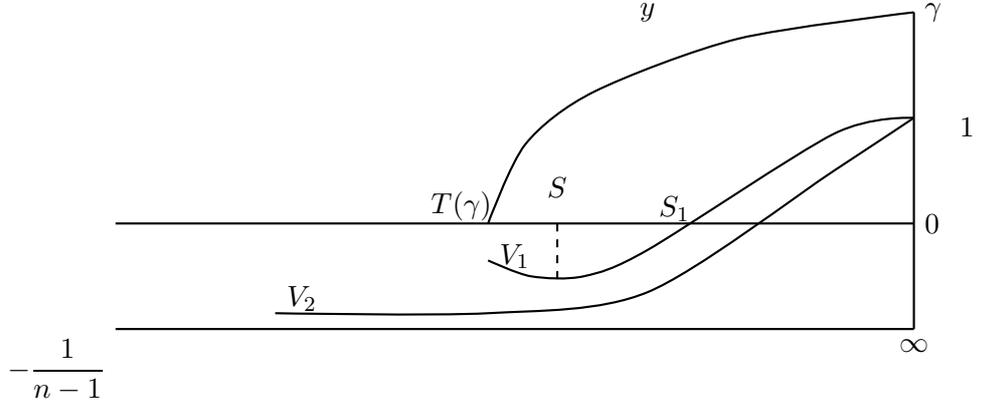
\begin{tikzpicture}[scale=0.7]
\draw[-] [thick] (-5,0) node[below left]{$-\frac{1}{n-1}$} -- (10,0) node[below]{$\infty$};
\draw[-] [thick] (-5,2) -- (10,2) node[right]{$0$};
\draw[-] [thick] (10,0)--(10,6) node[right]{$\gamma$};
\draw [thick, color=black] plot [smooth] coordinates {(2,2)(2.7,3.5) (4,4.5) (6.7,5.5) (10,6)};
\draw [thick, color=black] plot [smooth] coordinates {(10,4) (8.5,3.7) (5,1.5)  (3.8, 1) (2.8,1) (2,1.3) };
\draw [thick, color=black] plot [smooth] coordinates {(10,4) (8.5,3) (5,0.7) (2,0.3) (-2,0.3)  };
\draw[dashed] [thick] (3.3,2)--(3.3,0.9);
\node (P) at (1.5,2.8) [below]{$T(\gamma)$};
\node (P) at (2.5,1.8) [below]{$V_1$};
\node (P) at (-1.5,1) [below]{$V_2$};
\node (P) at (11,4.2) [below]{$1$};
\node (P) at (3.3,2.3) [above]{$S$};
\node (P) at (5,6) {$y$};
\node (P) at (5.5,2.7) [below]{$S_1$};
\end{tikzpicture}
 $$
\begin{lemm}\label{lemma4.1}
Let $T_\delta\leq t\leq \eta\leq \infty$, $A=\bo\left(\frac{g''}{(g')^2}(\log(1+e^\delta))^2\right)$ and
\begin{eqnarray}\label{eq4.29}
M_\eta(t)=\max\left\{\vert V_1(s)-V_2(s)\vert;\, t\leq s\leq \eta\right\}.
\end{eqnarray}
Then for $\gamma$ large,
\begin{eqnarray}\label{eq4.30}
M_\eta(t)&\leq& \vert V_1(\eta)-V_2(\eta)\vert+2\vert\eta-t\vert\vert V'_1(\eta)-V'_2(\eta)\vert + \frac{2n}{n-1}\log(1+e^\delta)M_\eta(t)\nonumber\\
&+&2A\left[\vert V_2'(\eta)\vert(\eta-t)+\log(1+e^{\frac{T_1-t}{n-1}})\right],
\end{eqnarray}
\begin{eqnarray}\label{eq4.31}
M_\infty(t)\leq 4\log(1+e^{\frac{T_1-t}{n-1}})[A+M_\infty(t)]\mbox{ and }
\end{eqnarray}
\begin{eqnarray}\label{eq4.32}
\vert V'_1(t)-V'_2(t)\vert\leq 2\vert V_1'(\eta)-V'_2(\eta)\vert+2A(\vert V'_2(\eta)\vert+V_2(\eta)) + 4M_\eta(t).
\end{eqnarray}
\end{lemm}
\proof From \eqref{eq3.24}, \eqref{eq4.25} and \eqref{eq4.17}, we have for $T_\delta\leq t\leq s\leq \eta\leq \infty$,
\begin{eqnarray*}
V_1(s)&=&V_1(\eta)-V_1'(\eta)\int_s^\eta\left(\frac{y'(\eta)}{y'(\theta)}\right)^{n-2}{\rm d}\theta-\int_s^{\eta}\rho_1 V_1 e^{-\theta}\left(\int_s^\theta\frac{{\rm d}x}{y'(x)^{n-2}}\right){\rm d}\theta\\
&=&V_1(\eta)-V'_1(\eta)e^A\int_s^\eta\left(\frac{z'(\eta)}{z'(\theta)}\right)^{n-2}{\rm d}\theta-e^A\int_s^\eta\rho_2 V_1 e^{-\theta}(\int_s^\theta\frac{{\rm d}x}{z'(x)^{n-2}}){\rm d}\theta\\
&=&V_1(\eta)-(V'_1(\eta)-V'_2(\eta))e^A\int_s^\eta\left(\frac{z'(\eta)}{z'(\theta)}\right)^{n-2}{\rm d}\theta-V_2'(\eta)e^A\int_s^\eta\left(\frac{z'(\eta)}{z'(\theta)}\right)^{n-2}{\rm d}\theta\\
&&-e^A\int_s^\eta\rho_2 V_2 e^{-\theta}(\int_s^\theta\frac{{\rm d}x}{z'(x)^{n-2}}){\rm d}\theta+ e^A\int_s^\eta \rho_2(V_2-V_1)e^{-\theta}(\int_s^\theta \frac{{\rm d}x}{z'(x)^{n-2}}){\rm d}\theta.
\end{eqnarray*}
Hence
\begin{equation}\label{eq4.33}
\begin{array}{lll}
V_1(s)-V_2(s)&=&(V_1(\eta)-V_2(\eta))-(V_1'(\eta)-V'_2(\eta))e^A\int_s^\eta\left(\frac{z'(\eta)}{z'(\theta)}\right)^{n-2}{\rm d}\theta\\
&&+(e^A-1)\left[-V_2'(\eta)\int_s^\eta\left(\frac{z'(\eta)}{z'(\theta)}\right)^{n-2}{\rm d}\theta-\int_s^\eta\rho_2V_2 e^{-\theta}(\int_s^\theta\frac{{\rm d}x}{(z'(x))^{n-2}}){\rm d}\theta\right]\\
&&+e^A\int_s^\eta\rho_2(V_2-V_1)e^{-\theta}(\int_s^\theta\frac{{\rm d}x}{z'(x)^{n-2}}){\rm d}\theta.
\end{array}
\end{equation}

Since $z'(\eta)\leq z'(\theta)$ for $\theta\leq \eta$ and $0\leq V_2\leq 1$, we have by using \eqref{eq4.27} for  $\eta\neq\infty$
\begin{eqnarray}\label{eq4.33bis}
\vert V_1(s)-V_2(s)\vert&\leq &\vert V_1(\eta)-V_2(\eta)\vert+e^A(\eta-s)\vert V'_1(\eta)-V'_2(\eta)\vert\nonumber\\
&&+(e^A-1)\left(\vert V'_2(\eta)\vert(\eta-s)+\left(\frac{n}{n-1}\right)\log(1+e^{\frac{T_1-s}{n-1}})\right)\nonumber\\
&&+\left(\frac{n}{n-1}\right)e^A\log(1+e^{\frac{T_1-s}{n-1}})M_\eta(s).
\end{eqnarray}
If $\eta=\infty$, using \eqref{eq2.4} and \eqref{eq2.13}, we get
\begin{eqnarray*}
\vert V_1(s)-V_2(s)\vert\leq (e^A-1)\frac{n}{n-1}\log\left(1+e^{\frac{T_1-s}{n-1}}\right)+\frac{n}{n-1}e^A\log\left(1+e^{\frac{T_1-s}{n-1}}\right)M_\infty(t).
\end{eqnarray*}
For $\gamma$ large, we have $e^A\leq 1+2A$ and $e^A\leq 2$, hence by taking the maximum of $\vert V_1-V_2\vert$ in $[t,\eta]$ we obtain if $\eta<\infty$
\begin{eqnarray*}
M_\eta(t)&\leq &\vert V_1(\eta)-V_2(\eta)\vert+2\vert t-\eta\vert\vert V_1'(\eta)-V'_2(\eta)\vert\\
&&+2A\left(\vert V_2'(\eta)\vert \eta-t\vert+\frac{n}{n-1}\log\left(1+e^{\frac{T_1-t}{n-1}}\right)\right)\\
&&+4\log\left(1+e^{\frac{T_1-t}{n-1}}\right)M_\eta(t),
\end{eqnarray*}
and if $\eta=\infty$, the estimate
\begin{eqnarray*}
M_\infty(t)\leq 4\log(1+e^{\frac{T_1-t}{n-1}})(A+M_\infty(t)).
\end{eqnarray*}
This proves \eqref{eq4.30} and \eqref{eq4.31}. Next,
\begin{eqnarray*}
V'_1(t)&=&\left(\frac{y'(\eta)}{y'(t)}\right)^{n-2} V'_1(\eta)+ \frac{1}{y'(t)^{n-2}}\int_t^\eta\rho_1 V_1 e^{-s}{\rm d}s\\
&=&  e^A V_1'(\eta)\left(\frac{z'(\eta)}{z'(t)}\right)^{n-2}+\frac{e^A}{z'(t)^{n-2}}\int_t^\eta\rho_2(V_1-V_2)e^{-s}{\rm d}s
+\frac{e^A}{z'(t)^{n-2}}\int_t^\eta\rho_2V_2 e^{-s}{\rm d}s\\
&=&e^A\left(V'_2(\eta)\left(\frac{z'(\eta)}{z'(t)}\right)^{n-2}+\frac{1}{z'(t)^{n-2}}\int_t^\eta\rho_2 V_2e^{-s}{\rm d}s\right)\\
&&+ e^A(V_1'(\eta)-V'_2(\eta))\left(\frac{z'(\eta)}{z'(t)}\right)^{n-2}
+\frac{e^A}{z'(t)^{n-2}}\int_t^\eta \rho_2(V_1-V_2)e^{-s}{\rm d}s.
\end{eqnarray*}
Hence from \eqref{eq4.16}, \eqref{eq4.28}, we have 
\begin{eqnarray*}
\vert V_1'(t)-V'_2(t)\vert\leq 2\vert V_1'(\eta)-V'_2(\eta)\vert +2A(\vert V_2'(\eta)\vert+V_2(\eta))+4M_\eta(t).
\end{eqnarray*}
This proves the lemma.\qed

In the next few lemmas we estimate the first zero $S_1$ and the first turning point $S$ defined by
\begin{eqnarray}\label{eq4.34}
S_1\eqdef\inf\{ t\,:\, V_1(s)>0, \quad \forall s\in (t,\infty)\},
\end{eqnarray}
\begin{eqnarray}\label{eq4.35}
S\eqdef\inf\{ t\,:\, V_1'(s)>0, \quad \forall s\in (t,\infty)\},
\end{eqnarray}
\begin{lemm}\label{lemma4.2}
Let $S_0$ be the zero of $V_2$ as defined in \eqref{eq4.7}. Then there exists a $k_1>0$, $C_3>0$ and $\gamma_0>0$ such that for all $\gamma\geq \gamma_0$, we have
\begin{eqnarray}\label{eq4.36}
\vert V_1(t)-V_2(t)\vert+\vert V'_1(t)-V'_2(t)\vert=\bo\left(\frac{g''}{(g')^4}\right)\quad\forall t\geq T_1+2(n-1)\log(g'),
\end{eqnarray}
\begin{eqnarray}\label{eq4.36b}
S_1-S_0=\bo\left(\frac{g''}{(g')^2}\right),
\end{eqnarray}
\begin{eqnarray}\label{eq4.37}
\vert V_1(t)-V_2(t)\vert+\vert V'_1(t)-V'_2(t)\vert=\bo\left(\frac{g''}{(g')^2}\right)\quad\forall t\geq S_1-(n-1)k_1,
\end{eqnarray}
\begin{eqnarray}\label{eq4.38}
V_2(S_1-(n-1)k_1)\leq -C_3.
\end{eqnarray}
\end{lemm}

\proof
Let $\delta=-2\log(g')$, then $A=\bo\left(\frac{g''}{(g')^2}(\log(1+e^\delta))^2\right)=\bo\left(\frac{g''}{(g')^4}\right)$ and hence for $t\geq T_\delta=T_1-(n-1)\delta$ and using Lemma \ref{lemma4.1}, we have
\begin{eqnarray*}
\vert V_1(t)-V_2(t)\vert=\bo\left(\frac{g''}{(g')^4}\right).
\end{eqnarray*}
Let $\eta=\infty$ in \eqref{eq4.32} and $t\geq T_1-(n-1)\delta$, we have
\begin{eqnarray*}
\vert V'_1(t)-V'_2(t)\vert= 2A+M_\infty(t)=\bo\left(\frac{g''}{(g')^4}\right).
\end{eqnarray*}
This proves \eqref{eq4.36}. 
Let $k_0>0$ be large such that $4\log(1+e^{-k_0})<\frac{1}{2}$ and let $-\delta=k_0$, then $A=\bo\left(\frac{g''}{(g')^2}\right)$.
Then from \eqref{eq4.31} and for $t\geq T_1+(n-1)k_0$ 
\begin{eqnarray*}
\frac{M_\infty(t)}{2}\leq 4\log(1+e^{\frac{T_1-t}{n-1}}))A. 
\end{eqnarray*}
Hence $M_\infty(t)=\bo(A)=\bo\left(\frac{g''}{(g')^2}\right)$. From \eqref{eq4.32} taking $\eta=\infty$ we obtain for $t\geq T_1+(n-1)k_0$ the estimate $\vert V'_1(t)-V'_2(t)\vert=\bo(A)$. 
Hence for $t\geq T_1+(n-1)k_0$, we have 
\begin{eqnarray}\label{eq4.39}
\vert V_1(t)-V_2(t)\vert +\vert V'_1(t)-V'_2(t)\vert=\bo\left(\frac{g''}{(g')^2}\right).
\end{eqnarray}

\begin{clm}
\label{clm1}
$S_1<T_1+(n-1)k_0$.
\end{clm}
\begin{proof}[Proof of Claim 1:]Since $T_1+(n-1)k_0=g'+(n-1)\log\left(\frac{g'(n-1)}{n}\right)+(n-1)k_0=S_0+(n-1)k_0+(n-1)\log(n-1)$, hence if $S_1\geq T_1+(n-1)k_0$, then $S_1>S_0$ and $S_0-S_1<-(n-1)(k_0+\log(n-1))\leq -(n-1)k_0$. This along with \eqref{eq4.39}, \eqref{eq4.4} gives
\begin{eqnarray*}
\frac{1-e^{\frac{S_0-S_1}{n-1}}}{1+e^{\frac{T_1-S_1}{n-1}}}=V_2(S_1)=V_2(S_1)-V_1(S_1)=\bo\left(\frac{g''}{(g')^2}\right).
\end{eqnarray*}
Thus,
\begin{eqnarray*}
1=e^{\frac{S_0-S_1}{n-1}}+\bo\left((1+e^{\frac{T_1-S_1}{n-1}})\frac{g''}{(g')^2}\right)\leq e^{-k_0}+\bo\left((1+e^{-k_0})\frac{g''}{(g')^2}\right)<1
\end{eqnarray*}
which is a contradiction and the claim follows.
\end{proof}

\begin{clm}\label{clm2} There exists a $C_1>0$ such that 
\begin{eqnarray}\label{eq4.40}
S_1\geq T_1+(n-1)k_0-C_1.
\end{eqnarray}\end{clm}
\begin{proof}[Proof of Claim 2:]
Let $S_2=T_1+(n-1)k_0$, then from \eqref{compar-y-z}, for $T(\gamma)\leq t\leq S_2$, we have $y'(S_2)\geq z'(S_2)$ and $y'(t)\leq y'(T(\gamma))$. Hence from \eqref{eq3.34}, there exists a $C_0>0$ such that for $\gamma$ large
\begin{eqnarray*}
\frac{y'(S_2)}{y'(t)}\geq \frac{z'(S_2)}{y'(T(\gamma))}=\frac{\frac{n}{g'(n-1)}(\frac{e^{-k_0}}{1+e^{-k_0}})}{(1+C_0)(\frac{n}{g'(n-1)}+\bo\left(\frac{1}{(g')^3}\right))}\geq C_2
\end{eqnarray*}
where $C_2>0$ independent of $\gamma$ ($C_2$ depends on $\ga_0$ but is independent of $\ga > \ga_0$). Hence from Claim \ref{clm1}, \eqref{eq2.13} and \eqref{eq4.39}
\begin{eqnarray*}
V_2(S_2)+O\left(\frac{g''}{(g')^2}\right)&=&V_1(S_2)=V_1(S_2)-V_1(S_1)\\
&\geq & V_1'(S_2)\int_{S_1}^{S_2}\left(\frac{y'(S_2)}{y'(t)}\right)^{n-2}{\rm d}t\\
&\geq & C_2 V_1'(S_2)(S_2-S_1)=C_2(V_2'(S_2)+\bo\left(\frac{g''}{(g')^2}\right))(S_2-S_1).
\end{eqnarray*}
Since $V_2(S_2)=\frac{1-\frac{e^{-k_0}}{n-1}}{1+e^{-k_0}}$, $V_2'(S_2)=\frac{n}{(n-1)^2}\frac{e^{-k_0}}{1+e^{-k_0}}$, hence there exists a $C_1>0$ independent of $\gamma$ such that $S_2-S_1\leq C_1$. This proves the claim.
\end{proof}

\begin{clm}\label{clm3} For $t\geq S_1$
\begin{eqnarray}\label{eq4.41}
S_1=S_0+\bo\left(\frac{g''}{(g')^2}\right),
\end{eqnarray}
\begin{eqnarray}
\vert V_1(t)-V_2(t)\vert+\vert V'_1(t)-V'_2(t)\vert=\bo\left(\frac{g''}{(g')^2}\right).
\end{eqnarray}
\end{clm}
\begin{proof}[Proof of Claim 3:]  
From Claims \ref{clm1} and \ref{clm2},we see that  $S_1=T_1+(n-1)\delta_0$ and  $\delta_0=\bo(1)$. Hence for $t\in [S_1,\infty)$,we have $0\leq V_1(t)\leq 1$, $S_2-S_1=\bo(1)$ and from \eqref{eq4.25} $\rho_1=\rho_2 e^{\bo\left(\frac{g''}{(g')^2}\right)}$. Now, using \eqref{eq4.28} and \eqref{eq3.24} we have
\begin{eqnarray}\label{eq4.43}
\frac{1}{y'(S_1)^{n-2}}\int_{S_1}^{S_2}(\rho_1-\rho_2)V_1V_2e^{-t}{\rm d}t&=&\bo\left(\frac{g''}{(g')^2}\right)\frac{1}{z'(S_1)^{n-2}}\int_{S_1}^{S_2}\rho_2 e^{-t}{\rm d}t\nonumber\\
&=&\bo\left(\frac{g''}{(g')^2}\right),
\end{eqnarray}
\begin{eqnarray}\label{eq4.44}
V'_1(t)=\frac{1}{y'(t)^{n-2}}\int_t^\infty\rho_1 e^{-s} V_1{\rm d}s=\bo\left(\frac{1}{z'(t)^{n-2}}\int_t^\infty\rho_2 e^{-s}{\rm d}s\right)=\bo(1),
\end{eqnarray}
\begin{eqnarray*}
\frac{1}{y'(S_1)^{n-2}}\left(y'(t)^{n-2}-z'(t)^{n-2}\right)V_2'(t)V_1(t)&=&\frac{V_1(t)}{y'(S_1)^{n-2}}\int_t^\infty(\rho_1-\rho_2)e^{-t}{\rm d}t+\bo\left(\frac{g''}{(g')^2}\right)\\
&=&\bo\left(\frac{g''}{(g')^2}\right)
\end{eqnarray*}
and similarly
\begin{eqnarray*}
\frac{1}{y'(S_1)^{n-2}}\int_{S_1}^{S_2}(y'(t)^{n-2}-z'(t)^{n-2})V'_2(t)V_1(t){\rm d}t=\bo\left(\frac{g''}{(g')^2}(S_2-S_1)\right)=\bo\left(\frac{g''}{(g')^2}\right).
\end{eqnarray*}
Therefore from the above estimates, we have
\begin{equation}
\label{eq4.45}
\begin{array}{lll}
\frac{1}{y'(S_1)^{n-2}}\int_{S_1}^{S_2}\rho_3 V_1{\rm d}t&=& \frac{1}{y'(S_1)^{n-2}}\left[(y'(t)^{n-2}-z'(t)^{n-2})V_2'(t)V_1(t)\right]_{S_1}^{S_2}\\
&&-\frac{1}{y'(S_1)^{n-2}}\int_{S_1}^{S_2}(y'(t)^{n-2}-z'(t)^{n-2})V_2'(t)V'_1(t){\rm d}t \\
&=&\bo\left(\frac{g''}{(g')^2}\right).
\end{array}
\end{equation}

Taking $\eta=S_2$, $t=S_1$ in \eqref{eq4.19} and from the above estimates we have,
\begin{eqnarray*}
V'_1(S_1)V_2(S_1)&=&\left(\frac{y'(S_2)}{y'(S_1)}\right)^{n-2}\left[(V'_1(S_2)-V'_2(S_2))V_2(S_2)+(V_2(S_2)-V_1(S_2))V'_1(S_2)\right]+\bo\left(\frac{g''}{(g')^2}\right)\\
&=&\bo\left(\frac{g''}{(g')^2}\right).
\end{eqnarray*}

From Claim \ref{clm1}, we have  $$V'_1(S_1)\geq V'_1(S_2)\frac{(y')^{n-2}(S_2)}{(y')^{n-2}(S_1)}=V'_2(S_2)\bo\left(e^{\frac{g''}{(g')^4}}\right)+\bo\left(\frac{g''}{(g')^2}\right)\geq C_3>0$$ for some $C_3$ independent of $\gamma$. Hence $V_2(S_1)=\bo\left(\frac{g''}{(g')^2}\right)$ and
\begin{eqnarray*}
\left\vert\frac{1-e^{\frac{S_0-S_1}{n-1}}}{1+e^{\frac{T_1-S_1}{n-1}}}\right\vert=\vert V_2(S_1)\vert=\bo\left(\frac{g''}{(g')^2}\right).
\end{eqnarray*}
Since $T_1-S_1=\bo(1)$ and hence from the above equation,we get $S_0-S_1=\bo\left(\frac{g''}{(g')^2}\right)$. This proves \eqref{eq4.40}. Also
\begin{eqnarray}\label{myeq3}
\vert V_2(S_1)-V_1(S_1)\vert=\vert V_2(S_1)\vert=\bo\left(\frac{g''}{(g')^2}\right).
\end{eqnarray}

We now have two situations to handle, but before that we shall prove a few other estimates.
Let $\varepsilon > 0$ such that $\frac{n}{(n-1)}\varepsilon = \frac12$ and taking $t=S_0$ and  $\eta= S_0 + \varepsilon$ in \eqref{eq4.33bis}, we get on $[S_0,S_0+\varepsilon]$
\begin{eqnarray*}
\vert V_1(s)-V_2(s)\vert&\leq& \vert V_1(\eta)-V_2(\eta)\vert+ e^A\vert V'_1(\eta)-V'_2(\eta)\vert\vert s-\eta\vert\\
&+&(e^A-1)\left[\vert V'_2(\eta)\vert \vert s-\eta\vert+\frac{n}{n-1}\log(1+e^{\frac{T_1-S_0}{n-1}})\right]+\frac{n}{n-1}\vert s-\eta\vert M_\eta(s)
\end{eqnarray*}
which now easily gives the estimate
\begin{eqnarray}\label{eq4.48}
M_\eta(s)\leq C_0\left(\vert V_1(\eta)-V_2(\eta)\vert+\vert V'_1(\eta)-V'_2(\eta)\vert+\bo\left(\frac{g''}{(g')^2}\right)\right)
\end{eqnarray}
for all $s \in [S_0, S_0+\varepsilon]$ and some constant $C_0>0$ independent of $\ga$.

Since $V_2(S_0+\varepsilon)=\frac{1-e^\varepsilon}{1+e^{\frac{T_1-S_0}{n-1}}}\geq C_1>0$ depending  only on $\varepsilon$, observe that if we replace $S_1$ by $S_0$ in \eqref{eq4.43}, \eqref{eq4.44} and \eqref{eq4.45}, still all the estimates hold. Hence for $t\in [S_0+\varepsilon, S_2]$ and from previous claims,
\begin{equation}\label{myeq2}
\begin{array}{lll}
V'_1(t)V_2(t)-V_1(t)V'_2(t)&=&\left(\frac{y'(S_2)}{y'(t)}\right)^{n-2}(V'_1(S_2)V_2(S_2)-V_1(S_2)V'_2(S_2))+\bo\left(\frac{g''}{(g')^2}\right)\\
&=&\left(\frac{y'(S_2)}{y'(t)}\right)^{n-2}(V_1'(S_2)-V_2'(S_2))V_2(S_2)+\\
&& + (V_2(S_2)-V_1(S_2))V'_2(S_2)+\bo\left(\frac{g''}{(g')^2}\right)\\
&=&\bo\left(\frac{g''}{(g')^2}\right).
\end{array}
\end{equation}

Note that we needed $S_0 + \varepsilon < S_2$, or in other words we needed $\frac12 \frac{(n-1)}{n} =: \varepsilon < (n-1) [k_0 + \log(n-1)]$. Choose $k_0>0$ large such that the required estimate holds.

Similar to Claim \ref{clm1}, we also have $S_0 + \varepsilon > S_1$ and hence $V_1(S_0+\varepsilon)>0$. In particular, 
\begin{equation*}
\label{myeq1}
 V_1(t) \geq C_6 > 0 \text{    for   } t \in [S_0 + \varepsilon, S_2]
\end{equation*}

\begin{proof}[Proof of $S_0 + \varepsilon > S_1$] Suppose not, then we must have $S_0 - S_1 \leq - \varepsilon$.  From the definition of $V_2$, we get
 \begin{equation*}
   V_2(S_1)  = \frac{1-e^{\frac{S_0-S_1}{n-1}}}{1+(n-1)e^{\frac{S_0-S_1}{n-1}}}  = V_2(S_1) - V_1(S_1)  = \bigo{\frac{g''}{(g')^2}}.
 \end{equation*}
This implies for $\ga $ large enough
\begin{equation*}
 1 = e^{ \frac{S_0-S_1}{n-1} } + \bigo{ \lbr 1 + e^{\frac{S_0-S_1}{n-1}} \rbr \frac{g''}{(g')^2} } \leq e^{-\varepsilon} +  \bo \lbr\  [1+ (n-1) e^{-\varepsilon} ]  \frac{g''}{(g')^2}\rbr < 1
\end{equation*}
which gives the necessary contradiction and this proves that $S_0+\varepsilon > S_1$. 
\end{proof}

As a consequence of $S_0 + \varepsilon > S_1$, we see that $V_1(t) \geq C_6 > 0$ for $t \in [S_0+\varepsilon, S_2]$. 
Since $V_2(S_0) =0$, we also have $V_2(s) \geq \tilde{C_6} >0$ for $s \in [S_0+\varepsilon, S_2]$. This allows us to divide by $V_1$ and $V_2$ in \eqref{myeq2} to get
\begin{eqnarray*}
\frac{V_1(t)}{V_2(t)}=\frac{V_1(S_2)}{V_2(S_2)}e^{\bo\left(\frac{g''}{(g')^2}\right)}.
\end{eqnarray*}
Integrating from $t$ to $S_2$, we now get
\begin{eqnarray*}
V_1(t)-V_2(t)&=&\frac{V_1(S_2)}{V_2(S_2)}V_2(t)-V_2(t)+\bo\left(\frac{g''}{(g')^2}\right)\\
&=&(V_1(S_2)-V_2(S_2))\frac{V_2(t)}{V_2(S_2)}+\bo\left(\frac{g''}{(g')^2}\right)=\bo\left(\frac{g''}{(g')^2}\right).
\end{eqnarray*}
Hence from \eqref{eq4.32} and \eqref{eq4.39}, $\vert V'_1(t)-V'_2(t)\vert= \bo\left(\frac{g''}{(g')^2}\right)$. That is for all $t\in [S_0+\varepsilon, S_2]$,
\begin{eqnarray*}
\vert V_1(t)-V_2(t)\vert+\vert V'_1(t)-V'_2(t)\vert=\bo\left(\frac{g''}{(g')^2}\right).
\end{eqnarray*}
Using this in \eqref{eq4.48}, we get
\begin{equation*}
 \vert V_1(t) - V_2(t) \vert = \bigo{\frac{g''}{(g')^2}} \text{   for   } t \in [S_0,S_0+\varepsilon]
\end{equation*}
and now making use of \eqref{eq4.32}, we get
\begin{equation*}
 \vert V_1'(t) - V_2'(t) \vert = \bigo{\frac{g''}{(g')^2}} \text{   for   } t \in [S_0,S_0+\varepsilon].
\end{equation*}
Let 
\begin{eqnarray*}
I=\left\{\begin{array}{ll}
&[S_0,S_1]\quad\mbox{if }S_0<S_1\\
&[S_1, S_0]\quad\mbox{if }S_0\geq S_1.
\end{array}\right.
\end{eqnarray*}

\noindent \emph{Case i:} $I = [S_0,S_1]$. In this case, our previous calculations automatically gives the desired estimate
\begin{equation*}
 \vert V_1(t) - V_2(t) \vert + \vert V_1'(t) - V_2'(t) \vert = \bigo{\frac{g''}{(g')^2}} \text{   for   } t \in I
\end{equation*}

\noindent \emph{Case ii:} $I = [S_1,S_0]$.  From \eqref{myeq3}, we see that for $t \in I$, 

\begin{eqnarray}\label{eq4.46}
V_1(t)-V_2(t)&=&V_1(S_1)-V_2(S_1)-\int_t^{S_1}(V'_1(\theta)-V'_2(\theta)){\rm d}\theta\nonumber\\
&=&\bo\left(\frac{g''}{(g')^2}\right)+\bo(\vert S_1-S_0\vert)=\bo\left(\frac{g''}{(g')^2}\right).
\end{eqnarray}
Hence from \eqref{eq4.30}, \eqref{eq4.32}, we obtain
\begin{eqnarray}\label{eq4.47}
\vert V'_1(t)-V'_2(t)\vert =\bo\left(\frac{g''}{(g')^2}\right) \text{    for   } t \in I.
\end{eqnarray}
This completes the proof of Claim \ref{clm3}. 
\end{proof}
Let $k_1>0$ to be chosen later and $S_3=S_1-(n-1)k_1$. Then $$S_3-S_0=S_1-S_0-(n-1)k_1=\bo\left(\frac{g''}{(g')^2}\right)-(n-1)k_1$$ and $$T_1-S_3=T_1-S_0-(S_1-S_0)+(n-1)k_1=\bo(1).$$ Hence
\begin{eqnarray}\label{eq4.49}
V_2(S_3)=\frac{1-e^{k_1+O\left(\frac{g''}{(g')^2}\right)}}{1+e^{O(1)}}\leq -C_3
\end{eqnarray}
which proves \eqref{eq4.38}.

 Consider $s \in [S_3, S_1]$ and now using \eqref{eq4.33}, we get with $\eta = S_1$
 \begin{equation*}
\begin{array}{lll}
V_1(s)-V_2(s)&\leq&\vert V_1(S_1)-V_2(S_1) \vert- \vert V_1'(S_1)-V'_2(S_1)\vert e^A\int_s^{S_1}\left(\frac{z'(S_1)}{z'(\theta)}\right)^{n-2}{\rm d}\theta\\
&&+(e^A-1)\left[\vert V_2'(S_1)\vert \int_s^{S_1}\left(\frac{z'(S_1)}{z'(\theta)}\right)^{n-2}{\rm d}\theta +\int_s^{S_1}\rho_2V_2 e^{-\theta}(\int_s^\theta\frac{{\rm d}x}{(z'(x))^{n-2}}){\rm d}\theta\right]\\
&&+e^A\int_s^{S_1}\rho_2\vert V_2-V_1\vert e^{-\theta}(\int_s^\theta\frac{{\rm d}x}{z'(x)^{n-2}}){\rm d}\theta\\
&\leq& \vert V_1(S_1)-V_2(S_1) \vert- \vert V_1'(S_1)-V'_2(S_1)\vert 2 (n-1) k_1 \\
&&+(e^A-1)\left[\vert V_2'(S_1)\vert (n-1)k_1 +\lbr \frac{n-1}{n} \rbr \log \lbr  1 + e^{\frac{T_1 - S_3}{n-1}}\rbr\right]\\
&&+e^A M_{S_3} (s) \frac{n}{(n-1)^2} k_1 (n-1).
\end{array}
\end{equation*}
Now choose $k_1$ small such that $e^A \frac{n}{(n-1)} k_1 \leq \frac12$ which implies
$$M_{S_3} (S_1) = \bigo{\frac{g''}{(g')^2}}.$$
This completes the proof of the Lemma. \qed 


\begin{lemm}\label{lemma4.3}
Let $k_1>0$ and $S_3=S_1-(n-1)k_1$ be as in Lemma \ref{lemma4.2}. Let $\delta>0$, $\delta=\bo(\log(g'))$ and $T_\delta=T_1-(n-1)\delta$. Define
\begin{eqnarray}\label{eq4.51}
S_4=\max\{T_\delta, S\}.
\end{eqnarray}
Then for all $t\in [S_4, S_3]$, we have
\begin{eqnarray}\label{eq4.52}
\vert V_1(t)-V_2(t)\vert+\vert V'_1(t)-V'_2(t)\vert=\bo\left(\frac{\delta^3 g''}{(g')^2}\right).
\end{eqnarray}
\end{lemm}
\proof
From Lemma \ref{lemma4.2} and \eqref{eq4.33}, $$V_1(S_3)=V_2(S_3)+\bo\left(\frac{g''}{(g')^2}\right)<-C_3+\bo\left(\frac{g''}{(g')^2}\right).$$ Hence for $t\in [S_4,S_3]$ and $\gamma$ large,we have  
\begin{gather*}
V_1(S_3)\leq -\frac{C_3}{2}<0,\quad \vert S_4-S_3\vert=\bo(\delta) \quad  \log(1+e^\delta)=\bo(\delta), \\
0\leq V'_1(t)\leq \frac{y'(S_4)^{n-2}}{y'(t)^{n-2}}V'_1(S_3)=V'_2(S_3)+\bo(1)=\bo(1).
\end{gather*}
Since $V_1(t)<0$ for $t\in [S_4,S_3]$, by using \eqref{eq4.25}, we obtain 
\begin{equation}\label{eq4.53}\begin{array}{lll}
\frac{1}{y'(t)^{n-2}}\int_t^{S_3}(\rho_1-\rho_2)V_1V_2 e^{-s}{\rm d}s&=&\bo\left(\frac{\delta^2 g''}{(g')^2}\vert\frac{1}{y'(t)^{n-2}}\int_t^{S_3}\rho_1 V_1 e^{-s}{\rm d}s\vert\right)\\
&=&\bo\left(\frac{\delta^2 g''}{(g')^2}(V'_1(t)-\left(\frac{y'(S_3)}{y'(t)}\right)^{n-2}V'_1(S_3))\right)\\\
&=&\bo\left(\frac{\delta^2 g''}{(g')^2}\right).\end{array}
\end{equation}
For $s\in [t, S_3]$ and from \eqref{eq3.24}, we have
\begin{eqnarray*}
\frac{1}{y'(t)^{n-2}}\left[y'(s)^{n-2}-z'(s)^{n-2}\right]=\bo\left((\left(\frac{y'(s)}{z'(s)}\right)^{n-2}-1)\left(\frac{z'(s)}{z'(t)}\right)^{n-2}\right)=\bo\left(\frac{\delta^2 g''}{(g')^2}\right),
\end{eqnarray*}
\begin{eqnarray}\label{eq4.54}
\frac{1}{y'(t)^{n-2}}\int_t^{S_3}\rho_3 V_1(s){\rm d}s&=&\frac{1}{y'(t)^{n-2}}\left[(y'(s)^{n-2}-z'(s)^{n-2})V_1V'_2\right]_t^{S_3}\nonumber\\
&&-\frac{1}{y'(t)^{n-2}}\int_t^{S_3}(y'(s)^{n-2}-z'(s)^{n-2})V'_1V'_2{\rm d}s\nonumber\\
&=&\bo\left(\frac{\delta^2 g''}{(g')^2}(1+\vert V_1(t)\vert)+\bo\left(\frac{\delta^2 g''}{(g')^2}\int_t^{S_3}V'_2(s){\rm d}s\right)\right)\nonumber\\
&=& \bo\left(\frac{\delta^2 g''}{(g')^2}\vert V_1(t)\vert\right).
\end{eqnarray}
From \eqref{eq4.49}, we see that $V_2(t) \leq -C_3$ for $t \in [S_4,S_3]$. Also since $V_1(S_3) < 0$, we have
$$V_1(t) \leq V_1(S_3) = V_2(S_3) + \bigo{\frac{g'' \de^2}{(g')^2}} \leq \frac{-C_3}{2}.$$
This gives
\begin{eqnarray}\label{eq4.55}
0<\frac{1}{V_1(t)V_2(t)}\leq \frac{4}{C_3^2}=\bo(1), \quad\frac{1}{\vert V_1(t)\vert}\leq \frac{2}{C_3}=\bo(1).
\end{eqnarray}
From Lemma \ref{lemma4.2}, \eqref{eq4.19}, \eqref{eq4.52}, \eqref{eq4.53} and \eqref{eq4.55} we have
\begin{eqnarray*}
V'_1(t)V_2(t)-V_1(t)V'_2(t)&=&\left(\frac{y'(S_3)}{y'(t)}\right)^{n-2}(V'_1(S_3)V_2(S_3)-V_1(S_3)V'_2(S_3)) +\bo\left(\frac{\delta^2 g''}{(g')^2}\vert V_1(t)\vert\right)\\
&=&\left(\frac{y'(S_3)}{y'(t)}\right)^{n-2}\left( (V'_1(S_3)-V'_2(S_3))V_1(S_3)+(V_2(S_3)-V_1(S_3))V'_2(S_3)\right) +\\
& &+\bo\left(\frac{\delta^2 g''}{(g')^2}\vert V_1(t)\vert\right) \\
&=&\bo\left(\frac{\delta^2 g''}{(g')^2}\vert V_1(t)\vert\right).\\
\end{eqnarray*}
Dividing by $V_1 V_2$ and integrating between $[t, S_3]$, we obtain 
\begin{eqnarray*}
\frac{V_1'}{V_1}-\frac{V_2'}{V_2}=\bo\left(\frac{\delta^2 g''}{(g')^2}\frac{1}{\vert V_2(t)\vert}\right)=\bo\left(\frac{\delta^2 g''}{(g')^2}\right),
\end{eqnarray*}
\begin{eqnarray*}
\displaystyle V_1(t)=V_2(t)\frac{V_1(S_3)}{V_2(S_3)}e^{\bo\left(\frac{\delta^3 g''}{(g')^2}\right)}
\end{eqnarray*}
\begin{eqnarray*}
V_1(t)-V_2(t)=-\frac{V_2(t)}{V_2(S_3)}(V_2(S_3)-V_1(S_3))+\bo\left(\frac{\delta^3 g''}{(g')^3}\right)=\bo\left(\frac{\delta^3 g''}{(g')^2}\right).
\end{eqnarray*}
From \eqref{eq4.31} and \eqref{eq4.32} we have
\begin{eqnarray*}
\vert V'_1(t)-V'_2(t)\vert=\bo\left(\vert V'_1(S_3)-V'_2(S_3)\vert+\bo\left(\frac{\delta^3 g''}{(g')^2}\right)\right)=\bo\left(\frac{\delta^3 g''}{(g')^2}\right).
\end{eqnarray*}
This proves the lemma.\qed

As an immediate  consequence of this, we have the following
\begin{coro}\label{corollary4.4}
There exists a constant $C_4\geq 0$ such that for $\gamma$ large
\begin{eqnarray}\label{eq4.56}
S\leq T_1+(n-1)\log\left(\frac{g''}{(g')^2}\right)+3(n-1)\log(\log(g'))+(n-1)C_4.
\end{eqnarray}
\end{coro}
\proof Let $S_5:=T_1+(n-1)\log\left(\frac{g''}{(g')^2}\right)+3(n-1)\log(\log(g'))$ and $X(t)\eqdef e^{\frac{T_1-t}{n-1}}$. Then,
\begin{eqnarray*}
V'_2(S_5)&=&\frac{n}{(n-1)^2}\frac{1}{X(S_5)}+\bo\left(\frac{1}{X(S_5)^2}\right)\\
&=&\frac{n}{(n-1)^2}\frac{g''}{(g')^2}(\log(g'))^3+\bo\left(\left(\frac{g''}{(g')^2}\right)^2(\log(g'))^6\right).
\end{eqnarray*}
If $S\leq S_5$, then we can trivially take $C_4=0$.

If $S>S_5$, then from Lemma \ref{lemma4.3}, there exists a $C_5>0$ such that 
\begin{eqnarray}\label{myeq5}
V_2'(S)=V'_2(S)-V'_1(S)\leq C_5\frac{(\log(g'))^3 g''}{(g')^2}.
\end{eqnarray}

Since $V_2'(S) = \bigo{\frac{g''\de^3}{(g')^2}} \rightarrow 0$ as $\ga \rightarrow \infty$, we must have 
\begin{equation}
\label{myeq4}
 V_2'(S) = \frac{n}{(n-1)^2} \frac{X(S)}{(1+X(S))^2} \rightarrow 0 \text{  as  } \ga \rightarrow \infty.
\end{equation}
\begin{clm}
\label{clm4}
We must have $S < T_1$ and hence $X(S) \geq 1$. 
\end{clm}
\begin{proof}[Proof of Claim \ref{clm4}]
 Suppose $S \in [S_3,S_0]$, then we must have $X(S_0) \leq X(S) \leq X(S_3)$. From the definition of $S_0$ and Lemma \ref{lemma4.2}, we see that $$(n-1) \leq X(S) \leq e^{k_1}$$
which is a contradiction to \eqref{myeq4}. Thus $S<S_3$ and hence trivially $S<T_1$ which proves the claim. 
\end{proof}

Using \eqref{myeq5} along with \eqref{myeq4}, we see that
\begin{equation*}
 \begin{array}{ll}
  V_2'(S) \geq C \frac{1}{X(S)} \text{  for  } \ga \text{ large}.
 \end{array}
\end{equation*}
This implies that there is a constant $C_4 > 0$ such that $$ S \leq T_1 + (n-1) \log \lbr \frac{g''}{(g')^2} \rbr + 3(n-1) \log \log g' + C_4.$$ \qed


The estimates so far obtained are rough estimates and we need to improve them in order to prove the theorem. For this, we need some explicit formulas as follows:
\begin{lemm}\label{lemma4.5}
Let $X(t)=e^{\frac{T_1-t}{n-1}}$, then
\begin{eqnarray*}
I_1(t)&=&\int_t^\infty(z')^{n-2}z''V'_2{\rm d}s\\
&=&\frac{g'}{(n-1)}\left(\frac{n}{(n-1)g'}\right)^n\left[-\frac{1}{n(n-1)}+\displaystyle\sum_{r=0}^{n-1}\frac{(-1)^r C_r^{n-1}}{(r+2)(1+X(t))^{r+2}}\right],
\end{eqnarray*}
\begin{eqnarray*}
I_2(t)=\int_t^\infty(z')^n V_2'{\rm d}s=\frac{g'}{(n+1)}\left(\frac{n}{(n-1)g'}\right)^n\frac{X(t)^{n+1}}{(1+X(t))^{n+1}}
\end{eqnarray*}
and
\begin{eqnarray}\label{eq4.57}
I(t)&=&\frac{g''}{g'}I_1(t)+g'' I_2(t)\nonumber\\
&=&\left(\frac{n}{(n-1)g'}\right)^ng''\left[\frac{1}{n}+\frac{1}{n-1}\displaystyle\sum_{r=0}^{n-1}\frac{(-1)^r C_r^{n-1}}{(r+2)(1+X(t))^{r+2}}\right]\nonumber\\
&+& \left(\frac{n}{(n-1)g'}\right)^ng''\left[\frac{n}{(n^2-1)}\frac{X(t)^{n+1}-(1+X(t))^{n+1}}{(1+X(t))^{n+1}}\right].
\end{eqnarray}
\end{lemm}
\proof Since ${\rm d}X=-\frac{X}{n-1}{\rm d}t$ and $V'_2=-g' z''$. Then from \eqref{eq3.32}
\begin{eqnarray*}
I_1&=&\int_t^\infty(z')^{n-2}z'' V_2'{\rm d}s=-g'\int_t^\infty(z')^{n-2}(z'')^2{\rm d}s\\
&=&-\frac{g'}{(n-1)^2}\left(\frac{n}{(n-1)g'}\right)^n\int_t^\infty\frac{X(s)^n}{(1+X(s))^{n+2}}{\rm d}s\\
&=&-\frac{g'}{(n-1)}\left(\frac{n}{(n-1)g'}\right)^n\int_0^{X(t)}\frac{X^{n-1}}{(1+X)^{n+2}}{\rm d}X\\
&=&-\frac{g'}{(n-1)}\left(\frac{n}{(n-1)g'}\right)^n\displaystyle\sum_{r=0}^{n-1}(-1)^rC_r^{n-1}\int_0^{X(t)}\frac{{\rm d}X}{(1+X)^{r+3}}\\
&=&\frac{g'}{(n-1)}\left(\frac{n}{(n-1)g'}\right)^n\left[-\frac{1}{n(n+1)}+\displaystyle\sum_{r=0}^{n-1}\frac{(-1)^r C_r^{n-1}}{(r+2)(1+X(t))^{r+2}}\right].
\end{eqnarray*}
\begin{eqnarray*}
I_2(t)&=&\int_t^\infty(z')^nV'_2{\rm d}s=-g'\int_t^\infty(z')^n z''{\rm d}s\\
&=&\frac{g'}{n+1}z'(t)^{n+1}=\frac{g'}{n+1}\left(\frac{n}{(n-1)g'}\right)^{n+1}\frac{X(t)^{n+1}}{(1+X(t))^{n+1}}.
\end{eqnarray*}
\begin{eqnarray*}
I(t)&=&\frac{g''}{g'}I_1(t)+g'' I_2(t)\\
&=&\left(\frac{n}{(n-1)g'}\right)^n g''\left[-\frac{1}{n(n^2-1)}+\frac{1}{(n-1)}\displaystyle\sum_{r=0}^{n-1}\frac{(-1)^rC_r^{n-1}}{(r+2)(1+X(t))^{n+2}}\right]\\
&+&g''\left(\frac{n}{(n-1)g'}\right)^n\left[\frac{n}{(n^2-1)}\frac{X(t)^{n+1}}{(1+X(t))^{n+1}}\right].
\end{eqnarray*}
Since $-\frac{1}{n(n^2-1)}+\frac{n}{(n^2-1)}=\frac{1}{n}$,  the lemma follows.\qed
\begin{lemm}\label{lemma4.6}
Let
\begin{eqnarray*}
J(t)=\left(1-g'(y(t))y'(t)-\frac{g''(y(t))}{g'(y(t))}\right)y'(t)^{n-2}V'_1(t)+((y')^{n-2}V'_1)'.
\end{eqnarray*}
Then for $a<b$
\begin{eqnarray}\label{eq4.58}
J(a)&=&J(b)+\int_a^b\frac{g''(y(s))}{g'(y(s))}y''(s)y'(s)^{n-2}V'_1(s){\rm d}s\nonumber\\
&+&\int_a^b\left(g''(y(s))+\frac{g'''(y(s))}{g'(y(s))}-\left(\frac{g''(y(s))}{g'(y(s))}\right)^2\right)y'(s)^nV'_1(s){\rm d}s.
\end{eqnarray}
\end{lemm}
\proof Multiply the equation \eqref{eq4.12} by $(g'(y)y'+\frac{g''(y)y'}{g'(y)}-1)$ and integrate to obtain
\begin{eqnarray*}
-\int_a^b(g'(y)y'+\frac{g''(y)}{g'(y)}y'-1)((y')^{n-2}V'_1)'{\rm d}s&=&\frac{1}{n-1}\int_a^b(e^{g(y)+\log(g'(y))-s})'V_1{\rm d}s\\
&=&\left[\frac{V_1}{n-1}g'(y)e^{g(y)-s}\right]_a^b-\frac{1}{n-1}\int_a^b g'(y)e^{g(y)-s}V'_1{\rm d}s
\end{eqnarray*}
and
\begin{eqnarray*}
-\int_a^b(g'(y)y'+\frac{g''(y)}{g'(y)}y'-1)((y')^{n-2}V'_1)'{\rm d}s&=&-\left[(g'(y)y'+\frac{g''(y)}{g'(y)}y'-1)(y')^{n-2}V'_1\right]_a^b\\
&&+\int_a^b(g''(y)+\frac{g'''(y)}{g'(y)}-\left(\frac{g''(y)}{g'(y)}\right)^2)(y')^nV'_1{\rm d}s\\
&&+\int_a^b(g'(y)+\frac{g''(y)}{g'(y)})y''(y')^{n-2}V'_1{\rm d}s.
\end{eqnarray*}
Since $(n-1)g'(y)y''(y')^{n-2}=((y')^{n-1})'g'(y)=-g'(y)e^{g(y)-s}$ hence cancelling this term on both sides yields the identity.\qed
\begin{lemm}\label{lemma4.7}
Let $\delta>0$, $\delta=\bo(\log(g'))$ and $T_\delta=T_1-(n-1)\delta$. For $t\geq \max\{S,T_\delta\}\eqdef S_4$ and $X(t)=e^{\frac{T_1-t}{n-1}}$, 
\begin{eqnarray}\label{eq4.59}
L(t)&\eqdef&\int_t^\infty\frac{g''(y)}{g'(y)}y''(y')^{n-2}V_1'{\rm d}s+\int_t^\infty(g''(y)+\frac{g'''(y)}{g'(y)}-\left(\frac{g''(y)}{g'(y)}\right)^2)(y')^nV_1'{\rm d}s\nonumber\\
&=&g''\left(\frac{n}{(n-1)g'}\right)^n\left[\frac{1}{n}+\frac{1}{(n-1)}\displaystyle\sum_{r=0}^{n-1}\frac{(-1)^rC_r^{n-1}}{(r+2)(1+X(t))^{r+2}}\right]\nonumber\\
&&+g''\left(\frac{n}{(n-1)g'}\right)^n\left[\frac{n}{(n^2-1)}\frac{X(t)^{n+1}-(1+X(t))^{n+1}}{(1+X(t))^{n+1}}+\bo\left(\frac{\delta^4 g''}{(g')^2}\right)\right],
\end{eqnarray}
then $J(t) = L(t)$. Furthermore, if  $t\leq S_3=S_1-(n-1)k_1$, then
\begin{eqnarray}\label{eq4.60}
((y')^{n-2}V'_1)'(t)=e^{\bo\left(\frac{\delta^3 g''}{(g')^2}\right)}((z')^{n-2}V_2')'(t).
\end{eqnarray}
\end{lemm}
\proof
Let $S_4 \leq t\leq S_3$, then $V_1(t)\leq -\tilde{C}_3$ and hence from \eqref{eq4.20}, we have
\begin{eqnarray*}
-y''=\frac{e^{g(y)-t}}{(n-1)(y')^{n-2}}=e^{\bo\left(\frac{g''\delta^2}{(g')^2}\right)}\frac{e^{g+g'(z-\gamma)-t}}{(n-1)(z')^{n-2}}= -e^{\bo\left(\frac{g''\delta^2}{(g')^2}\right)}z''
\end{eqnarray*}
and from \eqref{eq4.52} and \eqref{eq4.25}, we have
\begin{eqnarray*}
-((y')^{n-2}V'_1)'=\rho_1e^{-t}V_1=e^{\bo\left(\frac{g''\delta^2}{(g')^2}\right)}\rho_2 e^{-t}(V_2+\bo\left(\frac{\delta^3 g''}{(g')^2}\right))=-e^{\bo\left(\frac{g''\delta^3}{(g')^2}\right)}((z')^{n-2}V'_2)'.
\end{eqnarray*}
which proves \eqref{eq4.60}. From \eqref{eq4.20} to \eqref{eq4.25} and from Lemma \ref{lemma4.3}, we have
\begin{eqnarray*}
L(t)&=&e^{\bo\left(\frac{\delta^2g''}{(g')^2}\right)} \lbr[[]\frac{g''}{g'}\int_t^\infty(z')^{n-2}z''(V'_2+\bo\left(\frac{\delta^3 g''}{(g')^2}\right)){\rm d}s \rbr[.]	\\
&& \lbr[.]+g''\int_t^\infty(z')^n(V'_2+\bo\left(\frac{\delta^3 g''}{(g')^2}\right)){\rm d}s+\bo\left(\frac{g'''}{g'}+\left(\frac{g''}{g'}\right)^2\right)\int_t^\infty(z')^nV'_1{\rm d}s\rbr[]]\\
&=& e^{\bo\left(\frac{\delta^2g''}{(g')^2}\right)}\left[\frac{g''}{g'}I_1(t)+g'' I_2(t)\right]\\
&&+\bo\left(\frac{\delta^3(g'')^2}{(g')^3}\int_t^\infty(z')^{n-2}z''{\rm d}s
+\frac{\delta^3(g'')^2}{(g')^2}\int_t^\infty(z')^n{\rm d}s
+\left(\frac{g''}{(g')^n}+\left(\frac{g''}{g'}\right)^2\right)\int_t^\infty(z')^nV'_1{\rm d}s\right).
\end{eqnarray*}
Furthermore,
\begin{eqnarray*}
\int_t^\infty(z')^{n-2}z''{\rm d}s=\frac{z'(t)^{n-1}}{n-1}=\bo\left(\left(\frac{n}{(n-1)g'}\right)^{n-1}\right),
\end{eqnarray*}
and
\begin{eqnarray*}
\int_t^\infty(z')^n{\rm d}s&=&\bo\left(\left(\frac{n}{(n-1)g'}\right)^{n-1}\int_t^\infty z'(s){\rm d}s\right)=\bo\left(\left(\frac{n}{(n-1)g'}\right)^{n-1}(\gamma-z(t))\right)\\
&=&\bo\left(\left(\frac{n}{(n-1)g'}\right)^n\log(1+X(t))\right)=\bo\left(\left(\frac{n}{(n-1)g'}\right)^n\delta\right).
\end{eqnarray*}
From Lemma \ref{lemma4.3}, for $t\geq \max\{S,T_\delta\}$, we have $V_1(t)=V_2(t)+\bo\left(\frac{g''\delta^3}{(g')^2}\right)=\bo(1)$. Therefore
\begin{eqnarray*}
\int_t^\infty(z')^nV'_1{\rm d}s=\bo\left(\left(\frac{n}{(n-1)g'}\right)^n(1-V_1(t))\right)=\bo\left(\left(\frac{n}{(n-1)g'}\right)^n\right).
\end{eqnarray*}
Hence we get
\begin{eqnarray*}
L(t)=e^{\bo\left(\frac{\delta^2g''}{(g')^2}\right)}I(t)+\left(\frac{n}{(n-1)g'}\right)^n \bo\left(\frac{\delta^3(g'')^2}{(g')^2}+\frac{\delta^4(g'')^2}{(g')^2}\right).
\end{eqnarray*}
Now \eqref{eq4.49} follows from \eqref{eq4.57} and this completes the proof of the lemma. \qed

We now introduce an assumption on $S$:
\begin{asm}\label{assumption-on-S} Let
\begin{eqnarray}\label{eq4.61}
S_6=S_6(\gamma)\eqdef T_1-(\frac{4q}{q-1}+1) (n-1)\log(g'),
\end{eqnarray}
and now assume that there exists a sequence $\gamma_l\to\infty$ such that $S=S(\gamma_l)$ satisfies
\begin{eqnarray}\label{eq4.62}
S_6(\gamma_l)\leq S(\gamma_l)=S.
\end{eqnarray}
\end{asm}

For the next part, we assume that $S$ satisfies Assumption \ref{assumption-on-S} and derive the asymptotics as $\gamma_l\to\infty$. By a slight abuse of notation, for the subsequent sections, we denote $\gamma_l$ by $\gamma$, $S_6(\gamma_l)$ to be $S_6$ and $S(\gamma_l)$ by $S$. We suppress writing the subsequence $l$ and  mean $\gamma$ is large to denote $\gamma_l$ is large. Then we have the following crucial result.
\begin{lemm}\label{lemma4.8}
Assume that $S$ satisfies Assumption \ref{assumption-on-S}, then for $\gamma$ large, we have
\begin{eqnarray}\label{eq4.63}
S=T_1+(n-1)\log\left(\frac{(n-1)g''}{(g')^2}+\bo\left(\frac{\delta^4 (g'')^2}{(g')^4}\right)\right)
\end{eqnarray}
and
\begin{eqnarray}\label{eq4.64}
V'_2(S)=\frac{n}{n-1}\frac{g''}{(g')^2}\left(1+\bo\left(\frac{\delta^4 g''}{(g')^2}\right)\right).
\end{eqnarray}
\end{lemm}
\proof
As before, we let $X(t):=e^{\frac{T_1-t}{n-1}}$. Then from \eqref{eq4.56}, we trivially have
\begin{eqnarray*}
X(S)\geq \frac{e^{-C_4}(g')^2}{(\log(g'))^3 g''}\to\infty\quad\mbox{as }\gamma\to\infty.
\end{eqnarray*}
We have
\begin{eqnarray*}
(z')^{n-2}V_2'=\frac{n}{(n-1)^2}\left(\frac{n}{(n-1)g'}\right)^{n-2}\frac{X(t)^{n-1}}{(1+X(t))^n},
\end{eqnarray*}
\begin{eqnarray}\label{eq4.65}
((z')^{n-2}V'_2)'&=&\frac{n}{(n-1)^3}\left(\frac{n}{(n-1)g'}\right)^{n-2}\left(\frac{n X(t)^n}{(1+X(t))^{n+1}}-\frac{(n-1)X(t)^{n-1}}{(1+X(t))^n}\right)\nonumber\\
&=&\frac{n}{(n-1)^3}\left(\frac{n}{(n-1)g'}\right)^{n-2}\left(\frac{1}{X(t)}+\bo\left(\frac{1}{(X(t))^2}\right)\right).
\end{eqnarray}
Set $a=S$ and $b=\infty$ in \eqref{eq4.58} and from \eqref{eq4.59}, \eqref{eq4.60}, \eqref{eq4.65} we have for $\delta=\log(g')$
\begin{equation*}\begin{array}{ll}
((z')^{n-2}V'_2)'(S) & =\frac{n}{(n-1)^3}\left(\frac{n}{(n-1)g'}\right)^{n-2}\left(\frac{1}{X(S)}+\bo\left(\frac{1}{X(S)^2}\right)\right)\\
& =e^{\bo\left(\frac{\delta^3 g''}{(g')^2}\right)}((y')^{n-2}V_1')'(S)=e^{\bo\left(\frac{\delta^3 g''}{(g')^2}\right)}L(S)\\
& =e^{\bo\left(\frac{\delta^3 g''}{(g')^2}\right)}g''\left((\frac{n}{(n-1)g'}\right)^{n}\left(\frac{1}{n}+\bo\left(\frac{\delta^4 g''}{(g')^2}\right)\right)
\end{array}\end{equation*}
and
\begin{eqnarray*}
\frac{1}{X(S)}=\frac{(n-1)g''}{(g')^2}(1+\bo\left(\frac{1}{X(S)}\right))+\bo\left(\frac{1}{X(S)^2}+\delta^4\left(\frac{g''}{(g')^2}\right)^2\right).
\end{eqnarray*}
Hence
\begin{eqnarray*}
e^{\frac{S-T_1}{n-1}}=\frac{1}{X(S)}=\frac{(n-1)g''}{(g')^2}\left[1+\bo\left(\frac{\delta^4 g''}{(g')^2}\right)\right],
\end{eqnarray*}
which gives after taking logarithms
\begin{eqnarray*}
S=T_1+(n-1)\log\left(\frac{(n-1)g''}{(g')^2}+\bo\left(\delta^4\left(\frac{g''}{(g')^2}\right)^2\right)\right)
\end{eqnarray*} 
and
\begin{eqnarray*}
V'_2(S)&=&\frac{n}{(n-1)^2}\frac{X(S)}{(1+X(S))^2}
=\frac{n}{(n-1)^2}\left(\frac{1}{X(S)}-\frac{2}{X(S)^2}+\bo\left(\frac{1}{X(S)^3}\right)\right)\\
&=&\frac{n}{n-1}\frac{g''}{(g')^2}\left(1+\bo\left(\frac{\delta^4 g''}{(g')^2}\right)\right).
\end{eqnarray*}
This proves the lemma.\qed

Next we study the behaviour of $V_1$ for $t\in [T(\gamma), S]$.
\begin{lemm}\label{lemma4.9}
Let $S$ satisfies the Assumption \ref{assumption-on-S}, then there exists a $C_0>0$ such that for $\gamma$ large with $t\in [S-C_0g, S)$, the following holds:
\begin{eqnarray}\label{eq4.66}
V_1(t)<0, \quad \text{and    } \quad V'_1(t)<0.
\end{eqnarray}
\end{lemm}
\proof Let $S_6$ defined in \eqref{eq4.61} and
\begin{eqnarray}\label{eq4.67}
t_0\eqdef\inf\{ t\in[T(\gamma),S];\; V_1(s)<0\mbox{ for }s\in(t,S]\}.
\end{eqnarray}
Observe that from \eqref{eq4.63}, there exists a $C_1>0$ such that for $\gamma$ large $\tilde{T}\leq S-C_1g$.

\noindent \emph{Case i:} If $\tilde{T}\geq t_0$, then there is nothing to prove.

\noindent \emph{Case ii:} Assume that $\tilde{T}<t_0$. Then $V_1(t_0)=0$ and $0\leq -V_1(t)\leq -V_1(S)$ for $t\in (t_0,S)$. Hence from \eqref{eq3.20}, \eqref{eq3.41}, \eqref{eq4.10} and \eqref{eq4.12}, we have 
\begin{eqnarray*}
\int_{t_0}^{S_6}\rho_1 e^{-s} V_1\left(\int_{t_0}^s\frac{{\rm d}\theta}{(y'(\theta))^{n-2}}\right){\rm d}s
&=&\bo\left(\vert V_1(S)\vert (g')^{n-1}(S-t_0)\int_{\tilde{T}}^{S_6}e^{g(y(s))-s}{\rm d}s\right)\\
&=& \bo\left(\frac{\vert V_1(S)\vert (S-t_0)(g')^{n-1+\frac{q}{q-1}}}{(g')^{\frac{4q}{q-1}+n}}\right)=\bo\left(\frac{\vert V_1(S)\vert(S-t_0)}{g}\right).
\end{eqnarray*}
Similarly, now from \eqref{eq4.20}, we have
\begin{eqnarray*}
\int_{S_6}^S\rho_1 e^{-s} V_1\left(\int_{t_0}^s\frac{{\rm d}\theta}{(y'(\theta))^{n-2}}\right){\rm d}s&=&\bo\left(\vert V_1(S)\vert(S-T_0)(g')^{n-1}\int_{S_6}^S e^{g+g'(z-\gamma)-s}{\rm d} s\right)\\
&=&\bo\left(\vert V_1(S)\vert(S-t_0)(g')^{n-1}\vert z'(S)^{n-1}-z'(S_6)^{n-1}\vert\right)\\
&=&\bo\left(\vert V_1(S)\vert (S-t_0)\left\vert\frac{X(S)^{n-1}}{(1+X(S))^{n-1}}-\frac{X(S_6)^{n-1}}{(1+X(S_6))^{n-1}}\right\vert\right)\\
&=&\bo\left(\frac{\vert V_1(S)\vert (S-t_0)}{X(S)}\right)=\bo\left(\frac{\vert V_1(S)\vert (S-t_0)g''}{(g')^2}\right)\\
&=&\bo\left(\frac{\vert V_1(S)\vert (S-t_0)}{g}\right).
\end{eqnarray*}
Here we have used the fact that for any $a>0$ and $b>0$, we have
\begin{equation*}\begin{array}{ll}
 \frac{a^{n-1}}{(1+a)^{n-1}} - \frac{b^{n-1}}{(1+b)^{n-1}} & = \bigo{\frac{a^{n-1} (1+b)^{n-1} - b^{n-1} (1+a)^{n-1}}{(1+a)^{n-1}(1+b)^{n-1}}}\\
 & = \bigo{\frac{a^{n-1} b^{n-2} + b^{n-1} a^{n-2}}{a^{n-1}b^{n-1}}} = \bigo{\frac1a + \frac1b} = \bigo{\frac1a}. 
 \end{array}
\end{equation*}
Combining all the above estimates, we get 
\begin{eqnarray*}
\vert V_1(S)\vert&=&-V_1(S)+V_1(t_0) \\
&=&-\int_{t_0}^{S_6}\rho_1V_1 e^{-s}\left(\int_{t_0}^s\frac{{\rm d}\theta}{(y'(\theta))^{n-2}}\right){\rm d}s-\int_{S_6}^{S}\rho_1V_1e^{-s}\left(\int_{t_0}^s\frac{{\rm d}\theta}{(y'(\theta))^{n-2}}\right){\rm d}s\\
&=&\bo\left(\frac{\vert V_1(S)\vert(S-t_0)}{g}\right).
\end{eqnarray*}
This completes the proof of the lemma.\qed

Next we have
\begin{lemm}\label{lemma4.10}
Let $S$ satisfies the Assumption \ref{assumption-on-S} and $S_6$ be as in \eqref{eq4.61}, then the following holds: 

\begin{gather}
\label{eq4.68} V_1(S_6)=V_1(S)\left(1+(n-1)^2V'_2(S)-(n-1)(S-S_6)V'_2(S)\right)+\bo\left(\frac{\delta^4 g''}{(g')^4}\right),\\
\label{eq4.69} V'_1(S_6)=(n-1)V_1(S)V'_2(S)+\bo\left(\frac{\delta^3 g''}{(g')^4}\right). 
\end{gather}

\end{lemm}
\proof Let $t\in [S_6, S]$, then $X(t)=e^{\frac{T_1-t}{n-1}}\to\infty$ as $\gamma\to\infty$. From Lemma \ref{lemma4.9}, $V_1(t)<0$ for $t\in [S_6,S]$, hence 
\begin{eqnarray*}
\log(-(n-1)V_2(t))&=&\log\left(1-\frac{n}{1+X(t)}\right)=-\frac{n}{1+X(t)}+\bo\left(\frac{1}{X(t)^2}\right)\\
&=&-(n-1)^2\frac{n}{(n-1)^2}\frac{X(t)}{(1+X(t))^2}+\bo\left(\frac{1}{X(t)^2}\right)\\
&=&-(n-1)^2V'_2(t)+\bo\left(\frac{1}{X(t)^2}\right),
\end{eqnarray*}
and $V'_2(S_6)=\bo\left(\frac{1}{(g')^{\frac{4q}{q-1}+1}}\right)$. Therefore, from lemma \ref{lemma4.3},
\begin{eqnarray*}
\int_{S_6}^S\frac{V_1(s)V_2'(s)}{V_2(s)}{\rm d}s&=&\int_{S_6}^S(\log(-(n-1)V_2(s))'V_1(s){\rm d}s\\
&=&\left[V_1(s)\log(-(n-1)V_2(s)\right]_{S_6}^{S}-\int_{S_6}^SV'_1(s)\log(-(n-1)V_2(s)){\rm d}s\\
&=&-(n-1)^2V_1(S)V'_2(S)+\bo\left(V_1(S_6)V'_2(S_6)\right)+\bo\left(V'_1(S_6)\int_{S_6}^{S}V'_2(s){\rm d}s\right)\\
&=&-(n-1)^2V_1(S)V'_2(S)+\bo\left(\frac{V_1(S_6)}{(g')^{\frac{4q}{q-1}+1}}\right)+\bo\left(V'_1(S_6)(V_2(S)-V_2(S_6))\right).
\end{eqnarray*}
Since $V_2(S)-V_2(S_6)=\frac{n}{n-1}\left(\frac{1}{1+X(S)}-\frac{1}{1+X(S_6)}\right)=\bo\left(\frac{1}{X(S)}\right)$, hence from lemma \ref{lemma4.8},
\begin{eqnarray}\label{eq4.70new}
\int_{S_6}^S\frac{V_1(s)V'_2(s)}{V_2(s)}{\rm d}s=-(n-1)^2V_1(S)V'_2(S)+\bo\left(\frac{V'_1(S_6)g''}{(g')^2}+\frac{(V_1(S_6))}{(g')^{\frac{4q}{q-1}+1}}\right).
\end{eqnarray}
With $\delta=\log(g')$ for $t\in [S_6,S]$, we have
\begin{eqnarray*}
y'(t)=e^{\bo\left(\frac{\delta^2 g''}{(g')^2}\right)}z'(t),\quad y''(t)=-\frac{e^{g(y)-t}}{(n-1)(y')^{n-2}}=e^{\bo\left(\frac{\delta^2 g''}{(g')^2}\right)}z''(t),
\end{eqnarray*}
\begin{eqnarray*}
y'(t)^{n-3}y''(t)-z'(t)^{n-3}z''(t)&=&\left(1-e^{\bo\left(\frac{\delta^2 g''}{(g')^2}\right)}\right)(z')^{n-3}z''\\
&=&\bo\left(\frac{\delta^2 g''}{(g')^n}\frac{1}{X(S)}\right)=\bo\left(\frac{\delta^2 (g'')^2}{(g')^{n+2}}\right),
\end{eqnarray*}
\begin{eqnarray*}
y'(t)^{n-2}-z'(t)^{n-2}=\left(1-e^{\bo\left(\frac{\delta^2 g''}{(g')^2}\right)}\right)z'(t)^{n-2}=\bo\left(\frac{\delta^2 g''}{(g')^n}\right)
\end{eqnarray*}
and
\begin{eqnarray*}
V'_2(t)+\vert V''_2(t)\vert=\bo\left(\frac{1}{X(S)}\right)=\bo\left(\frac{g''}{(g')^2}\right).
\end{eqnarray*}
Hence
\begin{eqnarray*}
\rho_3&=&((y'(t)^{n-2}-z'(t)^{n-2})V'_2)'\\
&=&(y'(t)^{n-2}-z'(t)^{n-2})V_2''+(n-2)(y'(t)^{n-3}y''(t)-z'(t)^{n-3}z''(t)V'_2=\bo\left(\frac{\delta^2 g''}{(g')^{n+2}}\right).
\end{eqnarray*}
and
\begin{eqnarray*}
\frac{1}{y'(t)^{n-2}}\int_t^S\rho_3 V_1{\rm d}s=\bo\left(\frac{(S-S_6)\delta^2 g''}{(g')^4}\right)=\bo\left(\frac{\delta^3 g''}{(g')^4}\right).
\end{eqnarray*}
From \eqref{eq4.25},
\begin{eqnarray*}
\frac{1}{y'(t)^{n-2}}\int_t^S(\rho_1-\rho_2)V_1V_2 e^{-s}{\rm d}s=\bo\left(\frac{\delta^2 g''}{(g')^2}\frac{1}{y'(t)^{n-2}}\int_t^S\rho_1 V_1 e^{-s}{\rm d}s\right)
=\bo\left(\frac{\delta^2 g''}{(g')^2} V'_1(t)\right).
\end{eqnarray*}
Since $V'_2(S_6)=\bo\left(\frac{1}{X(S_6)}\right)=\bo\left(\frac{1}{(g')^{\frac{4q}{q-1}+1}}\right)$, we have from \eqref{eq4.19},
\begin{eqnarray*}
V'_1(t)&=&\frac{V_1(t)V'_2(t)}{V_2(t)}-\left(\frac{y'(S)}{y'(t)}\right)^{n-2}\frac{V_1(S)V'_2(S)}{V_2(t)}+\frac{1}{y'(t)^{n-2}V_2(t)}\int_t^S(\rho_1-\rho_2)V_1V_2 e^{-s}{\rm d}s\\
&&+\frac{1}{y'(t)^{n-2}V_2(t)}\int_t^S\rho_3 V_1{\rm d}s\\
&=&\frac{V_1(t)V'_2(t)}{V_2(t)}+(n-1)V_1(S)V'_2(S)+\bo\left(\frac{\delta^2 g''}{(g')^2}V'_1(t)+\frac{\delta^3g''}{(g')^4}\right).
\end{eqnarray*}
Hence solving for $V'_1(t)$ to obtain
\begin{eqnarray}\label{eq4.70}
V'_1(t)=\frac{V_1(t)V'_2(t)}{V_2(t)}+(n-1)V_1(S)V'_2(S)+\bo\left(\frac{\delta^3 g''}{(g')^4}\right).
\end{eqnarray}
Thus we get
\begin{eqnarray*}
V'_1(S_6)=(n-1)V_1(S)V'_2(S)+\bo\left(\frac{\delta^3 g''}{(g')^4}\right)
\end{eqnarray*}
and this proves \eqref{eq4.69}. 
Integrating \eqref{eq4.70} and using the above estimates to obtain
\begin{eqnarray*}
V_1(S)-V_1(S_6)&=&\int_{S_6}^SV_1(t)\frac{V'_2(t)}{V_2(t)}{\rm d}t+(n-1)V_1(S)V'_2(S)(S-S_6)+\bo\left(\frac{\delta^4 g''}{(g')^4}\right)\\
&=&-(n-1)^2 V_1(S)V'_2(S)+(n-1)V_1(S)V'_2(S)(S-S_6)+\bo\left(\frac{\delta^4 g''}{(g')^4}\right).
\end{eqnarray*}
This proves \eqref{eq4.68} and hence the lemma.\qed
Next, we have
\begin{lemm}\label{lemma4.11}
Assume $S$ satisfies \eqref{eq4.62}. Then for $t\in [\tilde{T}, S_6]$
\begin{eqnarray}\label{eq4.71}
V'_1(t)=V'_1(S_6)+\bo\left(\frac{(g'')^2}{(g')^4}\right),
\end{eqnarray}
\begin{eqnarray}\label{eq4.72}
V_1(\tilde{T})=V_1(S_6)+V_1'(S_6)(\tilde{T}-S_6)+\bo\left(\frac{g''}{(g')^2}\right).
\end{eqnarray}
\end{lemm}
\proof From Lemma \ref{lemma4.9}, $V_1(t)<0$, $V'_1(t)<0$ for $t\geq S-C_0g$.
Let $t_0$ as in \eqref{eq4.67} and define $t_1\leq t_0<S_6$ by  
\begin{eqnarray}\label{eqt1}
t_1&\eqdef&\displaystyle\inf\{t<t_0;\; V_1(s)>0, V'_1(s)<0\mbox{ for }s\in(t,t_0)\},
\end{eqnarray}
that is $t_0$ is the first zero before $S$ and $t_1$ is the second turning point of $V_1$ if it exists.\\
\begin{clm}
\label{claim5}
For $ t\in [\tilde{T},S_6]$, we have $V_1(t)=\bo(1)$ and $ V'_1(t)\leq 0$.
\end{clm}
\begin{proof}[Proof of Claim \ref{claim5}]
Assume that $t_1\geq \tilde{T}$ exists. Then either (i) $V'_1(t_1)=0$ or (ii) $t_1=\tilde{T}$ and $V_1(\tilde{T})>0$, $V'_1(t)<0$ for $t\in [\tilde{T}, S_6]$. From \eqref{eq4.69}, we have $V'_1(S_6)=\bo\left(\frac{g''}{(g')^2}\right)$ and 
\begin{eqnarray*}
V'_1(S_6)\int_t^{S_6}\left(\frac{y'(S_6)}{y'(s)}\right)^{n-2}{\rm d}s=\bo\left(\frac{g''(S_6-t)}{(g')^2}\right)=\bo\left(\frac{g''(S_6-\tilde{T})}{(g')^2}\right)=\bo\left(\frac{g'' g}{(g')^2}\right)=\bo(1).
\end{eqnarray*}
From \eqref{eq4.12} and \eqref{eq3.41}, we have for $t_1\leq a\leq b\leq S_6$
\begin{eqnarray*}
& &\int_a^b\rho_1 V_1 e^{-s}\left(\int_a^s\frac{{\rm d}\theta}{(y'(\theta))^{n-2}}\right){\rm d}s=\bo\left(\max\left\{\vert V_1(a)\vert, \vert V_1(b)\vert\right\}(g')^{n-1}\int_{\tilde{T}}^{S_6} e^{g(y(s))-s}{\rm d}s\right)\\
&=&\max\left\{\vert V_1(a)\vert, \vert V_1(b)\vert\right\}\bo\left(\frac{(\tilde{T}-S_6)^2(g')^{n-1}}{(g')^{\frac{2q}{q-1}+n}}\right)
=\max\left\{\vert V_1(a)\vert,\vert V_1(b)\vert\right\}\bo\left(\frac{1}{(g')^{\frac{q}{q-1}+1}}\right).
\end{eqnarray*}
Hence for $t\in [t_1, S_6]$
\begin{eqnarray*}
V_1(t)&=&V_1(S_6)-V'_1(S_6)\int_t^{S_6}\left(\frac{y'(S_6)}{y'(\theta)}\right)^{n-2}{\rm d}\theta-\int_t^{S_6}\rho_1 V_1 e^{-s}\left(\int_t^{s}\frac{{\rm d}\theta}{y'(\theta)^{n-2}}\right){\rm d}s\\
&=&\bo(1)+\max\left\{\vert V_1(t)\vert,\vert V_1(S_6)\vert\right\}\bo\left(\frac{1}{(g')^{\frac{q}{q-1}+1}}\right).
\end{eqnarray*}
This implies $V_1(t)=\bo(1)$. Suppose $V_1'(t_1)=0$, then taking $a=t_1$, $b=t_0$ to obtain
\begin{eqnarray*}
0=V_1(t_0)=V_1(t_1)-\int_{t_1}^{t_0}\rho_1V_1e^{-s}\left(\int_t^s\frac{\partial\theta}{(y'(\theta))^{n-2}}\right){\rm d}s=V_1(t_1)+\bo\left(\frac{V_1(t_1)}{(g')^{\frac{q}{q-1}+1}}\right)
\end{eqnarray*}
which gives a contradiction if $V_1(t_1)>0$ and this proves the claim.
\end{proof}
From \eqref{eq3.41}, \eqref{eq3.43}, \eqref{eq4.14}, \eqref{eq4.69} and the above claim, we have for $t\in [\tilde{T}, S_6]$
\begin{eqnarray*}
V'_1(t)&=&\left(\frac{y'(S_6)}{y'(t)}\right)^{n-2}V'_1(S_6)+\frac{1}{y'(t)^{n-2}}\int_t^{S_6}\rho_1V_1e^{-s}{\rm d}s\\
&=&\left(\frac{1+\bo\left(\frac{g''}{(g')^2}\right)}{1+\bo\left(\frac{g''}{(g')^2}\right)}\right)V'_1(S_6)+\bo\left((g')^{n-1}\int_t^{S_6} e^{g(y(s))-s}{\rm d}s\right)=V'_1(S_6)+\bo\left(\frac{(g'')^2}{(g')^4}\right).
\end{eqnarray*}
Integrating the above expression, we get
\begin{eqnarray*}
V_1(\tilde{T})&=&V_1(S_6)+V'_1(S_6)(\tilde{T}-S_6)+\bo\left(\frac{(\tilde{T}-S_6)(g'')^2}{(g')^4}\right)\\
&=&V_1(S_6)+V'_1(S_6)(\tilde{T}-S_6)+\bo\left(\frac{\gamma g'(g'')^2}{(g')^4}\right) \\
&=&V_1(S_6)+V'_1(S_6)(\tilde{T}-S_6)+\bo\left(\frac{g''}{(g')^2}\right).
\end{eqnarray*}
This proves the lemma.\qed

\begin{lemm}\label{lemma4.12}
Assume that $S$ satisfies Assumption \ref{assumption-on-S}, then for $t\in [T(\gamma), \tilde{T}]$, we have
\begin{eqnarray}\label{eq4.74}
V'_1(t)=V'_1(\tilde{T})+\bo\left(\frac{(g'')^2}{(g')^3}+\frac{e^{-(g-\frac{(n-1)\gamma g'}{n})}}{(g')^\alpha}\right),
\end{eqnarray}
\begin{eqnarray}\label{eq4.75}
V_1(T(\gamma))=V_1(\tilde{T})+V'_1(\tilde{T})(T(\gamma)-\tilde{T})+\bo\left(\frac{g''}{(g')^3}+\frac{e^{-(g-\frac{(n-1)\gamma g'}{n})}}{(g')^\alpha}\right).
\end{eqnarray}
\end{lemm}
\proof First we show that for $t\in [T(\gamma), \tilde{T}]$
\begin{eqnarray}\label{eq4.76}
V_1(t)=O(1), V'_1(t)<0.
\end{eqnarray}
For this,  let $t_0$ and $t_1$ as in \eqref{eq4.67} and in \eqref{eqt1}.

Then from Claim \ref{claim5}, it follows that $t_1<\tilde{T}$. Then either $t_1>T(\gamma)$ and $V'_1(t_1)=0$, $V_1(t_1)>0$ or $t_1=T(\gamma)$ and $V'_1(t)<0$ for all $T(\gamma)<t< S$.  For $t\in[T(\gamma),\tilde{T}]$, $y(t)\leq s_0$ and hence $\rho_1(t)=\frac{f'(y(t))}{n-1}=\bo\left(y(t)^{-1+\alpha}\right)$. Therefore for $t\in (T(\gamma), \tilde{T})$, there exists a $\xi\in (T(\gamma), t)$ such that $y(t)=y'(\xi)(t-T(\gamma))$. Hence from \eqref{eq1.4} and \eqref{eq3.52} we have $\rho_1(t)=\frac{f'(y(t))}{n-1}=\bo\left(\frac{(t-T(\gamma))^{-1+\alpha}}{(g')^{-1+\alpha}}\right)$. Hence from \eqref{2.11}, \eqref{eq2.12}, \eqref{eq3.42}, \eqref{eq4.15}, \eqref{eq4.69} and \eqref{eq4.71}, we have
\begin{equation}\label{myeq7} \begin{array}{lll}
V_1(t)&=&V_1(\tilde{T})-V_1'(\tilde{T})\int_t^{\tilde{T}}\left(\frac{y'(\tilde{T})}{y'(\theta)}\right)^{n-2}{\rm d}\theta-\int_t^{\tilde{T}}\rho_1V_1 e^{-s}\left(\int_t^s\frac{{\rm d}\theta}{(y'(\theta))^{n-2}}\right){\rm d}s\\
&=& V_1(\tilde{T})-V_1'(\tilde{T})\int_t^{\tilde{T}}\left(\frac{1+\bo\left(\frac{g''}{(g')^2}\right)}{1+\bo\left(\frac{g''}{(g')^2}\right)}\right){\rm d}\theta +(\displaystyle\max_{s\in[t,\tilde{T}]}\vert V_1(s)\vert)\bo\left((g')^{n-1-\alpha}\int_{T(\gamma)}^{\tilde{T}}(s-T(\gamma))^\alpha e^{-s}{\rm d}s\right)\\
&=& V_1(\tilde{T})-V_1'(\tilde{T})(\tilde{T}-t)+\bo\left(\left(\frac{g''}{(g')^2}\right)^2\right)(\tilde{T}-T(\gamma))\\
&&+(\displaystyle\max_{s\in[t,\tilde{T}]}\vert V_1(s)\vert)\bo\left((g')^{n-1-\alpha}\int_{T(\gamma)}^{\tilde{T}}(1+s-T(\gamma)) e^{-s}{\rm d}s\right)\\
&=& V_1(\tilde{T})-V_1'(\tilde{T})(\tilde{T}-t)+\bo\left(\frac{(g'')^2}{(g')^3}+\displaystyle\max_{s\in[t,\tilde{T}]}\vert V_1(s)\vert (g')^{n-\alpha-1}e^{-T(\gamma)}\right).
\end{array}\end{equation}
Therefore, we have
\begin{eqnarray*}
 V_1(t)=V_1(\tilde{T})-V_1'(\tilde{T})(\tilde{T}-t)+\bo\left(\frac{(g'')^2}{(g')^3}+ \frac{e^{-\left(g-\left(\frac{n-1}{n}\right)\gamma g'\right)}}{(g')^\alpha}\displaystyle\max_{s\in[t,\tilde{T}]}\vert V_1(s)\vert\right).
\end{eqnarray*}
Since from Claim \ref{claim5}, $V_1(\tilde{T})=\bo(1)$, hence 
\begin{equation}
\label{myeq6}
V_1(t)=\bo(1) \text{   for all   } t\in [T(\gamma),\tilde{T}] . 
\end{equation}
 Suppose $V'_1(t_1)=0$ and $V_1(t_1)>0$, then $V_1(t_0)=0$ and 
\begin{eqnarray*}
0&=&V_1(t_0)=V_1(t_1)+V'_1(t_1)(t_0-t_1)-\int_{t_1}^{t_0}\rho_1 V_1e^{-s}\left(\int_{t_1}^s\frac{{\rm d}\theta}{(y'(\theta))^{n-2}}\right){\rm d}s\\
&=&V_1(t_1)+\bo\left(\frac{e^{-(g-\frac{(n-1)\gamma g'}{n})}}{(g')^\alpha}\right)V_1(t_1)
\end{eqnarray*}
which is a contradiction if $V_1(t_1)>0$ and this proves \eqref{eq4.76}.

From \eqref{myeq7} and \eqref{myeq6}, the estimate
\begin{eqnarray*}
V_1(T(\gamma))=V_1(\tilde{T})+V'_1(\tilde{T})(T(\gamma)-\tilde{T})+\bo\left(\frac{g''}{(g')^3}+\frac{e^{-(g-\frac{(n-1)\gamma g'}{n})}}{(g')^\alpha}\right)
\end{eqnarray*}
follows and this proves the lemma.\qed

\begin{lemm}\label{lemma4.13}
Assume $S$ satisfies \eqref{eq4.62} and $\delta=\bo(\log(g'))$. Then
\begin{gather}
\label{eq4.77}V_1(T(\gamma))=V_1(S)\left(1-\frac{(n-1)\gamma g''}{(g')^2}+\bo\left(\frac{\delta^4 g''}{(g')^2}+\frac{\delta^3}{(g')^2}\right)\right),\\
\label{eq4.78}T'(\gamma)=\frac{\left(1+\bo\left(\frac{\delta^3 g''}{(g')^2}\right)\right)}{n}\left(g'-(n-1)\gamma g''+\bo\left(\frac{\delta^4 g''}{g'}+\frac{\delta^3}{g'}\right)\right).
\end{gather}
\end{lemm}

\proof Denote $\eta=\eta(\gamma)\eqdef \frac{e^{-(g-\frac{(n-1)\gamma g'}{n})}}{(g')^{\alpha}}$.  Then from Lemmas \ref{lemma4.8} to \ref{lemma4.12}, we have
\begin{eqnarray*}
V_1(T(\gamma))&=&V_1(\tilde{T})+V_1'(\tilde{T})(T(\gamma)-\tilde{T})+\bo\left(\frac{g''}{(g')^2}+\eta\right)\\
&=& V_1(S_6)+V'_1(S_6)(\tilde{T}-S_6)+(V'_1(S_6)+\bo\left(\frac{(g'')^2}{(g')^4}\right))(T(\gamma)-\tilde{T})
+\bo\left(\frac{g''}{(g')^2}+\eta\right)\\
&=&V_1(S_6)+V'_1(S_6)(T(\gamma)-S_6)+\left(\frac{g''}{(g')^2}+\eta\right)\\
&=& V_1(S)\left(1+(n-1)^2V'_2(S)+(n-1)V'_2(S)(S_6-S)\right)\\
&&+(n-1)V_1(S)V'_2(S)(T(\gamma)-S_6)
+\bo\left(\frac{\delta^3g''}{(g')^4}(T(\gamma)-S_6)+\eta\right)\\
&=& V_1(S)\left(1+(n-1)^2V'_2(S)+(n-1)V'_2(S)(T(\gamma)-S)\right)+\bo\left(\frac{\delta^3}{(g')^2}+\eta\right).
\end{eqnarray*}
Now, \eqref{eq2.12} and \eqref{eq4.63}
\begin{eqnarray*}
T(\gamma)-S&=& g-\frac{(n-1)\gamma g'}{n}+(n-1)\log(\frac{(n-1)g'}{n})+\frac{(n-1)\alpha_n\gamma g''}{g'}+A(\gamma)\\
&&-g-(n-1)\log(\frac{(n-1)g'}{n})-(n-1)\log(\frac{(n-1)g''}{(g')^2})+\bo\left(\frac{\delta g''}{(g')^2}+\frac{\delta^2}{g'}\right)\\
&=&-\frac{(n-1)\gamma g'}{n}-(n-1)\log\left(\frac{(n-1)g''}{(g')^2}\right)+\bo(1)
\end{eqnarray*}
and from \eqref{eq4.64}
\begin{eqnarray*}
(n-1)V'_2(S)(T(\gamma)-S)&=&\frac{ng''}{(g')^2}\left(1+\bo\left(\frac{\delta^4 g''}{(g')^2}\right)\right)(-\frac{(n-1)\gamma g'}{n}-(n-1)\log(\frac{(n-1)g''}{(g')^2})
+\bo(1))\\
&=&-\frac{(n-1)\gamma g''}{g'}-\frac{n(n-1)g''}{(g')^2}\log\left(\frac{g''}{(g')^2}\right)+\bo\left(\frac{\delta^4\gamma(g'')^2}{(g')^3}+\frac{\delta^5 (g'')^2}{(g')^4}\right)\\
&=&-\frac{(n-1)\gamma g''}{g'}+\bo\left(\frac{\delta^4 g''}{(g')^2}\right).
\end{eqnarray*}
Therefore,
\begin{eqnarray*}
V_1(T(\gamma))=V_1(S)\left(1+\frac{n(n-1)g''}{(g')^2}-\frac{(n-1)\gamma g''}{g'}\right)+\bo\left(\frac{\delta^4 g''}{(g')^2}+\frac{g''\delta^3}{(g')^2}+\eta\right).
\end{eqnarray*}
This proves \eqref{eq4.77}. From (H2) of Hypothesis \ref{hypothesis-f}, \eqref{eq2.13}, \eqref{eq3.52} and \eqref{eq4.76} we have
\begin{eqnarray*}
T'(\gamma)&=&-\frac{V_1(T(\gamma))}{y'(T(\gamma))}=-V_1(S)(1+\bo\left(\frac{1}{(g')^2}\right))\frac{(n-1)g'}{n}\left(1-\frac{(n-1)\gamma g''}{g'}+\bo\left(\frac{\delta^4 g''}{(g')^2}\right)\right)\\
&=&\frac{(1+\bo\left(\frac{\delta^3 g''}{(g')^2}\right))}{n}\left(g'-(n-1)\gamma g''+\bo\left(\frac{\delta^4 g''}{g'}\right)\right).
\end{eqnarray*}
This proves the lemma.\qed

 \begin{proof}[Proof of Theorem \ref{theo2.6}] Let $S_6(\gamma)$ be as  in \eqref{eq4.61}. Choose a $\gamma_0>0$ large such that for $\gamma\geq \gamma_0$, all the estimates  up to Lemma \ref{lemma4.7} hold. Let $S=S(\gamma)$ be the first turning point as defined earlier. Let
\begin{eqnarray}\label{eq4.79}
I_1\eqdef\{\gamma\geq \gamma_0 ;\; S(\gamma)<S_6(\gamma)\},
\end{eqnarray}
\begin{eqnarray}\label{eq4.80}
I_2\eqdef\{\gamma\geq \gamma_0; \; S(\gamma)\geq S_6(\gamma)\}.
\end{eqnarray}
Since $\gamma\to S(\gamma)$ is continuous, we have that  $I_2$ is a closed set. Furthermore from Lemma \ref{lemma4.8}, there exists a $\gamma_1\geq \gamma_0$ such that  if $\gamma\geq \gamma_1$, $\gamma\in I_2$, then
\begin{eqnarray}\label{eq4.81}
S(\gamma)=T_1+(n-1)\log\left(\frac{g''}{(g')^2}\right)+\bo(1).
\end{eqnarray}
Hence if $I_3\eqdef\{\gamma>\gamma_1 ;\; S(\gamma)\geq S_6(\gamma)\}$ is non empty, then $I_3$ is both open and closed in $(\gamma_1,\infty)$ which implies that $I_3=(\gamma_1,\infty)$. Therefore we have either $I_3=(\gamma_1,\infty)$ or  $I_3=\emptyset$. 

If $I_3=\emptyset$, then $$I_4\eqdef\{\gamma\geq\gamma_1 ;\; S(\gamma)<S_6(\gamma)\}=[\gamma_1, \infty).$$
\begin{clm}\label{claim7} $I_4=\emptyset$.\end{clm}

\begin{proof}[Proof of Claim \ref{claim7}] Suppose $I_4\neq \emptyset$, then $I_4=[\gamma_1,\infty)$. Let $\gamma\geq\gamma_1$ and $t_0$ be as in \eqref{eq4.67}.
Then, for $t_0\leq t\leq S$, we have $\vert V_1(t)\vert\leq \vert V_1(S)\vert$. Also $X(S_6)=e^{\frac{T_1-S_6}{n-1}}=(g')^{\frac{4q}{q-1}+1}$ and hence for $t\in [t_0, S_6(\gamma)]$,
\begin{gather*}
z'(t)=\bo\left(\frac{1}{g'}\right),\quad z''(t)=\bo\left(\frac{1}{X(S_6)g'}\right)=\bo\left(\frac{1}{(g')^{\frac{2q}{q-1}+2}}\right)\\
z'(t)^{n-3}z''(t)= \bo\left(\frac{1}{(g')^{\frac{2q}{q-1}+n-1}}
\right), \quad\vert V'_2(t)\vert+\vert V''_2(t)\vert=\bo\left(\frac{1}{X(S_6)}\right)=\bo\left(\frac{1}{(g')^{\frac{2q}{q-1}+1}}\right).
\end{gather*}
Therefore from \eqref{eq3.19}, \eqref{eq3.23} and \eqref{eq3.43}, we have the following estimates for $t\in [\max\{t_0,\tilde{T}\}, S_6(\gamma)]$:
\begin{eqnarray*}
\frac{1}{y'(t)^{n-2}}\int_t^S\rho_1V_1V_2 e^{-s}{\rm d}s&=&\bo\left(\vert V_1(S)\vert(g')^{n-1}\int_t^Se^{g+g'(z-\gamma)-s}{\rm d}s\right)\\
&=&\bo\left(\vert V_1(S)\vert(g')^{n-1}[(z'(S))^{n-1}-(z'(t))^{n-1}]\right)\\
&=&\bo\left(\vert V_1(S)\vert\frac{(g')^{n-1}}{X(S_6)(g')^{n-1}}\right)=\bo\left(\frac{\vert V_1(S)\vert}{(g')^{\frac{4q}{q-1}+1}}\right),
\end{eqnarray*}
\begin{eqnarray*}
\frac{1}{y'(t)^{n-2}}\int_t^S\rho_2V_1V_2 e^{-s}{\rm d}s&=&\bo\left(\frac{\vert V_1(S)\vert(g')^{n-2}}{(g')^{n-1}X(S_6)}\right)=\bo\left(\frac{\vert V_1(S)\vert}{(g')^{\frac{2q}{q-1}+2}}\right).
\end{eqnarray*}
From \eqref{eq3.39}, \eqref{eq3.40} and\eqref{eq3.41}, we have
\begin{eqnarray*}
-(n-1)(y')^{n-2}y''=-((y')^{n-1})'=e^{g(y(t))-t}\leq e^{\psi(y(t))}=\bo\left(\frac{1}{(g')^{\frac{2q}{q-1}+n}}\right)
\end{eqnarray*}
and hence we get
\begin{eqnarray*}
\rho_3(t)=((y')^{n-2}-(z'(t))^{n-2})V_2''+(n-2)\left(\frac{(y')^{n-2}y''}{y'}-\frac{(z')^{n-2}z''}{z'}\right)V_2'=\bo\left(\frac{1}{(g')^{\frac{2q}{q-1}+n-1}}\right).
\end{eqnarray*}
From the above estimates, \eqref{eq4.19} and the fact that $S_6-\tilde{T}=\bo(g)$, we get
\begin{eqnarray*}
V'_1(t)V_2(t)=V_1(t)V'_2(t)-\left(\frac{y'(S)}{y'(t)}\right)^{n-2}V'_2(S)V_1(S)+\bo\left(\frac{\vert V_1(S)\vert g}{(g')^{\frac{2q}{q-1}+1}}\right)=\bo\left(\frac{\vert V_1(S)\vert g}{(g')^{\frac{2q}{q-1}+1}}\right).
\end{eqnarray*}
Since $V_2(t)\leq -C$ for some $C>0$ independent of $\gamma$, we thus see that 
\begin{eqnarray}\label{eq4.82}
V'_1(t)=\bo\left(\frac{\vert V_1(S)\vert g}{(g')^{\frac{2q}{q-1}+1}}\right)=\bo\left(\frac{\vert V_1(S)\vert }{(g')^{\frac{q}{q-1}+1}}\right).
\end{eqnarray}
Now from \eqref{eq4.52}, we have
\begin{eqnarray*}
V_1(t)&=&V_1(S)+\bo\left(\frac{\vert V_1(S)\vert g(\tilde{T}-S)}{(g')^{\frac{2q}{q-1}+1}}\right)
\leq V_1(S_6)+\bo\left(\frac{\vert V_1(S)\vert g^2}{(g')^{\frac{2q}{q-1}+1}}\right)\\
&=&V_2(S_6)+\bo\left(\frac{(log(g'))^3 g''}{(g')^2}+\frac{\vert V_1(S)\vert}{g'}\right)
\leq -\frac{1}{n-1}+\bo\left(\frac{g'' \de^4}{(g')^2}\right)<0.
\end{eqnarray*}
Hence $t_0<\tilde{T}$ and
$V_1(t)=V_1(S)(1+\bo\left(\frac{1}{(g')^{\frac{q}{q-1}+1}}\right)).$
Now for $t\in [t_0, \tilde{T}]$ and from \eqref{eq4.14} and \eqref{eq4.82}, we have
\begin{eqnarray*}
V'_1(t)&=&\left(\frac{y'(\tilde{T})}{y'(t)}\right)^{n-2}V'_1(\tilde{T})+\frac{1}{y'(t)^{n-2}}\int_t^{\tilde{T}}f'(y)V_1e^{-s}{\rm d}s\\
&=&\bo\left(\frac{\vert V_1(S)\vert}{(g')^{\frac{q}{q-1}+1}}+\vert V_1(S)\vert(g')^{n-2}\int_t^{\tilde{T}} f'(y(s))e^{-s} {\rm d}s\right).
\end{eqnarray*}
As in \eqref{myeq7}, we therefore have
\begin{eqnarray*}
V_1(t)&=& V_1(\tilde{T})+\vert V_1(S)\vert \bo\left(\frac{\tilde{T}-T(\gamma)}{(g')^{\frac{q}{q-1}+1}}+\vert V_1(S)\vert e^{-T(\gamma)}(g')^{n-2}\right)\\
&=&V_1(S)\left(1+\bo\left(\frac{1}{(g')^{\frac{q}{q-1}}}+\frac{e^{-(g-\frac{(n-1)\gamma g'}{n})}}{(g')^{\alpha}}\right)\right).
\end{eqnarray*}
Since from (H2) of Hypothesis \ref{hypothesis-f} which gives $g-\frac{(n-1)\gamma g'}{n} \geq b$ for some $b \in \R$ as $\gamma\to\infty$, we hence get
\begin{eqnarray*}
V_1(t)=V_1(S)(1+o(1))<0.
\end{eqnarray*}
Therefore $t_0=T(\gamma)$ and 
\begin{eqnarray*}
T'(\gamma)&=&-\frac{V_1(T(\gamma))}{y'(T(\gamma))}=-\frac{(n-1)g'}{n}V_1(S)(1+o(1))\geq \frac{(n-1)g'}{n}\vert V_1(S_6)\vert(1+o(1))\\
&=&\frac{g'}{n}(1+o(1)).
\end{eqnarray*}
Integrating the above inequality from $\gamma_1$ to $\gamma$ and using \eqref{eq2.12} in Theorem \ref{theo2.5}, we obtain
\begin{eqnarray*}
g-\left(\frac{n-1}{n}\right)\gamma g'+\bo(\log(g'))=T(\gamma)\geq T(\gamma_1)+\bo(g(\gamma_1))+\frac{g}{n}(1+o(1)).
\end{eqnarray*}
Therefore we get
\begin{eqnarray*}
n-\frac{(n-1)\gamma g'}{g}\geq 1+o(1).
\end{eqnarray*}
Since from (H1) of Hypothesis \ref{hypothesis-f}, we see that $$(n-1)\left[1-\frac{\gamma g'}{g}\right]=(n-1)\left[1-q+\bo\left(\frac{\rho '\gamma}{g}+\frac{\rho}{g}\right)\right]\geq o(1)$$ which implies that $q \leq 1$ as $\gamma\to\infty$ which is a contradiction. This proves the claim. 
\end{proof}

Therefore $I_3=[\gamma_1,\infty)$ and for $\gamma\in I_3$ and from \eqref{eq4.78} 
the theorem follows.
\end{proof}
{\bf Acknowledgement:} {\sc Adimurthi} and {\sc J. Giacomoni} were funded by IFCAM (Indo-French Centre for Applied Mathematics) under the project ``Singular phenomena in reaction diffusion equations and in conservation laws" .

\end{document}